\documentclass[a4paper,12pt]{amsart}

\usepackage[utf8]{inputenc} % allow utf-8 input
\usepackage[T1]{fontenc}    % use 8-bit T1 fonts
\usepackage{url}            % simple URL typesetting
\usepackage{booktabs}       % professional-quality tables
\usepackage{amsfonts}       % blackboard math symbols
\usepackage{nicefrac}       % compact symbols for 1/2, etc.
\usepackage{microtype}      % microtypography
\usepackage{xcolor}         % colors

\usepackage{fullpage}
\usepackage{amsfonts,amsmath,amsthm,amssymb}
\usepackage{bm}
\usepackage[colorlinks,linkcolor=black,citecolor=black,urlcolor=black]{hyperref}

\usepackage{natbib}
\usepackage{derivative}
\usepackage{graphicx}
\usepackage{float}
\usepackage{enumitem} % flexible enumerate
\usepackage{booktabs} % for professional tables

\usepackage{caption}
\usepackage{subcaption}
\usepackage{mathtools}

\newcommand{\limd}{{\smash{\xrightarrow{~\mathrm{d}~}}}}
\newcommand{\limp}{{\smash{\xrightarrow{~\mathrm{p}~}}}}
\newcommand{\R}{\mathbb{R}}
\newcommand{\cov}{{\mathrm{cov}}}
\newcommand{\prox}{{\mathrm{prox}}}
\newcommand{\calI}{\mathcal{I}}

\newcommand{\z}{\mathsf{z}}

\newcommand{\g}{\mathsf{g}}

\newcommand{\y}{\mathsf{y}}
\newcommand{\p}{\mathsf{p}}
\renewcommand{\u}{\mathsf{u}}
\newcommand{\B}{\mathsf{B}}
\newcommand{\D}{\mathsf{D}}
\newcommand{\V}{\mathsf{V}}

\renewcommand{\H}{\mathsf{H}}
\renewcommand{\L}{\mathsf{L}}
\newcommand{\sfP}{\mathsf{P}}

\def\defas{\stackrel{\text{\tiny def}}{\coloneqq}}

\usepackage{cleveref} % clever reference, try \Cref
\usepackage{xspace} % avoid treating the periods as end of sentence periods. 

\usepackage{constants}

\usepackage{thmtools}
\usepackage{thm-restate}

\def\iid{i.i.d.\@\xspace}
\def\eg{e.g.\@\xspace}
\def\ie{i.e.\@\xspace}

\AtBeginDocument{%
   \def\MR#1{}
   \def\url#1{}
}

\DeclareMathOperator{\Var}{Var}

\DeclareMathOperator{\vect}{vec}
\DeclareMathOperator{\trace}{Tr}
\DeclareMathOperator{\Jac}{Jac}
\DeclareMathOperator{\rank}{rank}
\let\div\relax
\DeclareMathOperator{\div}{div}
\DeclareMathOperator*{\argmin}{arg\,min}

\def\E{\mathbb E}
\def\P{\mathbb P}
\def\N{\mathbb N}
\DeclareMathOperator{\diag}{diag}
\DeclareMathOperator{\rem}{Rem}

\DeclareBoldMathCommand\ba{a}

\usepackage[textsize=small]{todonotes}
\theoremstyle{plain} %for theorem, lemma, corollary, proposition. 
\newtheorem{theorem}{Theorem}[section]

\newtheorem{lemma}[theorem]{Lemma}
\newtheorem{corollary}[theorem]{Corollary}

\theoremstyle{definition} %for definitions, conditions, problems, examples. 
 
\newtheorem{assumption}{Assumption}[section]
\theoremstyle{remark} %for remark, notes, claims. 

\numberwithin{equation}{section}
\numberwithin{theorem}{section}

\Crefname{assumption}{Assumption}{Assumptions}

\graphicspath{{figures/}} %set the graphicspath
\title[High-Dimensional Multinomial Logistic Regression]{
    Multinomial Logistic Regression:
    Asymptotic Normality on Null Covariates in High-Dimensions
}

\author{Kai Tan}
\address[Kai Tan]{
Department of Statistics, Rutgers University, Piscataway, NJ 08854, USA.
}
\email{kai.tan@rutgers.edu}

\author{Pierre C. Bellec}
\address[Pierre C. Bellec]{
Department of Statistics, Rutgers University, Piscataway, NJ 08854, USA.
}
\email{pierre.bellec@rutgers.edu}

\usepackage[many]{tcolorbox}
\usepackage{tikz}
\usetikzlibrary{arrows.meta, % for arrows style
                positioning  % for positioning of boxes
               }

\begin{document}

\begin{abstract}
This paper investigates the asymptotic distribution of the maximum-likelihood estimate (MLE) in multinomial logistic models in the high-dimensional regime where dimension and sample size are of the same order. While classical large-sample theory provides asymptotic normality of the MLE under certain conditions, such classical results are expected to fail in high-dimensions as documented for the binary logistic case in the seminal work of \cite{sur2019modern}. We address this issue in classification problems with 3 or more classes, by developing asymptotic normality and asymptotic chi-square results for the multinomial logistic MLE (also known as cross-entropy minimizer) on null covariates. Our theory leads to a new methodology to test the significance of a given feature. Extensive simulation studies on synthetic data corroborate these asymptotic results and confirm the validity of proposed p-values for testing the significance of a given feature.
\end{abstract}
\maketitle
% \tableofcontents
\section{Introduction} \label{sec:intro}
Multinomial logistic modeling has become a cornerstone of classification problems
in machine learning, as witnessed by the omnipresence of both the 
cross-entropy loss (multinomial logistic loss)
and the softmax function (gradient of the multinomial logistic loss)
in both applied and theoretical machine learning.
We refer to \citet{cramer2002origins} for
an account of the history and early developments of logistic modeling.
% This omnipresence of logistic modeling in statistics and machine learning
% undoubtedly makes us guilty of many omissions.

Throughout, we consider a classification problem with
$K+1$ possible labels where $K$ is a fixed constant. This paper tackles 
asymptotic distributions of multinomial logistic estimates
(or cross-entropy minimizers) in generalized
linear models with moderately high-dimensions,
where sample size $n$ and dimension $p$ have the same order, for instance
$n,p\to+\infty$ simultaneously while the ratio $p/n$ converges to a finite
constant. Throughout the paper, let $[n] =\{1, 2, \ldots, n\}$ for all $n\in \N$, and $I\{\text{statement}\}$ be the 0-1
valued indicator function, equal to 1 if statement is true and 0 otherwise
(e.g., $I\{y_i=1\}$ in the next paragraph equals 1 if $y_i=1$ holds and 0 otherwise).

\paragraph{The case of binary logistic regression.}
Let $\rho(t)=\log(1+e^t)$ be the logistic loss
and $\rho'(t)=1/(1+e^{-t})$ be
its derivative, often referred to as the sigmoid
function. In the current moderately-high dimensional regime
where $n,p\to+\infty$ with $p/n\to\kappa>0$ for some constant $\kappa$,
recent works \citep{candes2020phase,sur2019modern,zhao2022asymptotic} provide a detailed theoretical
understanding of the behavior of the logistic Maximum Likelihood Estimate (MLE)
in binary logistic regression models.
Observing independent observations $(x_i, y_i)_{i\in[n]}$  
from a logistic model defined as
$\P(y_i=1|x_i) = \rho'(x_i^T\beta)$
where $x_i \sim N(\bm{0}, n^{-1}I_p)$,
and $\lim_{n\to\infty} \|\beta\|^2/n = \gamma^2$ for a constant $\gamma$
for the limiting squared norm of the unknown regression vector $\beta$.
These works prove that the behavior of the MLE $\hat \beta
=\argmin_{b\in\R^p}\sum_{i=1}^n \rho(x_i^Tb) - I\{y_i=1\} x_i^Tb$
is summarized by the solution $(\alpha_*,\sigma*,\lambda_*)$ of the system
of three equations
\begin{equation}\label{eq:system}
    \begin{cases}
        \sigma^2 & = \frac{1}{\kappa^2} \E [2 \rho'(\gamma Z_1) (\lambda \rho'(\prox_{\lambda\rho}(-\alpha\gamma Z_1 + \sqrt{\kappa} \sigma Z_2)))^2]\\
    0 & = \E [\rho'(\gamma Z_1) \lambda \rho'(\prox_{\lambda\rho}(-\alpha\gamma Z_1 + \sqrt{\kappa} \sigma Z_2))]\\
    1 - \kappa &= \E [{2\rho'(\gamma Z_1)}\big/\bigl({1 + \lambda \rho''(\prox_{\lambda\rho}(-\alpha\gamma Z_1 + \sqrt{\kappa} \sigma Z_2))}\bigr)]
    \end{cases},
\end{equation}
where $(Z_1, Z_2)$ are \iid $N(0,1)$ random variables
and the proximal operator is defined as 
$
\prox_{\lambda\rho}(z) = \argmin_{t\in \R} \bigl\{\lambda \rho(t) + (t-z)^2/2 \bigr\}
$.
The system \eqref{eq:system} characterize, among others, the following behavior
of the MLE $\hat\beta$:
for almost any $(\gamma,\kappa)$, the system admits a solution if and only if
$\hat \beta$ exists with probability approaching one and in this case,
$\|\hat\beta\|^2/n$ and $\|\hat\beta-\beta\|^2/n$ both have finite limits
that may be expressed as simple functions of $(\alpha_*,\sigma_*,\lambda_*)$,
and for any feature $j\in[p]$ such that $\beta_j=0$ (\ie, $j$ is a null covariate), the $j$-th coordinate of the MLE satisfies
$$
\hat \beta_j \limd N(0, \sigma_*^2). 
$$
The proofs in \cite{sur2019modern} are based on approximate message passing (AMP)
techniques; we refer to \cite{berthier2020state,feng2022unifying,gerbelot2021graph} and the references therein for recent surveys and general results.
More recently, \cite{zhao2022asymptotic} extended the result of \cite{sur2019modern} from isotropic design to Gaussian covariates with an arbitrary covariance structure: if now $x_i\sim N(\bm{0},\Sigma)$ for some positive definite $\Sigma$
and $\lim_{n,p\to+\infty} \beta^T\Sigma\beta = \kappa$, null covariates
$j\in [p]$ (in the sense that $y_i$ is independent of $x_{ij}$ given $(x_{ik})_{k\in[p]\setminus \{j\}}$) of the MLE satisfy
\begin{equation}
(n/\Omega_{jj})^{1/2} \hat \beta_j \limd N(0, \sigma_*^2),
    \label{eq:binary_null_2}
\end{equation}
where $\sigma_*$ is the same solution of \eqref{eq:system} and 
$\Omega=\Sigma^{-1}$. \cite{zhao2022asymptotic} also obtained
asymptotic normality results for non-null covariates, that is,
features $j\in[p]$ such that $\beta_j\ne 0$.
The previous displays can be used to test the null hypothesis
$H_0: y_i$ is independent of $x_{ij}$ given $(x_{ik})_{k\in[p]\setminus \{j\}}$
and develop the corresponding p-values
if $\sigma_*$ is known; in this binary logistic regression model
the ProbeFrontier \citep{sur2019modern} and SLOE \cite{yadlowsky2021sloe}
give means to estimate the solutions $(\alpha_*,\sigma_*,\lambda_*)$
of system \eqref{eq:system} without the knowledge of $\gamma$.
\cite{mai2019large} studied the performance of Ridge regularized
binary logistic regression in mixture models.
\cite{salehi2019impact} extended \cite{sur2019modern} to separable
penalty functions.
\cite{bellec2022observable} derived asymptotic normality results similar
to \eqref{eq:binary_null_2}
in single-index models including binary logistic regression without resorting to the system \eqref{eq:system}, showing
that for a null covariate $j\in[p]$ in the unregularized case that
\begin{equation}
    \label{eq:binary_null_3}
(n/\Omega_{jj})^{1/2} (\hat v/\hat r) \hat \beta_j \limd N(0, {1})
\end{equation}
where $\hat v=\frac 1 n \sum_{i=1}^{n} \rho''(x_i^T\hat\beta)-\rho''(x_i^T\hat\beta)^2x_i^T[\sum_{l=1}^nx_l\rho''(x_l^T\hat\beta)x_l^T]^{-1}x_i$ is scalar
and so is $\hat r^2 = \frac 1 n \sum_{i=1}^n (I\{y_i=1\}-\rho'(x_i^T\hat\beta))^2$. In summary, in this high dimensional binary logistic model,
% the MLE $\hat\beta$,
\begin{enumerate}
    \item The phase transition from \cite{candes2020phase} splits the $(\gamma,\kappa)$ plane into two connected components: in one component the MLE does not exist with high probability, in the other component the MLE exists and $\|\Sigma^{1/2}\hat\beta\|^2$ is bounded with high probability (boundedness is a consequence of the fact that
        $\|\Sigma^{1/2}\hat\beta\|^2$ or $\|\Sigma^{1/2}(\hat\beta-\beta)\|^2$ admit finite limits); 
    \item In the component of the $(\gamma,\kappa)$ plane where the MLE exists, for any null covariate $j\in[p]$, the asymptotic normality results
        \eqref{eq:binary_null_2}-\eqref{eq:binary_null_3} holds.
        % and the asymptotic variate $\sigma_*^2$.
\end{enumerate}
\paragraph*{Multiclass classification.}
The goal of this paper is to develop a theory for the asymptotic normality of the multinomial logistic regression MLE (or cross-entropy minimizer) on null covariates when the number of classes, $K+1$, is greater than 2 and $n,p$ are of the same order.
In other words, we aim to generalize results such as \eqref{eq:binary_null_2} or \eqref{eq:binary_null_3} for three or more classes. 
Classification datasets with 3 or more classes are ubiquitous
in machine learning (MNIST, CIFAR to name a few), which calls for such
multiclass generalizations.
In Gaussian mixtures and logistic models, 
\cite{thrampoulidis2020theoretical} derived characterizations of the 
performance of 
of least-squares and class-averaging estimators, excluding
cross-entropy minimizers or minimizers of non-linear losses.
%{
\cite{loureiro2021learning} extended \cite{sur2019modern,zhao2022asymptotic,salehi2019impact} to multiclass classification problems in a Gaussian mixture model, and obtained the fixed-point equations that characterize
the performance and empirical distribution of the minimizer of
the cross-entropy loss plus a convex regularizer.
    In the same vein as \cite{loureiro2021learning},
    \cite{cornacchia2022learning} studied the limiting fixed-point equations 
    in a multiclass teacher-student learning model where labels are generated by a noiseless channel with response
    $\argmin_{k\in\{1,...,K\}} x_i^T \beta_k$ where $\beta_k\in\R^p$ is unknown for each
    class $k$.
    These two aforementioned works assume a multiclass Gaussian mixture model, which is different than the normality assumption for $x_i$
    used in the present paper.
    More importantly, 
    these results cannot be readily used for the purpose testing significant covariates (cf. \eqref{H0} below) since solving the fixed-point equations
    require the knowledge of several unknown parameters, including the
    limiting spectrum of the mixture covariances and empirical distributions
    of the mixture means (cf. for instance
    Corollary 3 in \cite{loureiro2021learning}).
    In the following sections, we fill this gap with a new methodology
    to test the significance of covariates. This is made possible by
    developing new
    asymptotic normality results for cross-entropy minimizers
    that generalize \eqref{eq:binary_null_3},
    without relying on the low-dimensional fixed-point equations.

\paragraph{Notation.}
Throughout, $I_p\in\R^{p\times p}$ is the identity matrix,
for a matrix $A\in \R^{m\times n}$, $A^T$ denotes the transpose of $A$,
$A^{\dagger}$ denotes the Moore-Penrose inverse of $A$.
If $A$ is psd, $A^{1/2}$ denotes the unique symmetric square root,
i.e., the unique positive semi-definite matrix such that $(A^{1/2})^2=A$.
The symbol $\otimes$ denotes the Kronecker product of matrices.
Given two matrices $A\in\R^{n\times k},B\in\R^{n\times q}$ with the same number or rows,
$(A,B)\in\R^{n\times (k+q)}$ is the matrix obtained by stacking the columns of $A$ and $B$ horizontally. If $v\in\R^n$ is a column vector with dimension equal to the number of rows in $A$, we construct
$(A, v)\in\R^{n\times (k+1)}$ similarly.
We use $\bm{0}_n$ and $\bm{1}_n$ to denote the all-zeros vector and all-ones vector in $\R^n$, respectively; we do not bold vectors and matrices
other than $\bm 0_n$ and $\bm 1_n$. 
We may omit the subscript giving the dimension
if clear from context;
e.g., in $I_{K+1}-\frac{\bm 1\bm 1^T}{K+1}$
the vector $\bm 1$ is in $\R^{K+1}$.
The Kronecker product between two matrices is denoted by $\otimes$
and $\vect(M)\in\R^{nd}$ is the vectorization operator applied to a matrix $M\in\R^{n\times d}$.
For an integer $K\ge 2$ and $\alpha\in(0,1)$, the quantile $\chi^2_K(\alpha)$
is the unique real number satisfying
$\P(W>\chi^2_K(\alpha))=\alpha$
where $W$ has a chi-square distribution with $K$ degrees of freedom.
The symbols $\limd$ and $\limp$ denote convergence in distribution
and in probability.

Throughout, 
\emph{classical asymptotic regime} refers to the scenario where the feature dimension $p$ is fixed and the sample size $n$ goes to infinity. 
In contrast, the term \emph{high-dimensional regime} refers to the situation where $n$ and $p$ both tend to infinity with the ratio $p/n$ converging to a limit smaller than 1.

\subsection{Multinomial logistic regression}
Consider a multinomial logistic regression model with $K+1$ classes. 
We have $n$ \iid data samples $\{(x_i, \y_i)\}_{i=1}^n$, where $x_i \in \R^p$ is the feature vector and $\y_i=(\y_{i1},...,\y_{i(K+1)})^T\in\R^{K+1}$
is the response. Each response $\y_i$ is the one-hot encoding
of a single label, i.e., 
$\y_i\in\{0,1\}^{K+1}$ with $\sum_{k=1}^{K+1} \y_{ik} = 1$
such that $\y_{ik}=1$ if and only if the label for $i$-th observation is $k$. 
A commonly used generative model for $\y_i$ is 
the multinomial regression model, namely
%log-linear conditional probabilities
\begin{equation}\label{model:over-specified}
\P(\y_{ik} = 1 | x_i) 
= 
\frac{\exp(x_i^T \B^* e_k)}{\sum_{k'=1}^{K+1} \exp(x_i^T\B^* e_{k'})}, \quad k\in\{1,2, \ldots, K+1\} 
\end{equation}
where $\B^*\in \R^{p\times (K+1)}$ is an unknown logistic model parameter
and $e_k\in \R^{K+1},e_{k'}\in\R^{K+1}$ are the $k$-th and $k'$-th canonical basis vectors. 
% in $\R^{K+1}$ for $k,k'\in\{1,2,\ldots,K+1\}$.
The %Maximum Likelihood Estimate (MLE)
MLE
for $\B^*$ in the model \eqref{model:over-specified} is any solution that minimizes the cross-entropy loss,
\begin{equation}\label{barB}
    \textstyle
	 \hat \B \in \argmin_{\B\in \R^{p\times (K+1)}} \sum_{i=1}^n \L_i (\B^T x_i),
\end{equation}
where $\L_i: \R^{K+1} \to \R$ is defined as
$\L_i (\u) = - \sum_{k=1}^{K+1} \y_{ik} \u_k + \log \sum_{k'=1}^{K+1} \exp(\u_{k'}).$
If the solution set in \eqref{barB} is non-empty,
we define for each observation $i \in [n]$ the vector
of predicted probabilities $\hat \p_i = (\hat \p_{i1}, ..., \hat \p_{i(K+1)})^T$ with 
\begin{equation}
    \label{p_i} 
\hat \p_{ik} \defas
\P(\hat\y_{ik} = 1) 
= \frac{\exp(x_i^T \hat\B e_k)}{\sum_{k'=1}^{K+1}\exp(x_i^T \hat\B e_{k'})}
\qquad 
\qquad 
\text{ for each }k\in\{1,...,K+1\}.
\end{equation}
Our results will utilize the gradient and Hessian of $\L_i$ evaluated at $\hat\B^T x_i$, denoted by
\begin{equation}\label{eq:g_bar-H_bar}
	\g_i \defas \nabla \L_i(\hat\B^T x_i) = - \y_i + \hat\p_i, \qquad
	\H_i \defas \nabla^2 \L_i(\hat\B^T x_i) = \diag(\hat\p_i) - \hat\p_i \hat\p_i^T.
\end{equation}
The quantities 
$(\hat\B, \hat \p_i, \g_i, \H_i)$ 
can be readily computed from the data $\{(x_i, \y_i)\}_{i=1}^n$. 
To be specific, the MLE $\hat \B$ in \eqref{barB} can be obtained by invoking a multinomial regression solver (\eg, \verb|sklearn.linear_model.LogisticRegression| from \cite{scikit-learn}), and the quantities $\hat \p_i, \g_i, \H_i$ can be further computed from \cref{p_i,eq:g_bar-H_bar} by a few matrix multiplications and application
of the softmax function.

\paragraph{Log-odds model and reference class.}
\label{reference-class}
The matrix $\B^*$ in \eqref{model:over-specified} is not identifiable
since the conditional distribution of $\y_i|x_i$ in the model \eqref{model:over-specified} remains unchanged if we replace columns of $\B^*$ by $\B^* - b\bm{1}_{K+1}^T$ for any $b\in\R^p$. 
In order to obtain an identifiable model, a classical and natural
remedy is to model the log-odds, here with the class $K+1$ as the reference class:
\begin{equation}\label{eq:log-odds}
    \log \frac{\P(\y_{ik}=1|x_i)}{\P(\y_{i(K+1)}=1|x_i)} = x_i^T A^* e_k, \qquad \forall k\in[K]
\end{equation}
where $e_k$ is the $k$-th canonical basis vector of $\R^K$,
and $A^* \in \R^{p\times K}$ is the unknown parameter. 
The matrix $A^*\in \R^{p\times K}$ in log-odds model \eqref{eq:log-odds} is related to $\B^*\in \R^{p\times (K+1)}$ in the model \eqref{model:over-specified} by 
$A^* = \B^* (I_K, -\bm1_{K})^T$. 
This log-odds model has two benefits: 
First 
it is identifiable since the unknown matrix $A^*$ is uniquely defined. 
Second, the matrix $A^*$ lends itself well to interpretation
as its $k$-th column represents the contrast coefficient between class $k$ and the reference class $K+1$. 

The MLE $\hat A$ of $A^*$ in \eqref{eq:log-odds} is
$\hat A = \argmin_{A\in\R^{p\times K}} \sum_{i=1}^n \L_i((A,\bm{0}_p)^T x_i)$.
If the solution set in \eqref{barB} is non-empty, $\hat A$ is related to any solution $\hat\B$ in \eqref{barB} by $\hat A = \hat \B (I_K, -\bm1_{K})^T$. Equivalently, 
\begin{equation}
\hat A_{jk} = \hat \B_{jk} - \hat \B_{j(K+1)}
\label{eq:relationship_A_B}
\end{equation}
for each $j\in[p]$ and $k\in [K]$. 

If there are three classes (\ie $K+1=3$), this parametrization allows us to draw scatter plots of realizations of $\sqrt{n}e_j^T \hat A = (\sqrt{n}\hat A_{j1}, \sqrt{n}\hat A_{j,2})$ as in \Cref{fig:scatter-plot}.

\subsection{Hypothesis testing for the $j$-th feature
and classical asymptotic normality for MLE}

\paragraph*{Hypothesis testing for the $j$-th feature.}
Our goal is to develop a methodology
to test the significance of the
$j$-th feature.
Specifically, for a desired confidence level
$(1-\alpha)\in(0,1)$ (say, $1-\alpha=0.95$) and a given feature
$j\in [p]$ of interest,
our goal is to test
\begin{equation}
    \label{H0}
    H_0: \y_i \text{ is conditionally independent of }
    x_{ij} \text{ given } (x_{ij'})_{j'\in[p]\setminus \{j\}}
    .
\end{equation}
Namely, we want to test whether the $j$-th variable is independent from the response given all other explanatory variables $(x_{ij'}, j'\in[p]\setminus\{j\})$. 
Assuming normally distributed $x_i$ and a multinomial model
as in \eqref{model:over-specified} or \eqref{eq:log-odds}, it is equivalent to test
\begin{equation}\label{eq:H01}
    H_0: e_j^T A^* = \bm 0_{K}^T
    \qquad
    \text{ versus }
    \qquad
    H_1:
    e_j^T A^* 
    \ne \bm 0_{K}^T,
\end{equation}
where $e_j\in \R^p$ is the $j$-th canonical basis vector. 

If the MLE $\hat \B$ in \eqref{barB} exists in the sense
that the solution set in \eqref{barB} is nonempty, the
conjecture that rejecting $H_0$ when
$e_j^T \hat \B$ is far from $\bm 0_{K+1}$ is a reasonable starting point.
The important question, then, is to determine a quantitative statement
for the informal ``far from $\bm 0_{K+1}$'',
similarly to \eqref{eq:binary_null_2} or \eqref{eq:binary_null_3}
in binary logistic regression.

\paragraph{Classical theory with $p$ fixed.}
If $p$ is fixed and $n\to\infty$ in model
\eqref{eq:log-odds},
classical maximum likelihood theory \citep[Chapter 5]{van1998asymptotic} 
provides the asymptotic distribution of the MLE $\hat A$, which can be further used to test \eqref{eq:H01}. 

Briefly, if $x$ has the same distribution as any $x_i$,
% and if $\vect(\cdot)$ denotes the usual vectorization operator,
the MLE $\hat A$ 
in the multinomial logistic model 
is asymptotically normal 
with
$$\sqrt{n} (\vect(\hat A) - \vect(A^*)) \limd N (\bm{0}, \calI^{-1})
\quad
\text{ where }
\quad
\calI = %q
\E[(xx^T) \otimes (\diag(\pi^*) - \pi^* \pi^*{}^T)]$$
is the Fisher information matrix evaluated at the true parameter $A^*$, 
vec$(\cdot)$ is the usual vectorization operator,
and $\pi^*\in\R^K$ has random entries
$$\pi^*_k = \exp(x^T A^* e_k)\big/({1 + \sum_{k'=1}^{K}\exp(x^T A^* e_{k'})})$$ for each $k\in [K]$.
In particular, under $H_0:e_j^T A^* = \bm{0}_K^T$, 
\begin{equation}\label{eq:classical-A}
    \sqrt{n} \hat A{}^T e_j \limd 
    N(\bm{0}, S_j%/q
    )
\end{equation}
where $S_j = (e_j^T \otimes I_K) \calI^{-1} (e_j \otimes I_K) = e_j^T (\cov(x))^{-1} e_j [\mathbb{E}~ (\text{diag}(\pi^*) - \pi^* \pi^*{}^T)]^{-1}$.
When \eqref{eq:classical-A} holds, by the delta method we also have $\sqrt n
%q
S_j^{-1/2}\hat A^T e_j \limd N(\bm{0}, I_K)$ and 
\begin{equation}\label{eq:chi2-classical}
    %q
    n \|S_j^{-1/2} \hat A^T e_j\|^2 \limd \chi^2_K.
\end{equation}
where the limiting distribution is chi-square with $K$ degrees of freedom. 
This further suggests 
the size $\alpha$ test that
rejects $H_0$ when 
$T_n^j(X, Y)> \chi^2_K(\alpha)$, where 
$T_n^j(X, Y)= %q
n \|S_j^{-1/2} \hat A^T e_j\|^2$ is the test statistic.
If \eqref{eq:chi2-classical} holds, this test is guaranteed to 
have a type I error converging to $\alpha$. 
The p-value of this test is given by
\begin{equation}
    \label{eq:classical-p-value}
    \textstyle
    \int_{T_n^j(X, Y)}^{+\infty} f_{\chi^2_K}(t)dt,
\end{equation}
where $f_{\chi^2_K}(\cdot)$ is the density of the chi-square distribution with $K$ degrees of freedom. 

As discussed in the introduction, \cite{sur2019modern}
showed that in binary logistic regression,
classical normality results for the MLE such as \eqref{eq:classical-A} fail in the high-dimensional regime because the variance in \eqref{eq:classical-A} 
underestimates the variability of the MLE even for null covariates;
see also the discussion surrounding \eqref{eq:binary_null_2}.
Our goal is to develop, for classification problems with $K+1\ge 3$ classes,
a theory that correctly characterize the asymptotic distribution
of  $\hat A^Te_j$ for a null covariate $j\in[p]$ in the high-dimensional regime.

\begin{figure}[b]
    \centering
    \begin{subfigure}[b]{0.32\textwidth}
        \centering
        \includegraphics[width=\textwidth]{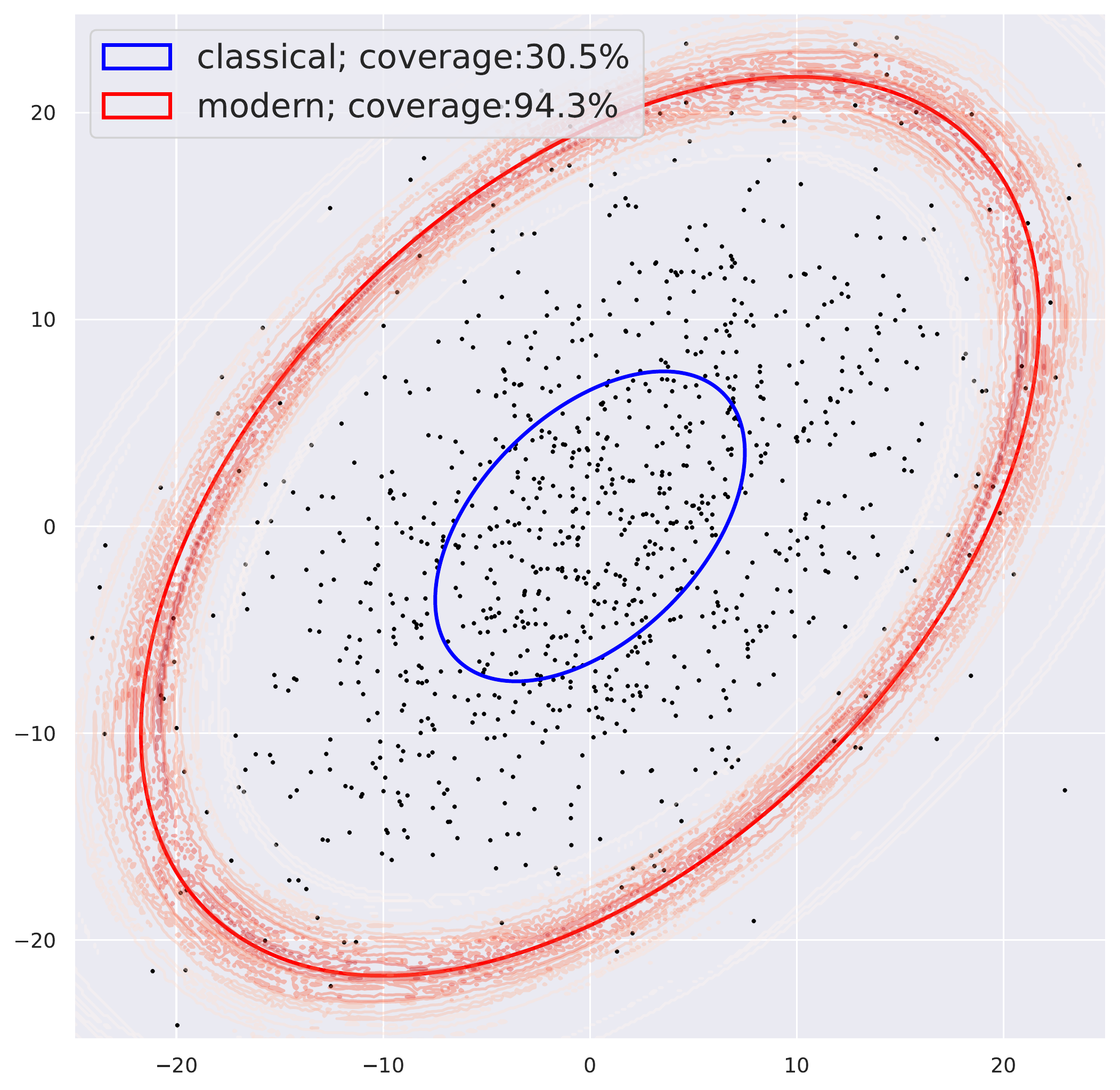}
        \caption{\footnotesize $(n,p)=(2000,600)$}
        % \label{fig:sub-A}
    \end{subfigure} 
    \hfill
    \begin{subfigure}[b]{0.32\textwidth}
        \centering
        \includegraphics[width=\textwidth]{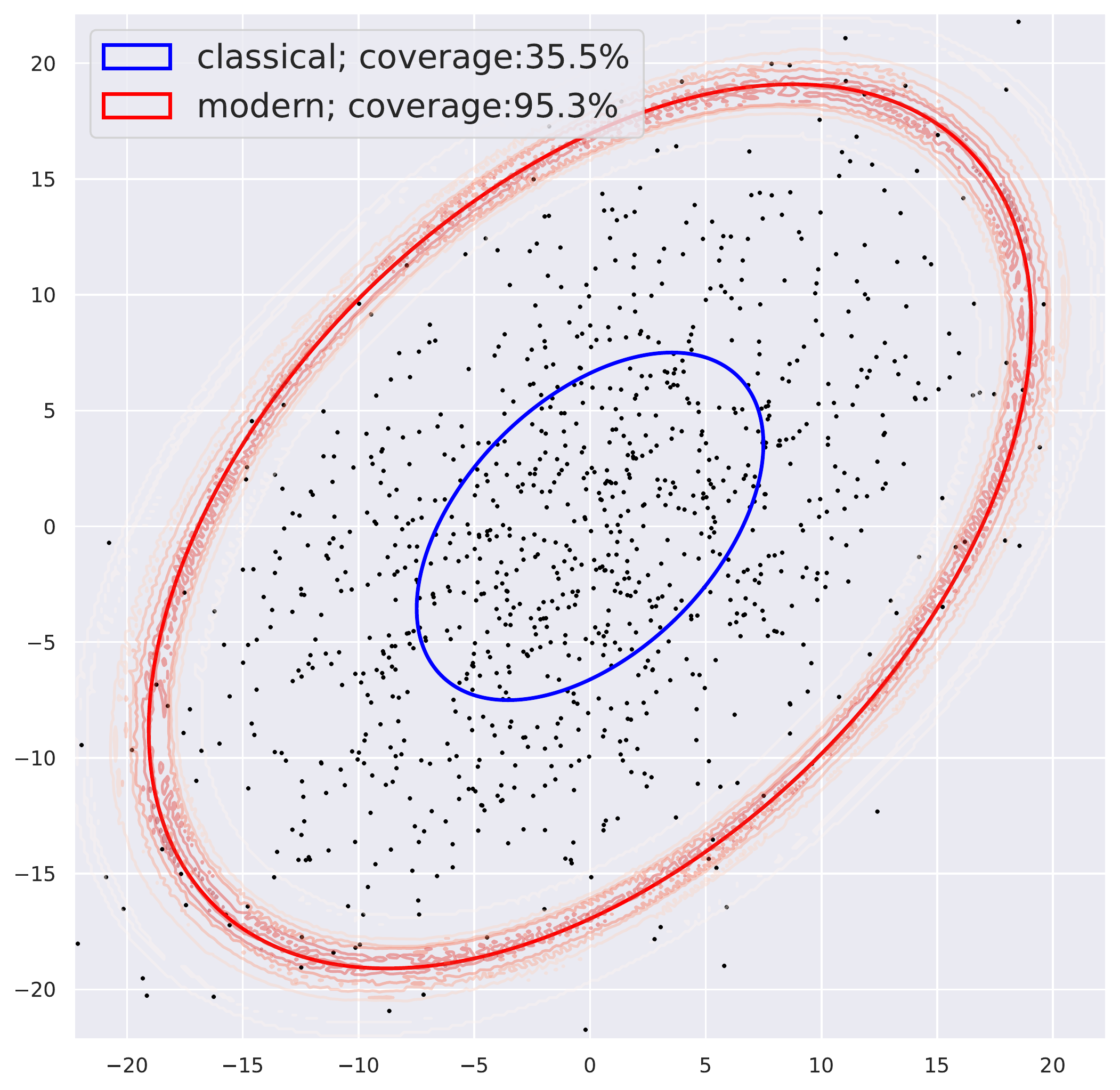}
        \caption{\footnotesize $(n,p)=(3500,1000)$}
        % \label{fig:sub-B}
    \end{subfigure}
    \hfill 
    \begin{subfigure}[b]{0.32\textwidth}
        \centering
        \includegraphics[width=\textwidth]{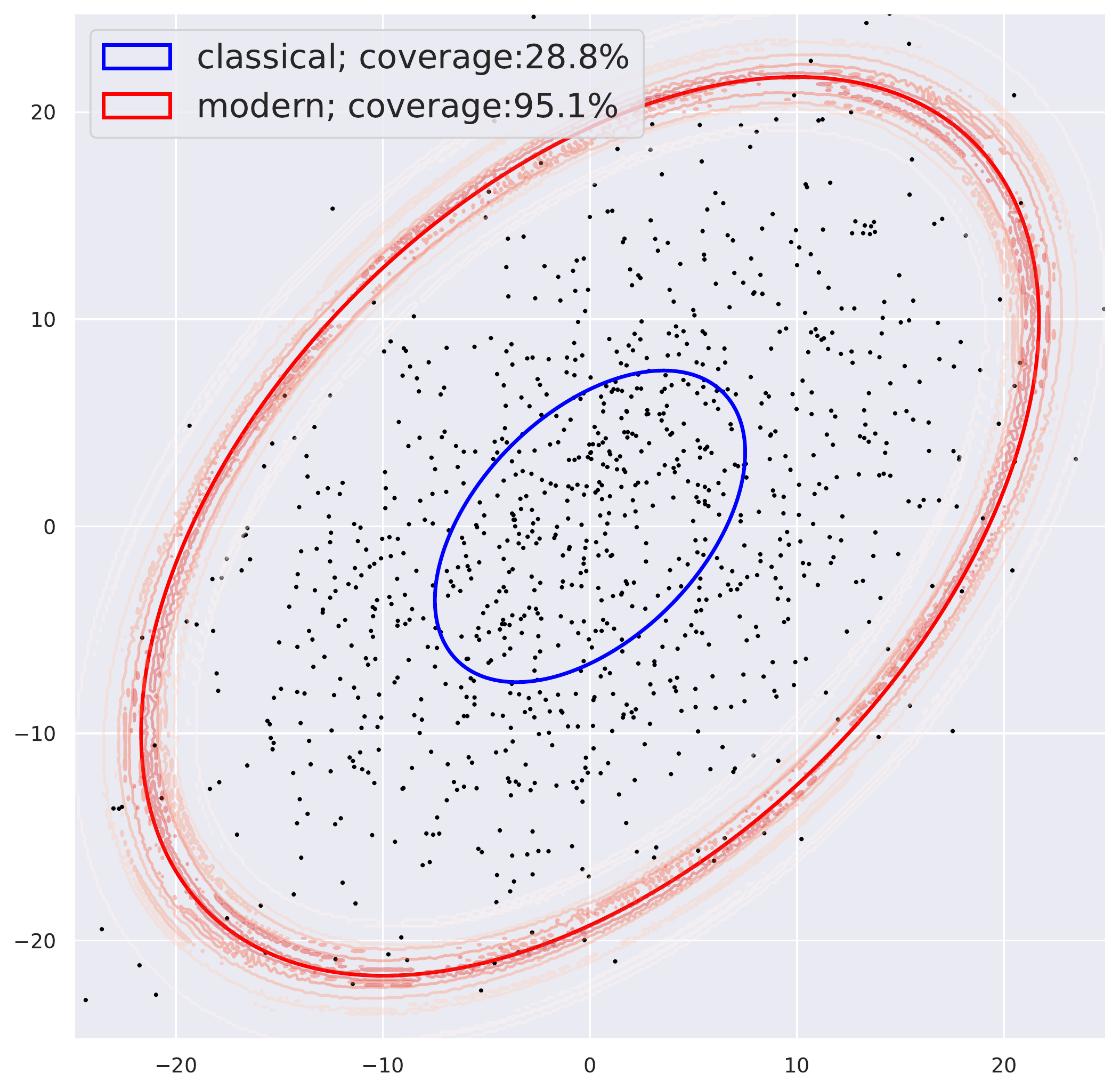}
        \caption{\footnotesize $(n,p)=(5000,1500)$}
        % \label{fig:sub-B}
    \end{subfigure}
    \caption{Scatter plot of pairs $(\sqrt{n}\hat A_{j1}, \sqrt{n}\hat A_{j2})$ with $K=2$ over 1000 repetitions. 
        The blue ellipsoid is the boundary of the $95\%$ confidence set for $\sqrt{n} \hat A^T e_j$ under $H_0$ from the classical MLE theory \eqref{eq:classical-A}-\eqref{eq:chi2-classical} based on the Fisher information, the dashed red ellipsoids are the boundaries of the $95\%$ confidence set for $\sqrt{n} \hat A^T e_j$ under $H_0$ from this paper (cf. \eqref{eq:chi2-A} below). Each of the 1000 repetition gives a slightly different dashed ellipsoid. The solid red ellipsoid is the average of these 1000 dashed ellipsoids.
    Each row of $X$ is \iid sampled from $N(\bm{0}, \Sigma)$ with $\Sigma=(0.5^{|i-j|})_{p\times p}$. 
    The first $\lceil p/4\rceil$ rows of $A^*$ are \iid sampled from $N(\bm{0},I_K)$ while other rows are set to zeros. 
    We further normalize $A^*$ such that 
    $A^*{}^T\Sigma A^* = I_K$.
    The last coordinate $j=p$ is used as the null coordinate.
    }
    \label{fig:scatter-plot}
\end{figure}
We present first some motivating simulations that demonstrate the failure of classical normal approximation \eqref{eq:classical-A} in finite samples. 
These simulations are conducted for various configurations of $(n, p)$ with $K+1=3$ classes.
We fix the true parameter $A^*$ and obtain 1000 realizations of $(\hat A_{j1}, \hat A_{j2})$ by independently resampling the data $\{(x_i, \y_i)\}_{i=1}^n$ 1000 times.
If the result \eqref{eq:classical-A} holds, then 
$\P(\sqrt{n} \hat A^T e_j \in \mathcal C_{\alpha}^j) \to 1-\alpha$, where 
$\mathcal C_{\alpha}^j = \{u\in \R^K: \|S_j^{-1/2} u\| \le \chi^2_K(\alpha)\}$.
\Cref{fig:scatter-plot} displays scatter plots of $\sqrt{n}(\hat A_{j1}, \hat A_{j2})$ along with the boundary of 95\% confidence set $\mathcal C_{\alpha}^j$ with $\alpha = 0.05$. 
We observe that, across the three different configurations of $(n,p)$, the 95\% confidence sets from our theory (\Cref{thm:normal-A} presented in next section) cover around 95\% of the realizations, while the set $\mathcal C_{\alpha}^j$ from classical theory only covers approximately $30\%$ of the points, which is significantly lower than the desired coverage rate of $95\%$. 
Intuitively and
by analogy with results in binary classification \citep{sur2019modern},
this is because the classical theory \eqref{eq:classical-A} underestimates the variation of the MLE in the high-dimensional regime.
Motivated by this failure of classical MLE theory and the results in binary classification \cite[among others]{sur2019modern},
the goal of this paper is to
develop a theory for multinomial logistic regression that achieves the following objectives:
\begin{itemize}
\item Establish asymptotic normality of the multinomial MLE $\hat A^Te_j$ for null covariates as $n,p\to+\infty$ simultaneously with a finite limit for $n/p$.
\item Develop a valid methodology for hypothesis testing of \eqref{H0} in this regime, i.e., testing for the presence of an effect of a feature $j\in[p]$
    on the multiclass response.
\end{itemize}

The contribution of this paper 
% on high-dimensional multinomial logistic regression 
is two-fold: 
(i) For a null covariate $j\in[p]$, we establish asymptotic normality results 
for $\hat A^T e_j$ that are valid in the high-dimensional regime
where $n$ and $p$ have the same order;
(ii) we propose a user-friendly test for assessing the significance of a feature in multiclass classification problems.

\section{Main result: asymptotic normality of $\hat\B{}^Te_j$ and $\hat A^Te_j$ on null covariates}
In this section, we present the main theoretical results of our work and discuss their significance. We work under the following assumptions. 
% We first state a few assumptions that will be used in our theorems. 
\begin{assumption}\label{assu:X}
    For constants $\delta>1$,
    % $q\in\mathbb Z_+$,
    assume that $n,p\to \infty$ with
    $p/n \le \delta^{-1}$,
    and that the design matrix $X\in\R^{n\times p}$ has $n$ \iid rows $(x_i)_{i\in[n]} \sim N(\bm{0}, \Sigma)$ for some invertible $\Sigma\in\R^{p\times p}$.
    The observations $(x_i,\y_i)_{i\in [n]}$ are \iid and 
    each $\y_i$ is of the form $\y_i=f(U_i,x_i^T\B^*)$ for some deterministic 
    function $f$, deterministic matrix $\B^*\in\R^{p\times (K+1)}$ such that $\B^*\bm 1_{K+1} = \bm{0}_p$,
    and latent random variable $U_i$ independent of $x_i$. 
    % Here $K$ is an integer greater than 2. 
\end{assumption}
\begin{assumption}[One-hot encoding]\label{assu:one_hot}
    The response matrix $Y$ is in $\R^{n\times (K+1)}$.
        Its $i$-th row $\y_i$ is a one-hot encoded vector, that is,
        valued in $\{0,1\}^{K+1}$
        with $\sum_{k=1}^{K+1}\y_{ik} = 1$ for each $i\in[n]$.
\end{assumption}
The model $\y_i=f(U_i,x_i^T\B^*)$ for some deterministic $f$ and $\B^*$
and latent random variable $U_i$
in \Cref{assu:X} is more general than a specific generative model such as
the multinomial logistic conditional probabilities in \eqref{model:over-specified}, as broad choices for $f$ are allowed. In words, the model
$\y_i=f(U_i,x_i^T\B^*)$ with $\B^*\bm1_{K+1}=0_p$ means that $\y_i$ only depends on $x_i$
through a $K$ dimensional projection of $x_i$ (the projection on the row-space of $\B^*$). The assumption $p/n\le \delta^{-1}$ is more general
than assuming a fixed limit for the ratio $p/n$; this allows us to cover
low-dimensional settings satisfying $p/n\to 0$ as well.

The following assumption requires the labels to be ``balanced'':
we observe each class at least $\gamma n$ times for some constant $\gamma>0$.
If $(\y_i)_{i\in[n]}$ are \iid as in \Cref{assu:X}
with distribution independent of $n,p$,
by the law of large numbers this assumption is equivalent
to $\min_{k\in[K+1]} \P(\y_{ik}=1)>0$.
\begin{assumption}\label{assu:Y}
    There exits a constant $\gamma\in (0,\tfrac{1}{K+1}]$, such that for each $k\in [K+1]$,
    with probability approaching one 
    at least $\gamma n$ observations  
    $i\in[n]$ are such that $\y_{ik}=1$. 
    In other words, 
    $\P(\sum_{i=1}^n I(\y_{ik} = 1) \ge \gamma n)\to 1$ for each $k\in [K+1]$.
\end{assumption}

As discussed in item list (i) on page 2, in binary logistic regression,
\cite{candes2020phase,sur2019modern} show that the plane
$(\frac pn, \|\Sigma^{1/2}\beta^*\|)$ is split by a smooth curve into two connected open
components: in one component the MLE does not exist with high probability,
while in the other component,
with high probability the MLE exists and is bounded in the sense
that $\|\Sigma^{1/2}\hat\beta\|^2<\tau'$ or equivalently
$\frac1n\|X\hat\beta\|^2<\tau$ for constants $\tau,\tau'$ independent 
of $n,p$. The next assumption requires the typical situation of the
latter component, in the current multiclass setting: $\hat\B$
in \eqref{barB} exists in the sense that the minimization problem
has solutions, and at least one solution is bounded.

\begin{assumption}\label{assu:MLE}
    Assume $\P(\hat \B \text{~exists} \text{ and } \|X \hat \B(I_{K+1} - \tfrac{\bm{1}\bm{1}^T}{K+1})\|_F^2 \le n \tau) \to 1$ as $n,p \to +\infty$ for some large enough constant $\tau$. 
\end{assumption}
Note that the validity of \Cref{assu:MLE} can be assessed using the data at hand;
if a multinomial regression solver (e.g. \verb|sklearn.linear_model.LogisticRegression|) converges and $\tfrac 1n\|X \hat \B(I_{K+1} - \tfrac{\bm{1}\bm{1}^T}{K+1})\|_F^2$ is no larger than a predetermined large constant $\tau$,
then we know \Cref{assu:MLE} holds. Otherwise the algorithm does not converge or produces an unbounded estimate: we know \Cref{assu:MLE} fails to hold and 
we need collect more data.

Our first main result, \Cref{thm:normal-K+1}, provides the asymptotic distribution of $\hat\B^T e_j$ where $j\in[p]$ is a null covariate,
where $\hat\B$ is any minimizer $\hat\B$ of \eqref{barB}. 
Throughout, we denote by $\Omega$ the precision matrix
defined as $\Omega=\Sigma^{-1}$.

\begin{restatable}{theorem}{mytheorem}
    \label{thm:normal-K+1}
    Let \Cref{assu:X,assu:one_hot,assu:Y,assu:MLE} be fulfilled. 
    Then for any $j\in[p]$ such that $H_0$ in \eqref{H0} holds, and 
    any minimizer $\hat\B$ of \eqref{barB}, we have
    \begin{equation}\label{eq:normal_K+1}
    \underbrace{
        \vphantom{\sum_{i=1}^n}
        \sqrt{\frac{n}{\Omega_{jj}}}
    }_{\text{scalar}}
    \Bigl(
    \underbrace{
    \Bigl(
    \frac1n\sum_{i=1}^n (\y_i - \hat\p_i)(\y_i - \hat\p_i)^T
    \Bigr)^{1/2}
    }_{\text{square root pseudo-inverse }\R^{(K+1)\times (K+1)}}
    \Bigr)^\dagger
    \underbrace{
    \Bigl(\frac1n\sum_{i=1}^n \V_i\Bigr)
    }_{\R^{(K+1)\times(K+1)}}
    \underbrace{
        \vphantom{\sum_{i=1}^n}
        \hat\B^T e_j
    }_{\R^{K+1}}
    \limd  N\Bigl(\bm{0},
    \underbrace{
        \vphantom{\sum_{i=1}^n}
        I_{K+1}-\tfrac{\bm{1}\bm{1}^T}{K+1}
    }_{\text{cov. }\R^{(K+1)\times(K+1)}}
    \Bigr), 
    \end{equation}
    where 
    $\V_i=
    \H_i - (\H_i \otimes x_i^T)[\sum_{l=1}^n \H_l\otimes (x_lx_l^T)]^{\dagger}(\H_i\otimes x_i)$. 
\end{restatable}

The proof of \Cref{thm:normal-K+1} is given in Supplementary \Cref{proof:K+1}. 
\Cref{thm:normal-K+1} establishes that under $H_0$, $\hat\B^T e_j$ converges to a singular multivariate Gaussian distribution in $\R^{K+1}$.
In \eqref{eq:normal_K+1}, the two matrices
$\frac1n\sum_{i=1}^n (\y_i - \hat\p_i)(\y_i - \hat\p_i)^T$
and $\frac1n\sum_{i=1}^n\V_i$ are symmetric with kernel being
the linear span of $\bm1_{K+1}$, and similarly, if a solution exists,
we may replace $\hat\B$ by $\hat\B(I_{K+1}-\frac{\bm1\bm1^T}{K+1})$ which
is also solution in \eqref{barB}. In this case, all matrix-matrix
and matrix-vector multiplications, matrix square root and pseudo-inverse in \eqref{eq:normal_K+1} happen
with row-space and column space contained in the orthogonal component of $\bm
1_{K+1}$, so that the limiting Gaussian distribution in $\R^{K+1}$ is also
supported on this $K$-dimensional subspace.

Since the distribution of the left-hand side of \eqref{eq:normal_K+1}
is asymptotically pivotal for all null covariates $j\in[p]$,
\Cref{thm:normal-K+1} opens the door of statistical inference for multinomial logistic regression in high-dimensional settings. 
By construction, the multinomial logistic estimate $\hat A\in\R^{p\times K}$ in \eqref{eq:relationship_A_B} ensures 
$(\hat A,\bm{0}_p)$ is a minimizer of \eqref{barB}. %  when the solution set is nonempty, 
Therefore, we can 
deduce the following theorem from \Cref{thm:normal-K+1}. 

\begin{restatable}{theorem}{mytheoremA}
    \label{thm:normal-A}
    Define the matrix $R = (I_K, \bm 0_K)^T\in \R^{(K+1)\times K}$ using block matrix notation.
    Let \Cref{assu:X,assu:one_hot,assu:Y,assu:MLE} be fulfilled.   
    For $\hat A$ in \eqref{eq:relationship_A_B} and any $j\in[p]$
    such that $H_0$ in \eqref{H0} holds,
    \begin{equation}\label{normal-A}
        \underbrace{\vphantom{\sum_{i=1}^n}
        \Bigl(I_K + \frac{\bm{1}_K\bm{1}_K^T}{\sqrt{K+1} + 1}\Bigr)R^T
        }_{\text{matrix }\R^{K\times(K+1)}}
        \underbrace{\vphantom{\sum_{i=1}^n}
            \sqrt{\frac{n}{\Omega_{jj}}}
        }_{\text{scalar}}
        \Bigl(
        \underbrace{
        \Bigl(
        \frac1n\sum_{i=1}^n \g_i \g_i^T 
        \Bigr)^{1/2}
        }_{\text{matrix }\R^{(K+1)\times(K+1)}}
        \Bigr)^\dagger
        \underbrace{\vphantom{\sum_{i=1}^n}
            \Bigl(\frac1n\sum_{i=1}^n \V_i
            R\Bigr)
        }_{\R^{(K+1)\times K}}
        \underbrace{\vphantom{\sum_{i=1}^n}
            \hat A^T e_j
        }_{\R^K}
        \limd  N(\bm{0}_K, I_K)  
    \end{equation}
    where $\g_i$ is defined in \eqref{eq:g_bar-H_bar}
    and $\V_i$ is defined in \Cref{thm:normal-K+1}.
    Furthermore, for the same $j\in[p]$,
        \begin{equation}\label{eq:chi2-A}
    \mathcal T_n^j (X, Y) \defas 
    \frac{n}{\Omega_{jj}}
    \Bigl\|
            \Bigl(
            \Bigl(
            \frac1n\sum_{i=1}^n \g_i \g_i^T 
            \Bigr)^{1/2} \Bigr)^\dagger
            \Bigl(\frac1n\sum_{i=1}^n \V_i\Bigr)
            R \hat A^T e_j
    \Bigr\|^2
    ~\text{ satisfies }~
    \mathcal T_n^j (X, Y)
            \limd \chi_K^2. 
        \end{equation}
\end{restatable}

\Cref{thm:normal-A} is proved in Supplementary \Cref{proof:A}. 
To the best of our knowledge, \Cref{thm:normal-A} is the first result that characterizes the distribution of null MLE coordinate $\hat A^T e_j$ in high-dimensional multinomial logistic regression with 3 or more classes. 
It is worth mentioning that the quantities $(\g_i, \V_i,\hat A)$ used in \Cref{thm:normal-A} can be readily computed from the data $(X, Y)$. 
Therefore, \Cref{thm:normal-A} lets us test the significance of a specific feature:
for testing $H_0$, 
this theorem suggests the
test statistic
$\mathcal T_n^j(X,Y)$ in 
\eqref{eq:chi2-A}
and the rejection region 
$
\mathcal E_\alpha^j
\defas
\bigl\{
(X, Y): \mathcal T_n^j (X, Y) 
\ge \chi^2_K(\alpha)
\bigr\}.
$
Under the null hypothesis $H_0$ in \eqref{H0}, \Cref{thm:normal-A} guarantees 
$
\P\big((X, Y) \in \mathcal E_\alpha^j
\big) \to \alpha. 
$
In other words, 
the test that rejects $H_0$ if $(X,Y)\in\mathcal E_\alpha^j$ has 
type I error converging to $\alpha$.
The p-value of this test is
\begin{equation}
    \textstyle
    \label{eq:modern_p_value}
    \text{p-value} = 
    \int_{\mathcal T_n^j(X, Y)}^{+\infty} f_{\chi^2_K}(t)dt,
    %\P(W \ge \mathcal T_n^j(X,Y)) = 1 - F_{\chi^2_K}(\mathcal T_n^j(X,Y)),
\end{equation}
where $f_{\chi^2_K}(\cdot)$ is the density of the chi-square distribution with $K$ degrees of freedom.

\textbf{Unknown $\Omega_{jj}=e_j^T\Sigma^{-1}e_j$.}
If $\Sigma$ is unknown, we describe a consistent estimate of the
quantity $\Omega_{jj}$ appearing in \eqref{eq:normal_K+1}, \eqref{normal-A}, and \eqref{eq:chi2-A}.
Under the Gaussian \Cref{assu:X}, the quantity $\Omega_{jj}$ is the reciprocal of the conditional variance $\Var(x_{ij}| x_{i,-j})$, which is also the noise variance in the linear model of regressing $X e_{j}$ onto $X_{-j}$ (the submatrix of $X$ excluding the $j$-th column). 
According to standard results in linear models, we have 
$\Omega_{jj} 
\|[I_n - X_{-j}(X_{-j}^TX_{-j})^{-1}X_{-j}^T] X e_{j}\|^2 
\sim \chi^2_{n-p+1}. 
$
Since $\chi^2_{n-p+1}/(n-p+1)\to1$ almost surely
by the strong law of large numbers,
\begin{equation}\label{eq:Omega_hat}
    \hat \Omega_{jj} 
= (n-p+1)\big/\|[I_n - X_{-j}(X_{-j}^TX_{-j})^{-1}X_{-j}^T] Xe_j\|^2
\end{equation}
is a consistent estimator of $\Omega_{jj}$.
Therefore, the previous asymptotic results in \Cref{thm:normal-K+1,thm:normal-A}
still hold by Slutsky's theorem
if we replace $\Omega_{jj}$ by the estimate $\hat\Omega_{jj}$ in \eqref{eq:Omega_hat}. 

\section{Numerical experiments}\label{sec:simulation}
% We evaluate the finite sample behavior for the normal approximation and $\chi^2$ approximation. 
This section presents simulations to examine finite sample properties of
the above results and methods.

\textbf{Simulation settings.}

We set $p=1000$ and consider different combinations of $(n,K)$. The covariance matrix $\Sigma$ is specified to be the correlation matrix of an AR(1) model with parameter $\rho=0.5$, that is, $\Sigma=(0.5^{|i-j|})_{p\times p}$. We generate the regression coefficients $A^* \in \R^{p\times K}$ once and for all as follows: sample $A_0\in\R^{p\times K}$ with first $\lceil p/4\rceil$ rows being \iid $N(\bm{0},I_K)$, and set the remaining rows to 0. We then scale the coefficients by defining $A^*=A_0(A_0^T\Sigma A_0)^{-1/2}$ so that $A^*{}^T\Sigma A^*=I_K$. With this construction, the $p$-th variable is always a null covariate , and we use this null coordinate $j=p$ to demonstrate the effectiveness of our theoretical results presented in \Cref{thm:normal-A} and the suggested test for testing $H_0$ as described in \eqref{H0}.
Using the above settings, we generate the design matrix $X\in \R^{n\times p}$ from $N(0, \Sigma)$, and then simulate the labels from a multinomial logistic model as given in \eqref{eq:log-odds}, using the coefficients $A^*\in \R^{p\times K}$. For each simulation setting, we perform 5,000 repetitions.

\textbf{Assessment of $\chi^2$ approximations.}
To assess the $\chi^2$ approximation 
\eqref{eq:chi2-A} from this paper and that of the classical theory \eqref{eq:chi2-classical}, 
we compute the two $\chi^2_K$ test statistics for each sample $(x_i, y_i)_{i=1}^n$. 
\Cref{fig:qqplots} shows the 
empirical quantiles of the two statistics versus the $\chi^2_K$ distribution quantiles. 
The results demonstrate that the quantiles (in blue) from our high-dimensional theory closely match the 45-degree line (in red), whereas the quantiles (in orange) from the classical theory significantly deviate from the 45-degree line. These findings highlight the accuracy of our proposed $\chi^2$ approximation \eqref{eq:chi2-A} over the classical result \eqref{eq:chi2-classical} when $p$ is not sufficiently small compared to $n$.
\begin{figure}[]
    \centering
    \begin{subfigure}[b]{0.32\textwidth}
        \centering
        \includegraphics[width=\textwidth]{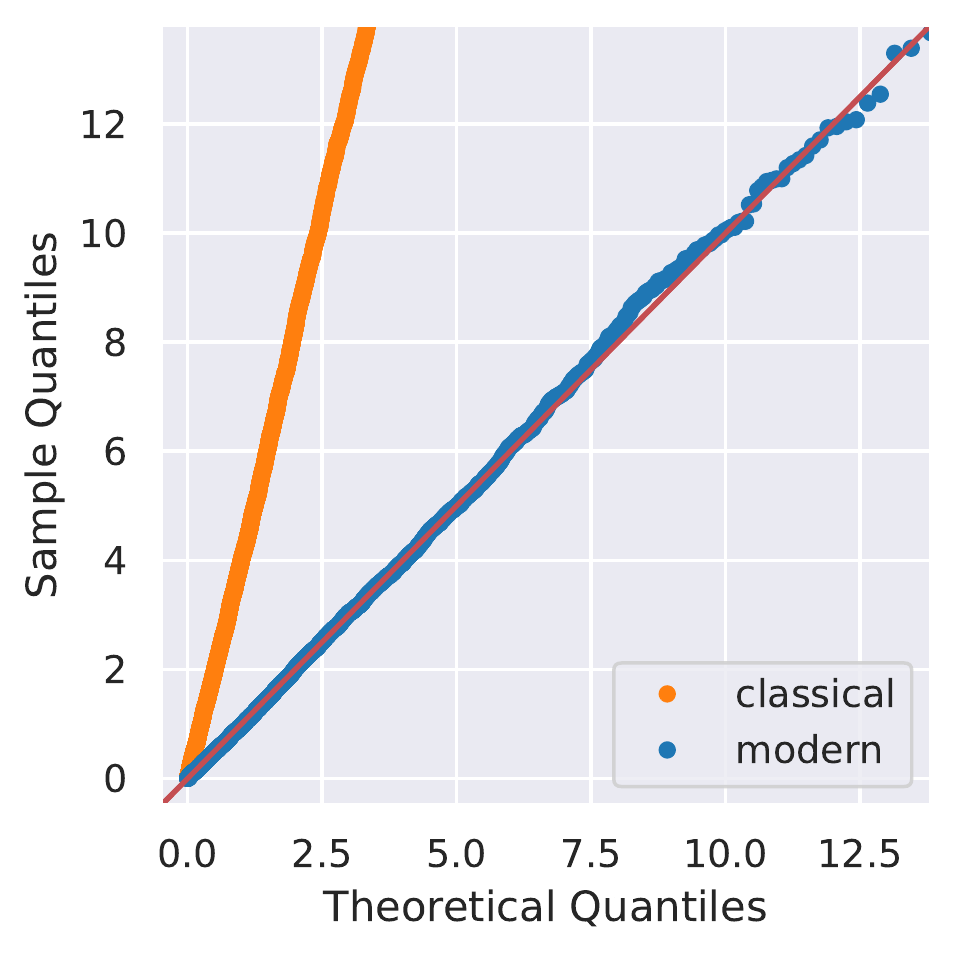}
        \caption{\footnotesize $(n,K)=(4000,2)$}
    \end{subfigure}
    \hfill
    \begin{subfigure}[b]{0.32\textwidth}
        \centering
        \includegraphics[width=\textwidth]{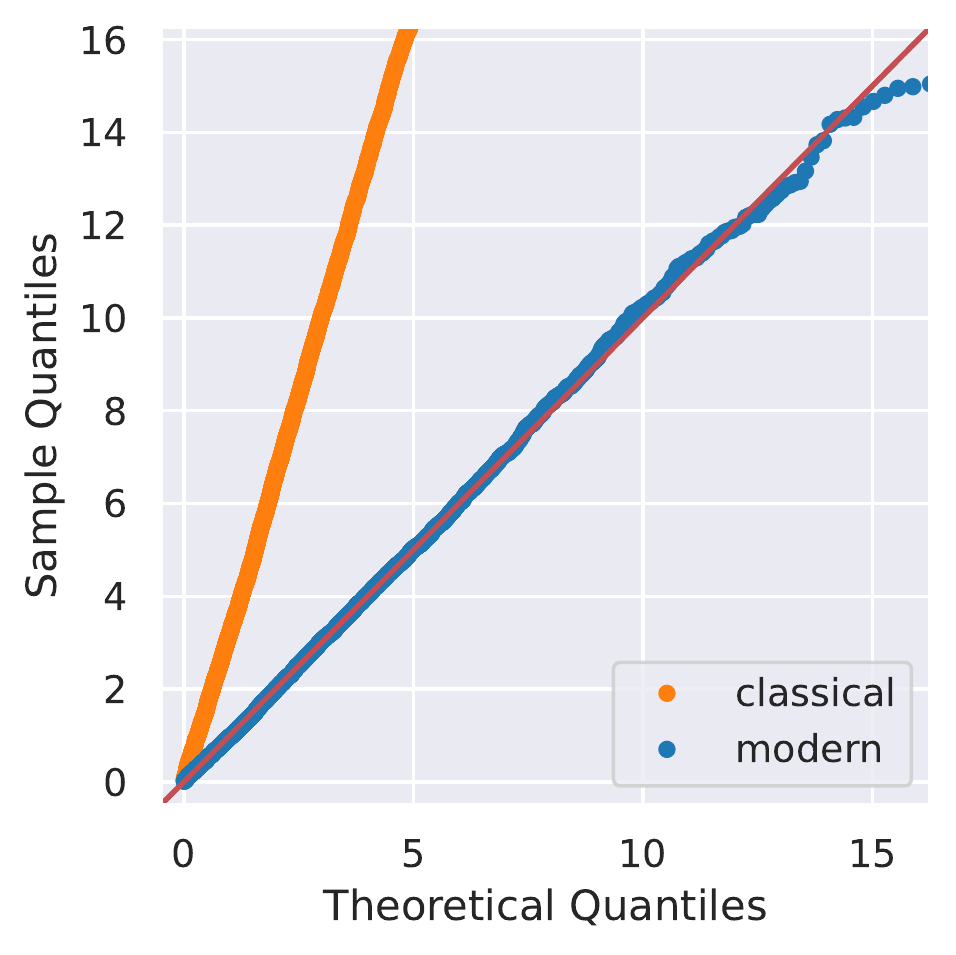}
        \caption{\footnotesize $(n,K)=(5000,3)$}
        \end{subfigure}
    \hfill
    \begin{subfigure}[b]{0.32\textwidth}
        \centering
        \includegraphics[width=\textwidth]{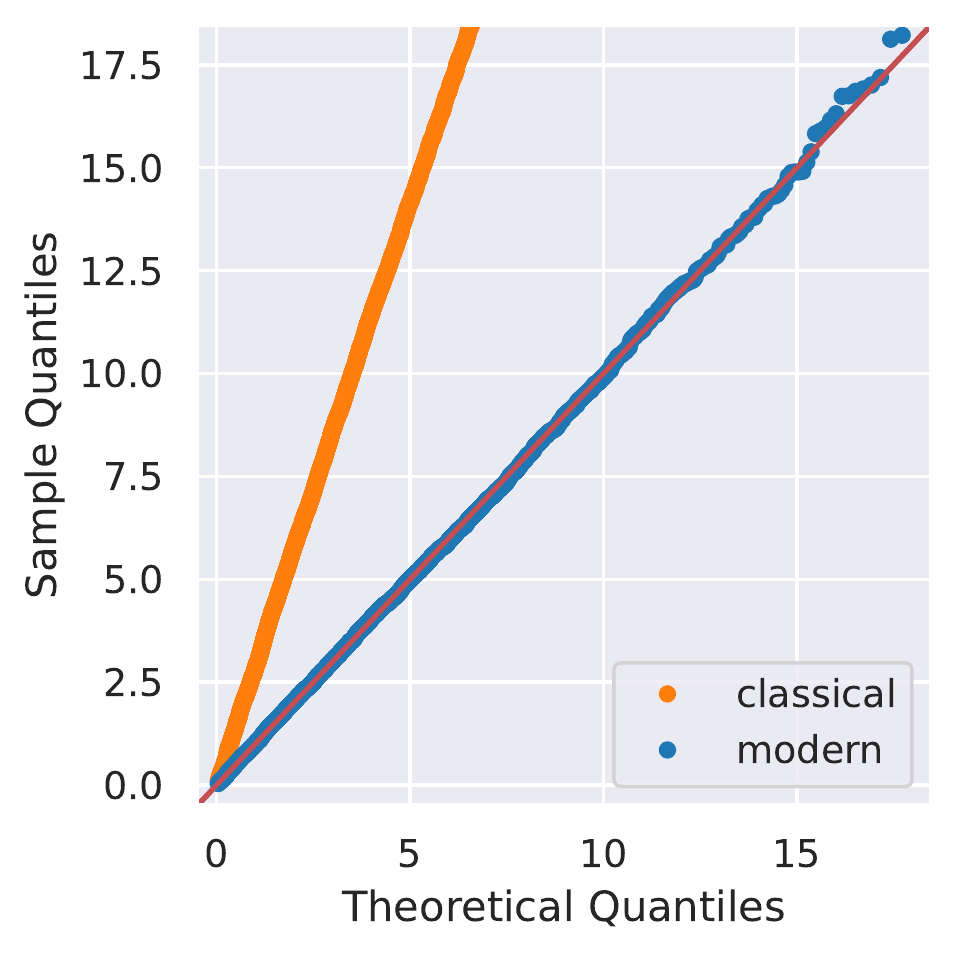}
        \caption{\footnotesize $(n,K)=(6000,4)$}
    \end{subfigure}
    \caption{Q-Q plots of the test statistic in the left-hand side of
    \eqref{eq:chi2-classical} (in orange) and in the left-hand side of \eqref{eq:chi2-A} (in blue) for different $(n,K)$ and $p=1000$.}
    \label{fig:qqplots}
\end{figure}
\begin{figure}[]
    \centering
    \begin{subfigure}[b]{0.32\textwidth}
        \centering
        \includegraphics[width=\textwidth]{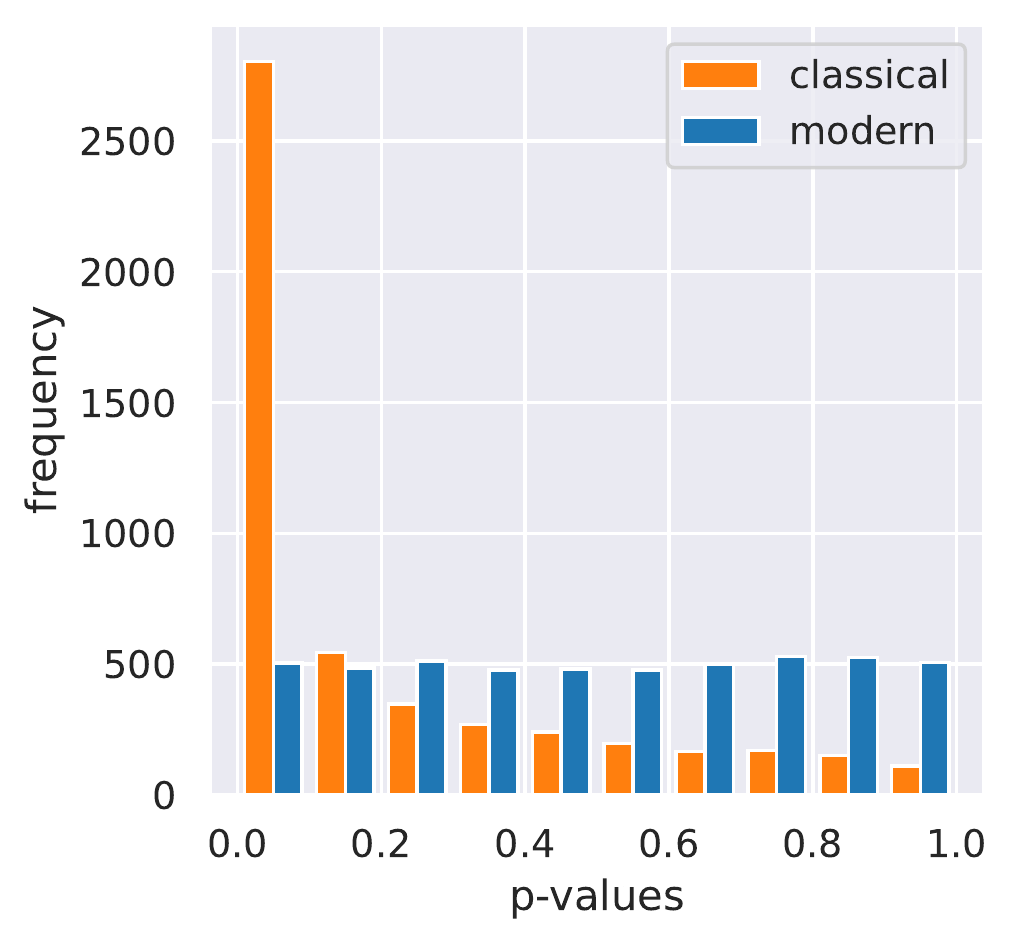}
        \caption{\footnotesize $(n,K)=(4000,2)$}
    \end{subfigure}
    \hfill
    \begin{subfigure}[b]{0.32\textwidth}
        \centering
        \includegraphics[width=\textwidth]{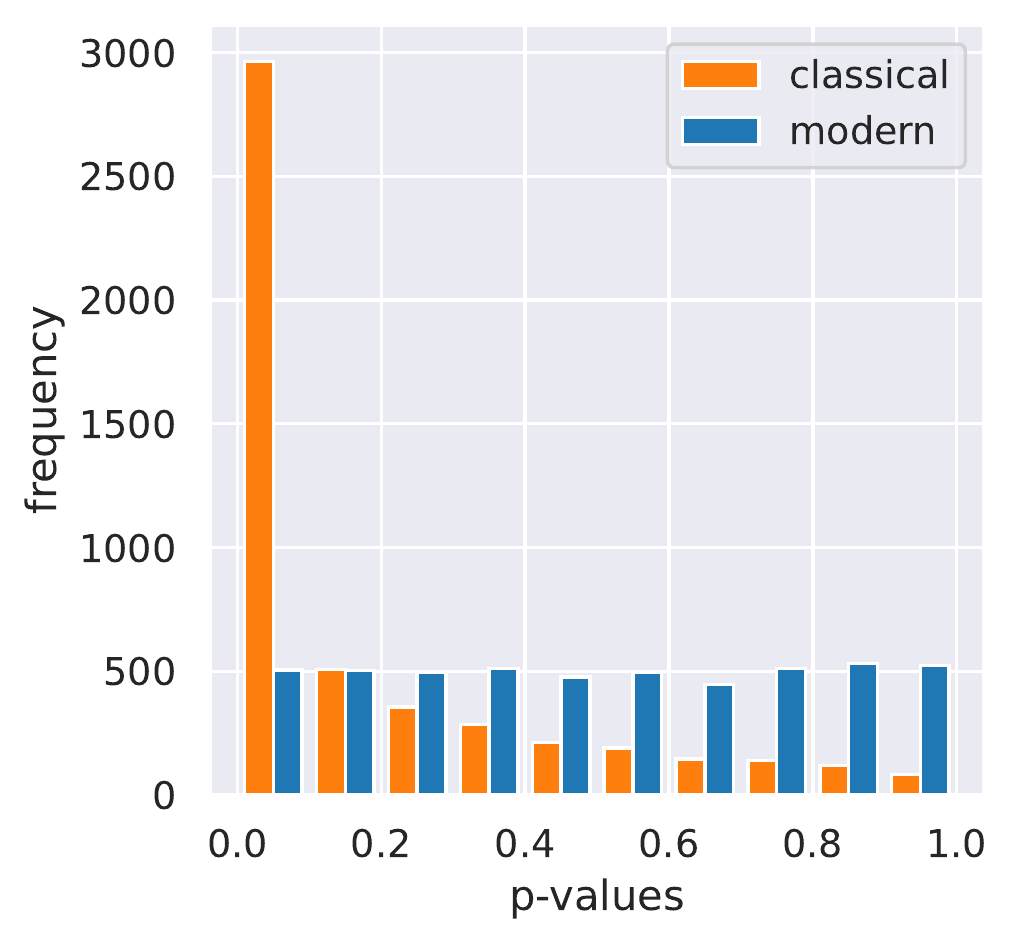}
        \caption{\footnotesize $(n,K)=(5000,3)$}
        \end{subfigure}
    \hfill
    \begin{subfigure}[b]{0.32\textwidth}
        \centering
        \includegraphics[width=\textwidth]{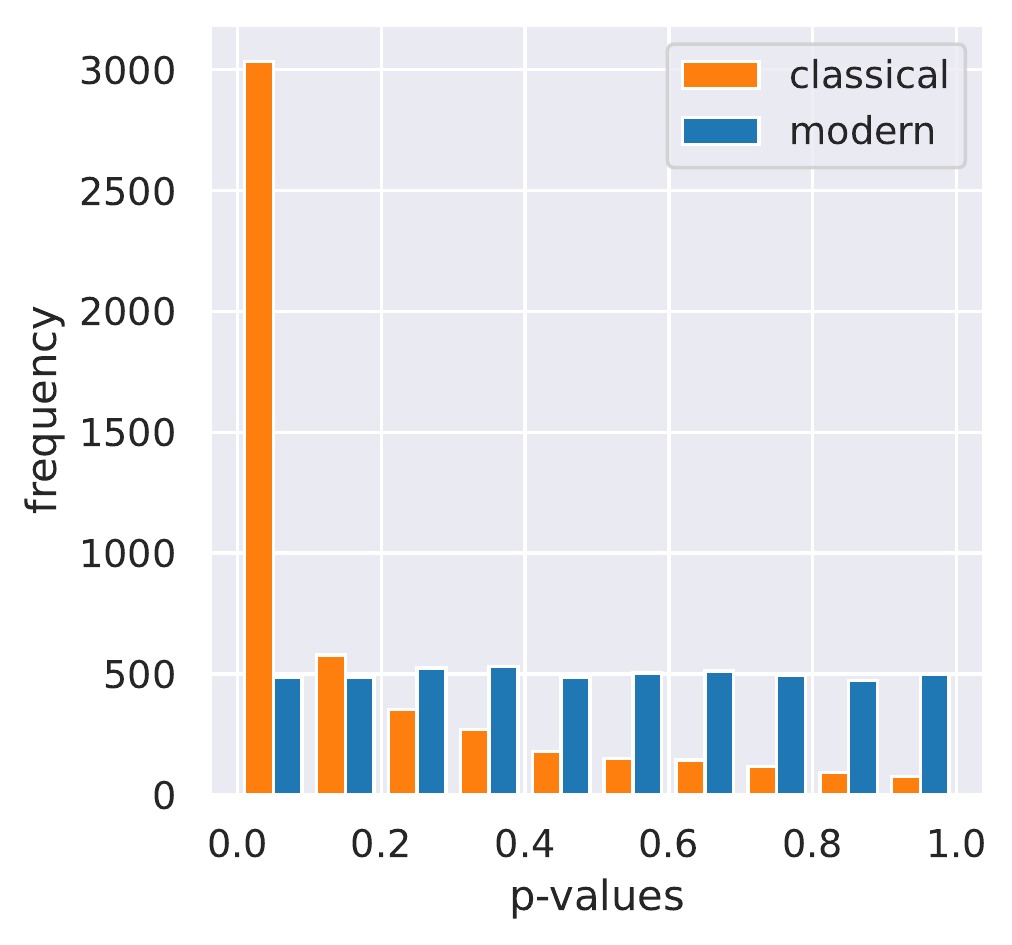}
        \caption{\footnotesize $(n,K)=(6000,4)$}
    \end{subfigure}
    \caption{Histogram for p-values of the classical test \eqref{eq:classical-p-value} (in orange) and of the proposed test 
    \eqref{eq:modern_p_value} (in blue) under $H_0$ in simulated data with different $(n,K)$ and $p=1000$.}
    \label{fig:p-values}
\end{figure}

\textbf{Uniformity of null p-values.}
Recall that the p-value from the classical test \eqref{eq:chi2-classical} is given by \eqref{eq:classical-p-value}, while the p-value from this paper taking into account high-dimensionality is given by \eqref{eq:modern_p_value}. 
\Cref{fig:p-values} displays the histograms of these two sets of p-values out of 5000 repetitions. 
The results in \Cref{fig:p-values} show that the p-values obtained from the classical test deviate significantly from the uniform distribution, with a severe inflation in the lower tail. 
This indicates that the classical test tends to produce large type I errors due to the excess of p-values close to 0.
In contrast, the p-values proposed in this paper exhibit a uniform distribution, further confirming the effectiveness and applicability of the theory in \Cref{thm:normal-A} for controlling type I error when testing for null covariates with \eqref{H0}.

\textbf{Unknown $\Omega_{jj}$.}
In the situation where the covariance matrix $\Sigma$ is unknown, we can estimate the diagonal element $\Omega_{jj}=e_j^T\Sigma^{-1}e_j$ by ${\smash{\hat\Omega_{jj}}}$ defined in \eqref{eq:Omega_hat}. 
To evaluate the accuracy of the normal and chi-square approximations and the associated test with $\Omega_{jj}$ replaced by ${\smash{\hat\Omega_{jj}}}$, 
we conduct simulations similar to those in \Cref{fig:qqplots,fig:p-values}, but we replace $\Omega_{jj}$ with its estimate ${\smash{\hat\Omega_{jj}}}$.
The results are presented in Figure~\ref{fig:qqplots-estimated}. The plots
are visually indistinguishable from the plots using $\Omega_{jj}$.
These confirm that the chi-square approximation and the associated test using ${\smash{\hat\Omega_{jj}}}$ are accurate.

\textbf{Non-Gaussian covariates and unknown $\Omega_{jj}$.}
Although our theory assumes Gaussian covariates, we expect that the same results hold for other distributions with sufficiently light tails. 
To illustrate this point, we consider the following two types of non-Gaussian covariates:
(i)
        The design matrix $X$ has \iid Rademacher entries, i.e., $\P(x_{ij}=\pm1)=\frac12$,
(ii)
        Each $x_{ij}$ takes on values 0, 1 and 2 with respectively probabilities $a_j^2, 2a_j(1-a_j)$, and $(1-a_j)^2$, where $a_j$ varies in $[0.25, 0.75]$. Each columns of $X$ are then centered and normalized to have $0$ mean and unit variance. 
    This generation of non-Gaussian covariates is adopted from single-nucleotide poly-morphisms (SNPs) example in \cite{sur2019modern}. 
For these two types of non-Gaussian covariates, we further rescale the feature vectors to ensure that $x_i$ has the same covariance as in the Gaussian case at the beginning of \Cref{sec:simulation}, that is $\Sigma=(0.5^{|i-j|})_{p\times p}$. 
We present the Q-Q plots in \Cref{fig:qqplots-non-Gaussian} using the same settings as in \Cref{fig:qqplots-estimated}, with the only difference being that the covariates in \Cref{fig:qqplots-non-Gaussian} are non-Gaussian distributed.
The Q-Q plots of $\mathcal T_n^j(X,Y)$ in \eqref{eq:chi2-A} plotted in \Cref{fig:qqplots-non-Gaussian} still closely match the diagonal line.
These empirical successes suggest that the normal and $\chi^2_K$ approximations
\eqref{eq:normal_K+1}-\eqref{eq:chi2-A} apply
to a wider range of covariate distributions beyond normally distributed data.

\section{Discussion and future work}
Multinomial logistic regression estimates and their p-values 
are ubiquitous throughout the sciences for analyzing the significance
of explanatory variables on multiclass responses.
Following the seminal work of \cite{sur2019modern} in binary
logistic regression, this paper develops the first valid 
tests and p-values for multinomial logistic estimates
when $p$ and $n$ are of the same order.
For 3 or more classes, this methodology and the corresponding
asymptotic normality results in \Cref{thm:normal-K+1,thm:normal-A}
are novel and provide new understanding of
multinomial logistic estimates (also known as cross-entropy minimizers)
in high-dimensions. 
We expect similar asymptotic normality and chi-square results
to be within reach for loss functions different than the cross-entropy
or a different model for the response $\y_i$; for instance
\Cref{sec:supplement_q_repeated} provides an extension to the 
$q$-repeated measurements model, where $q$ responses are observed
for each feature vector $x_i$.

Let us point a few follow-up research directions that we leave open for future
work. A first open problem regards extensions of our methodology to 
confidence sets for 
$e_j^T\B^*$ when $H_0$ in \eqref{H0} is violated for the $j$-th covariate.
This would require more stringent assumptions on the generative model
than \Cref{assu:X} as
$\B^*$ there is not identifiable (e.g., modification of both $\B^*$
and $f(\cdot,\cdot)$ in \Cref{assu:X} is possible without changing $\y_i$).
A second open problem is to relate this paper's theory to the 
fixed-point equations and limiting Gaussian model obtained in
multiclass models, e.g., \cite{loureiro2021learning}.
While it may be straightforward to obtain the limit
of $\frac 1n\sum_{i=1}^n\g_i\g_i^T$ and of the empirical distribution
of the rows of ${\smash{\hat A}}$ in this context (e.g., using Corollary 3
in \cite{loureiro2021learning}), the relationship between the fixed-point
equations and the matrix $\frac1n\sum_{i=1}^n\V_i$ appearing in \eqref{eq:normal_K+1} is unclear and not explained by typical results from this literature.
A third open problem is to characterize the exact phase transition below
which the multinomial logistic MLE exists and is bounded with high-probability
(\Cref{assu:MLE}); while this is settled for two classes \citep{candes2020phase} and preliminary results are available for 3 or more classes
\citep{loureiro2021learning,kini2021phase}, a complete understanding
of this phase transition is currently lacking.
A last interesting open problem is to prove that our theory 
extend to non-Gaussian data, as observed in simulations.
This challenging problem is often referred to as ``universality'' and has received
intense attention recently
\citep{montanari2022universality,gerace2022gaussian,pesce2023gaussian,dandi2023universality}, showing that in several settings of interest (although none exactly the one considered here), the asymptotic behavior
of the minimizers is unchanged if the distribution of the covariates
is modified from normal to another distribution with the same covariance.

\bibliographystyle{plainnat} %unsrt plainnat
\bibliography{biblio}

\newpage
%\appendix 
\setcounter{section}{0}
\setcounter{figure}{0}
\renewcommand{\thesection}{S\arabic{section}}
\renewcommand{\thefigure}{S\arabic{figure}}

%\tableofcontents

{\center\Large
    Supplementary Material of 
	``Multinomial Logistic Regression:
    Asymptotic Normality on Null Covariates in High-Dimensions''
    \\[5mm]
}
Let us define some standard notation that will be used in the rest of this supplement. 
For a vector $v\in \R^{n}$, let $\|v\|_{\infty} = \max_{i\in[n]} |v_i|$ denote the infinity norm of vector $v$. 
%For a real matrix $A$, we define $\phi_{\min}(A)$ and $\phi_{\max}(A)$ as the minimum and maximum singular values of matrix $A$, respectively.
If $A$ is symmetric, we define $\lambda_{\min}(A)$ and $\|A\|_{op}$ as the minimal and maximal eigenvalues of $A$, respectively.
For two symmetric matrices $A,B$ of the same size, we write $A\preceq B$ if and only if $B-A$ is positive semi-definite.

\section*{Diagram: Organization of the proofs}

The following diagram summarizes the different theorems and lemmas, and
the relationships between them.

\begin{center}
{\tiny
\begin{tikzpicture}[node distance=2.2in]
    \node(22){
            \begin{tcolorbox}[width=1.7in,title=\Cref{thm:normal-A}]
                Asymptotic normality for $\hat A^Te_j$ on null covariates,
                where $\hat A\in\R^{p\times K}$ is the multinomial logistic MLE
                with class $K+1$ fixed as the reference class
                (see \eqref{eq:relationship_A_B}).
        \end{tcolorbox}
    };
    \node(21) [right of=22]{
            \begin{tcolorbox}[width=1.7in,title=\Cref{thm:normal-K+1}]
                Asymptotic normality for $\hat \B^Te_j$ on null covariates,
                where $\hat \B\in\R^{p\times (K+1)}$ is the multinomial logistic MLE
                in \eqref{barB}.
        \end{tcolorbox}
    };
    \node(S31) [right of =21] { 
            \begin{tcolorbox}[width=1.7in,title=\Cref{thm:normal}]
                Asymptotic normality for $\hat B^Te_j$ on null covariates,
                where $\hat B\in\R^{p\times K}$ is the multinomial logistic MLE
                using the parameter space from \Cref{sec:model-K}.

                The proof uses that the conditions in \Cref{thm:general_Sigma}
                on the loss function are satisfied by the cross-entropy.
        \end{tcolorbox}
        };
    \node(Q) [below of =22, node distance=1.7in] {
            \begin{tcolorbox}[width=1.7in,title=$K$-dimensional orthogonal parametrization defined by the matrix the $Q$]
                \Cref{sec:model-K} defines the matrix $Q\in\R^{(K+1)\times K}$
                and discusses a convenient parametrization of the model isometric to the subspace
                orthogonal to $\bm 1_{K+1}$. 
            \end{tcolorbox}
        };
    \node(GradHes) [below of =21, node distance=2.4in] {
            \begin{tcolorbox}[width=1.7in,title=Control of $\g_i$ and $\H_i$ for the cross-entropy loss]
\Cref{lem:Hessian>0,lem:gradient>0} give deterministic arguments to control
the gradients and Hessians of the cross-entropy loss. \Cref{lem:HX} controls the Hessian of the cross-entropy loss at the minimizer, in a specific high-probability event.
\Cref{lem:XtoG-Lipschitz} defines this high-probability event.
            \end{tcolorbox}
        };
    \node(GeneralSigma) [below of =S31] {
            \begin{tcolorbox}[width=1.7in,title=\Cref{thm:general_Sigma}]
                Asymptotic
                normality on null covariates for general loss functions,
                $\Sigma\ne I_p$.

                Deduced from \Cref{thm:general_identity} by rotational
                invariance.
            \end{tcolorbox}
        };
    \node(GeneralIdentity) [below of =GeneralSigma] {
            \begin{tcolorbox}[width=1.7in,title=\Cref{thm:general_identity}]
                Asymptotic
                normality on null covariates for general loss functions,
                $\Sigma=I_p$.
            \end{tcolorbox}
        };
    \node(Chi2) [below of =GradHes, node distance=2.5in] {
            \begin{tcolorbox}[width=1.7in,title=\Cref{lem:chi2}]
                Normal and $\chi^2$ approximations for random variables
                defined as a differentiable function of standard normal vectors.
            \end{tcolorbox}
       };
    \node(Derivatives) [below of =Q, node distance=1.8in] {
            \begin{tcolorbox}[width=1.7in, title=\Cref{lem:dot-g}]
                    \Cref{lem:dot-g} computes the derivatives of the
                    minimizer with respect to $X$, used in the proof
                    of \Cref{thm:general_identity}.
            \end{tcolorbox}
       };
\draw[blue, line width=2mm,-{Triangle[angle=60:1pt 3]}] (21) -- (22);
\draw[blue, line width=2mm,-{Triangle[angle=60:1pt 3]}] (S31) -- (21);
\draw[blue, line width=2mm,-{Triangle[angle=60:1pt 3]}] (Q) -- (S31);
\draw[blue, line width=2mm,-{Triangle[angle=60:1pt 3]}] (GradHes) -- (S31);
\draw[blue, line width=2mm,-{Triangle[angle=60:1pt 3]}] (GeneralSigma) -- (S31);
\draw[blue, line width=2mm,-{Triangle[angle=60:1pt 3]}] (GeneralIdentity) -- (GeneralSigma);
\draw[blue, line width=2mm,-{Triangle[angle=60:1pt 3]}] (Chi2) -- (GeneralIdentity);
 \draw[blue, line width=2mm,-{Triangle[angle=60:1pt 3]}] (Derivatives) -- (GeneralIdentity);
\end{tikzpicture}
}
\end{center}

\newpage
\section{Extension: $q$ repeated measurements}
\label{sec:supplement_q_repeated}

Let integer $q\ge 1$ be a constant independent of $n,p$.
Our results readily extend if $q$ labels are observed for each observed feature
vector $x_i$, and the corresponding $q$ one-hot encoded vectors
are averaged into $\y_i\in\{0,\frac1q,\frac2q,...,1\}^{K+1}$.
Concretely, for each observation $i\in[n]$,
$q$ \iid labels $(Y_i^m)_{m\in[q]}$ are observed with each $Y_i^m\in\{0,1\}^{K+1}$ one-hot encoded and $\y_{ik} = \frac1q \sum_{m=1}^q Y_{ik}^m$,
for instance in a repeated multinomial regression model with
$\P(Y_{ik}^m=1|x_i)$ equal to right-hand side of \eqref{model:over-specified}.
In this case where $(Y_i^m)_{m\in[q]}$ are \iid,
\Cref{assu:Y} is satisfied by the law of large numbers
if $\min_{k\in[K+1]}\P(Y_{ik}^m = 1)>0$ since $q$ is constant.
For this $q$ repeated measurements model, the negative log-likelihood function of a parameter $\B\in \R^{p\times (K+1)}$ is 
\begin{align*}
    &-\sum_{i=1}^n \sum_{m=1}^q \sum_{k=1}^{K+1}
    Y_{ik}^m \Bigl[x_i^T \B e_k - \log \sum_{k'=1}^{K+1} \exp(x_i^T \B e_{k'})\Bigr]\\
    =~& q \sum_{i=1}^n \sum_{k=1}^{K+1}
    \y_{ik} \Bigl[-x_i^T \B e_k + \log \sum_{k'=1}^{K+1} \exp(x_i^T \B e_{k'})\Bigr]\\
    =~& q \sum_{i=1}^n \Bigl[\sum_{k=1}^{K+1}
    -\y_{ik} x_i^T \B e_k + \log \sum_{k'=1}^{K+1} \exp(x_i^T \B e_{k'})\Bigr]\\
    =~& q \sum_{i=1}^n \L_i(\B^T x_i),
\end{align*}
where the first equality uses 
$\y_{ik} = \frac{1}{q} \sum_{m=1}^q Y_{ik}^m$, the second equality uses 
$\sum_{k=1}^{K+1} \y_{ik}=1$ under the following \Cref{assu:one_hot_q}, and the last equality uses the definition of $\L_i$ after \eqref{barB}. 
\begin{assumption}\label{assu:one_hot_q}
    For all $i\in[n],$ the response $\y_i$ is in $\{0,1/q,2/q,...,1\}^{K+1}$
    with $\sum_{k=1}^{K+1} \y_{ik}=1$.
\end{assumption}
In such repeated measurements model, we replace \Cref{assu:one_hot} with
\Cref{assu:one_hot_q} under which the following \Cref{thm:repeated} holds.

\begin{theorem}
    \label{thm:repeated}
    Let $q\ge 2$ be constant.
    Let \Cref{assu:X,assu:one_hot_q,assu:Y,assu:MLE} be fulfilled. 
    For any $j\in[p]$ such that $H_0$ in \eqref{H0} holds,
    we have the convergence in distribution \eqref{eq:normal_K+1}, \eqref{normal-A} and
    \eqref{eq:chi2-A}.
\end{theorem}

\begin{proof}[Proof of \Cref{thm:repeated}]
Under the assumptions in \Cref{thm:repeated}, 
the MLE $\hat \B$ for this $q$ repeated measurements model is the minimizer of the optimization problem 
$$\hat \B \in \argmin_{\B\in \R^{p\times (K+1)}}\sum_{i=1}^n \L_i(\B^T x_i)$$
as in \eqref{barB}. 
Similar to the non-repeated model, the MLE $\hat A$ for the identifiable log-odds model can be expressed as
$$\hat A = \argmin_{A\in \R^{p\times K}}\sum_{i=1}^n \L_i((A,\bm{0}_p)^T x_i). $$ 
The only difference between this $q$ repeated measurements model and the non-repeated model considered in the main text is that the response $\y_{ik}$ for this $q$ repeated measurements model is now valued in $\{0, 1/q, 2/q, ..., 1\}$. 
Because the proofs of \Cref{thm:normal-K+1,thm:normal-A} do not require the value of $\y_{ik}$ to be $\{0,1\}$-valued. 
\Cref{thm:repeated} can be proved by the same arguments used in the proof of \Cref{thm:normal-K+1,thm:normal-A}. 
\end{proof}

\section{Implementation details and additional figures}
The pivotal quantities in our main results \Cref{thm:normal-K+1,thm:normal-A} involve only observable quantities that can be computed from the data $(x_i, \y_i)_{i\in[n]}$. 
In this section we provide an efficient way of computing the matrix $\V_i$ appearing in \Cref{thm:normal-K+1,thm:normal-A}. 

\paragraph{Fast computation of $\V_i$.}
Recall the definition of $\V_i$ in \Cref{thm:normal-K+1}, 
$$\V_i=
    \H_i - (\H_i \otimes x_i^T)\Bigl[\sum_{l=1}^n \H_l\otimes (x_lx_l^T)\Bigr]^{\dagger}(\H_i\otimes x_i). $$
The majority of computational cost in calculating $\V_i$ lies in the step of calculating its second term
$$(\H_i \otimes x_i^T)\Bigl[\sum_{l=1}^n \H_l\otimes (x_lx_l^T)\Bigr]^{\dagger}(\H_i \otimes x_i^T).$$
Here we provide an efficient way to compute this term using the Woodbury matrix identity. 
Since 
$\H_i \bm 1_{K+1} = \bm 0_{K+1}$, we have 
$\ker(\H_i\otimes (x_ix_i^T))$ is the span of 
$\{\bm 1_{K+1}\otimes e_j : j\in [p]\}$, where $\bm 1_{K+1}$ is the all-ones vector in $\R^{K+1}$. 
Therefore, the second term in $\V_i$ can be rewritten as 
\begin{align*}
    &(\H_i \otimes x_i^T) \Bigl[\sum_{l=1}^n \H_l\otimes (x_lx_l^T)\Bigr]^{\dagger}(\H_i\otimes x_i) \\
    =~& (\H_i \otimes x_i^T) \Bigl[\sum_{l=1}^n \H_l\otimes (x_lx_l^T) 
    - \sum_{j=1}^p(\bm 1\otimes e_j) (\bm 1 \otimes e_j)^T\Bigr]^{-1}(\H_i\otimes x_i).
\end{align*}
We now apply the Woodbury matrix identity to compute the matrix inversion in the above display.
Recall $\H_i = \diag(\hat\p_i) - \hat\p_i \hat\p_i^T$, we have 
$$
\sum_{i=1}^n \H_i\otimes (x_ix_i^T)
= \sum_{k=1}^{K+1} (e_k e_k^T) \otimes (\sum_{i=1}^n \hat\p_{ik}x_ix_i^T)
- 
\sum_{i=1}^n (\hat\p_i\otimes x_i)(\hat\p_i\otimes x_i)^T. 
$$
Let 
$
A = \sum_{k=1}^{K+1} (e_k e_k^T) \otimes (\sum_{i=1}^n \hat\p_{ik}x_ix_i^T)
$, 
and 
$
U \in\R^{p(K+1)\times (n+p)}
$ with the first $n$ columns being $(\hat\p_i\otimes x_i)_{i\in[n]}$ and the following $p$ columns $(\bm 1\otimes e_j)_{j\in[p]}$. 
Then the term we want to invert is $A - UU^T$, where $A$ is a block diagonal matrix and can be inverted by inverting each block separately.
By the Woodbury matrix identity, we have 
\begin{align*}
(A - U U^T)^{-1} 
    &= A^{-1} - A^{-1} U (-I_{n+p} + U^T A^{-1} U)^{-1} U^T A^{-1}. 
\end{align*}
The gain of using the above formula is significant for large $K$:
instead of inverting the $p(K+1)\times p(K+1)$ matrix $\sum_{l=1}^n \H_l \otimes (x_lx_l^T)$ in the left-hand side, the right-hand side only requires to invert a block diagonal matrix $A$ and a $(n+p)\times (n+p)$ matrix $-I_{n+p} + U^T A^{-1} U$. 
\newpage

\begin{figure}[H]
    \centering
    \begin{subfigure}[b]{0.32\textwidth}
        \centering
        \includegraphics[width=\textwidth]{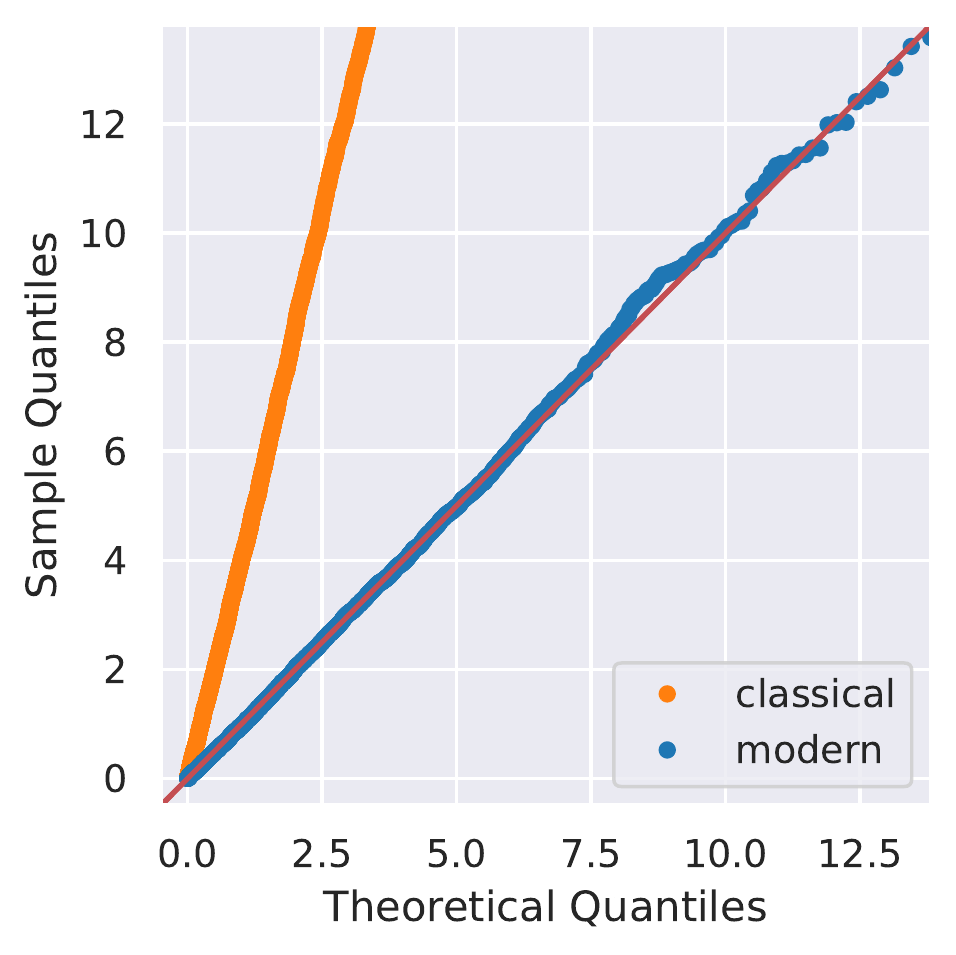}
        % \caption{\footnotesize $(n,K)=(4000,2)$}
    \end{subfigure}
    \hfill
    \begin{subfigure}[b]{0.32\textwidth}
        \centering
        \includegraphics[width=\textwidth]{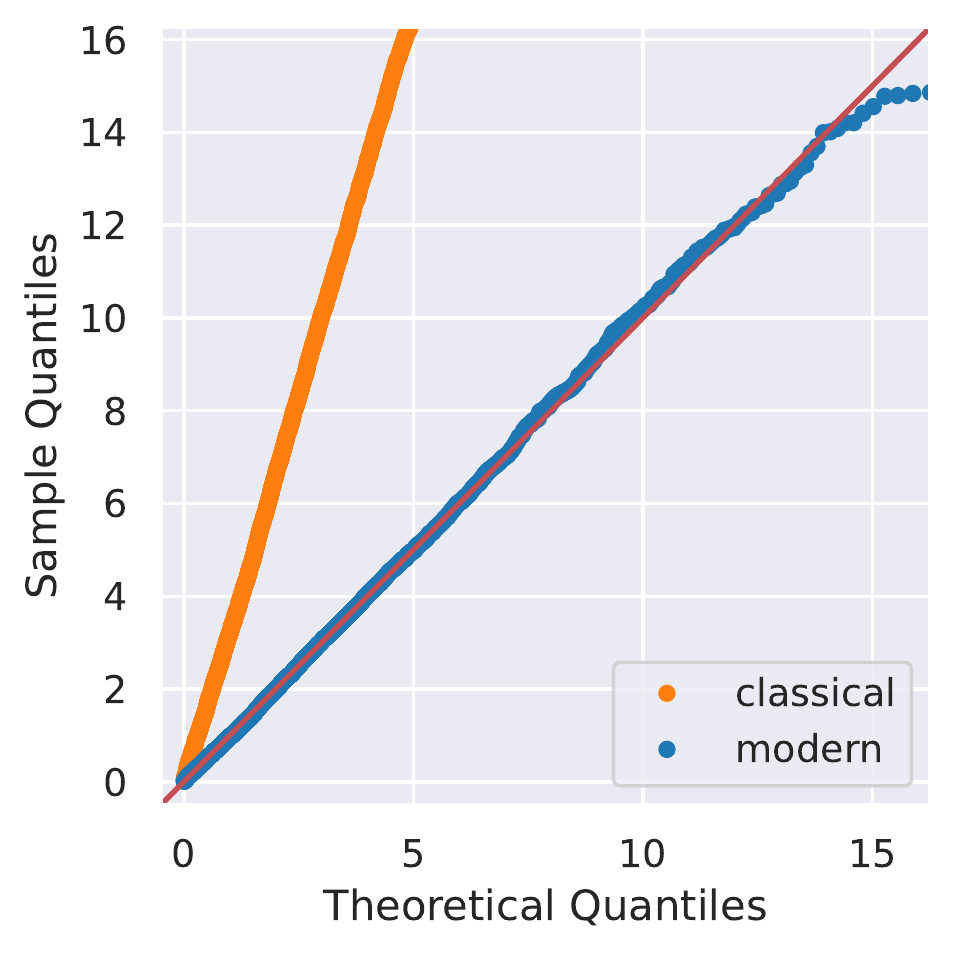}
        % \caption{\footnotesize $(n,K)=(5000,3)$}
        \end{subfigure}
    \hfill
    \begin{subfigure}[b]{0.32\textwidth}
        \centering
        \includegraphics[width=\textwidth]{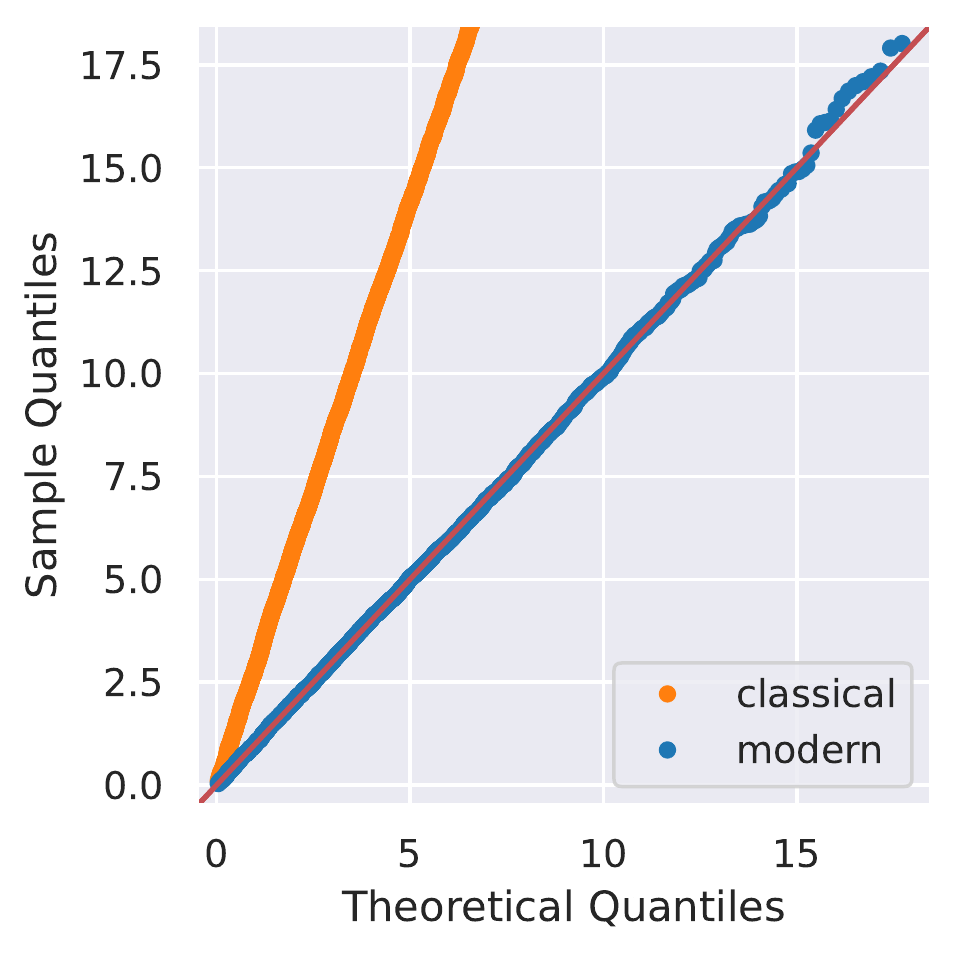}
        % \caption{\footnotesize $(n,K)=(6000,4)$}
    \end{subfigure}
    \\
    \begin{subfigure}[b]{0.32\textwidth}
        \centering
        \includegraphics[width=\textwidth]{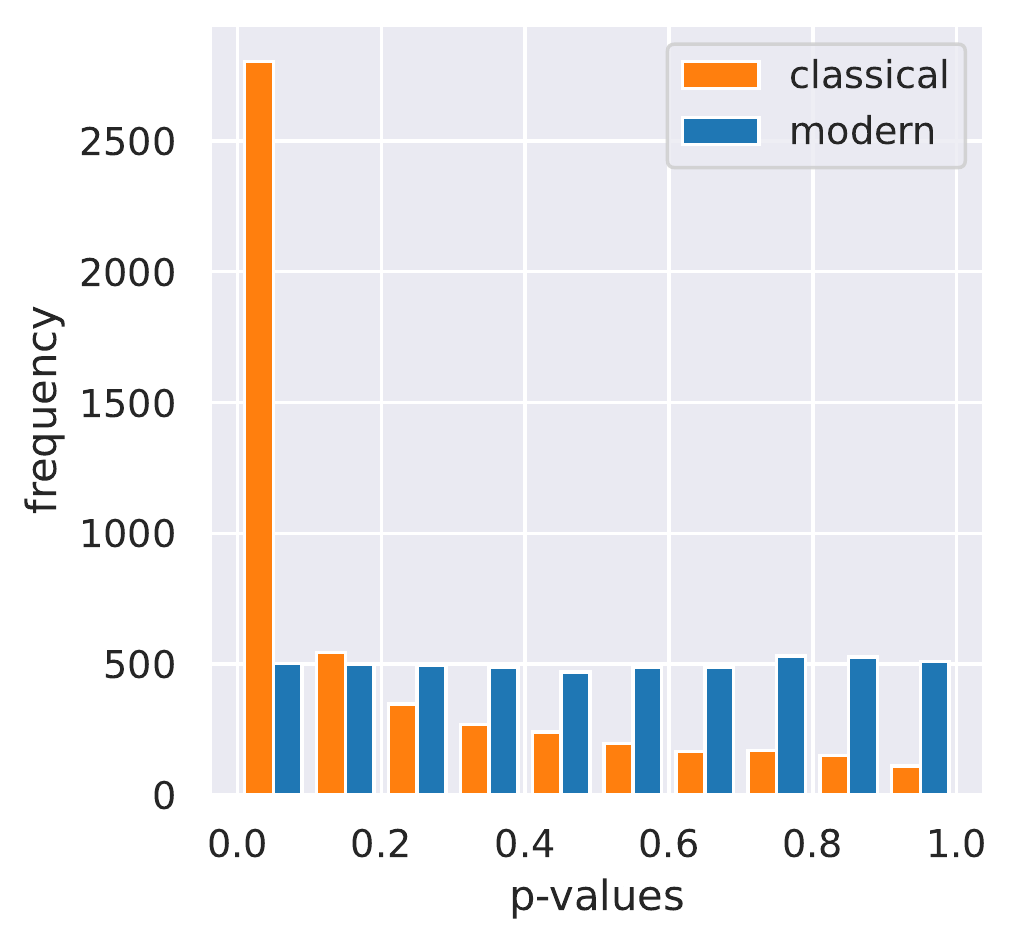}
        \caption{\footnotesize $(n,K)=(4000,2)$}
    \end{subfigure}
    \hfill
    \begin{subfigure}[b]{0.32\textwidth}
        \centering
        \includegraphics[width=\textwidth]{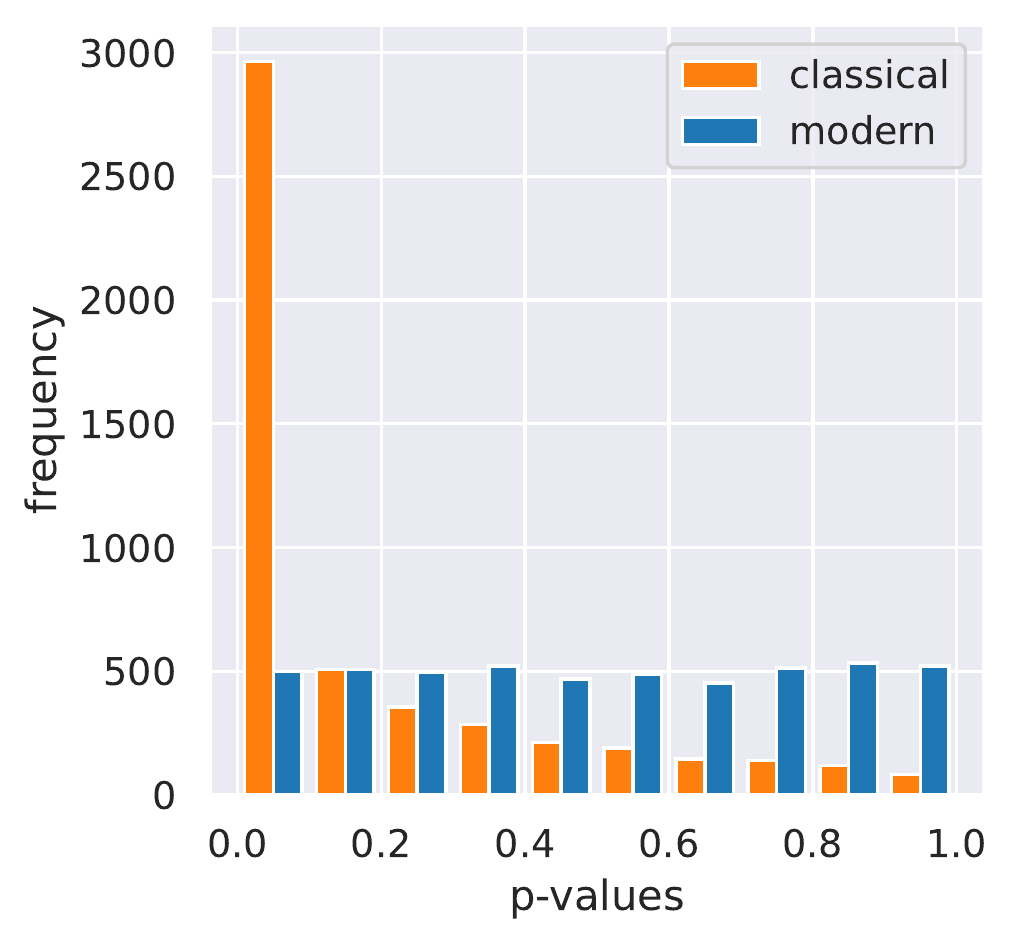}
        \caption{\footnotesize $(n,K)=(5000,3)$}
        \end{subfigure}
    \hfill
    \begin{subfigure}[b]{0.32\textwidth}
        \centering
        \includegraphics[width=\textwidth]{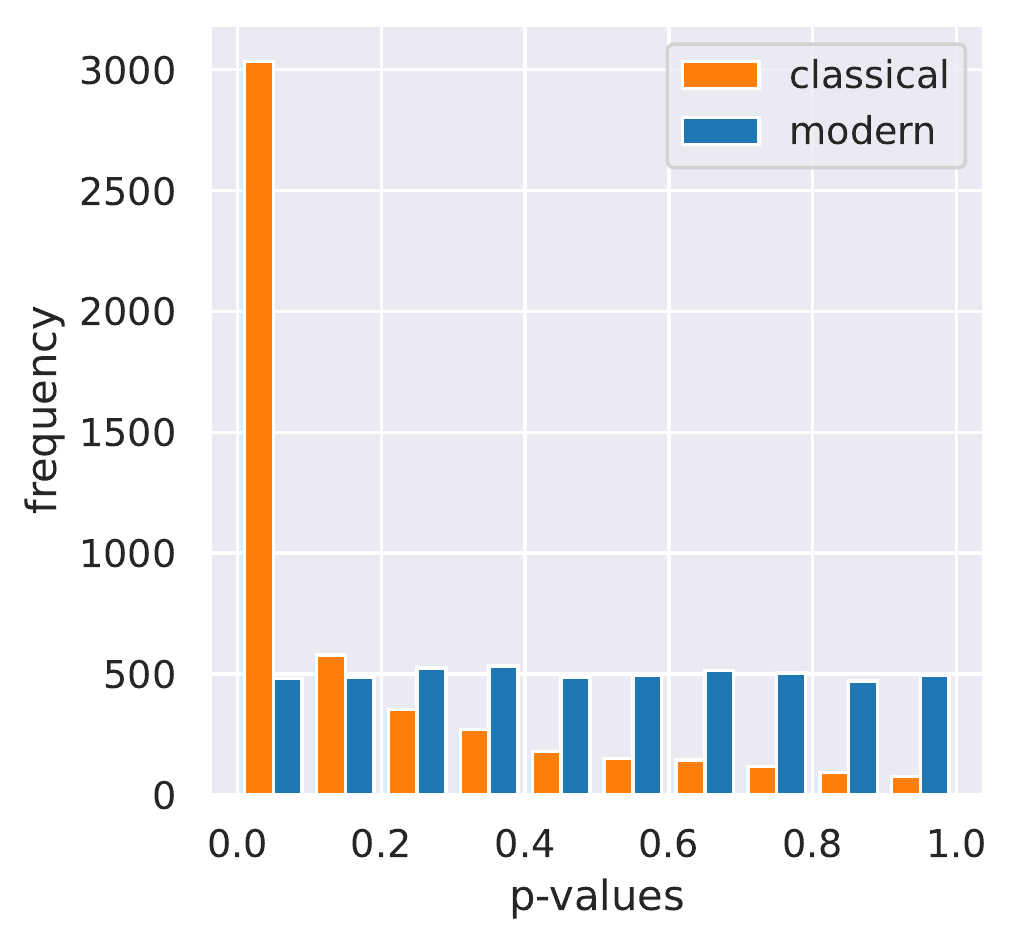}
        \caption{\footnotesize $(n,K)=(6000,4)$}
    \end{subfigure}
    \caption{\textbf{The upper row:} Q-Q plots of the test statistics from \eqref{eq:chi2-A} (in blue) and \eqref{eq:chi2-classical} (in orange) for different $(n,K)$ and $p=1000$ using $\hat\Omega_{jj}$.
    \textbf{The lower row:} histograms of p-values from classical test and our test for different $(n,K)$ and $p=1000$ using $\hat\Omega_{jj}$.}
    \label{fig:qqplots-estimated}
\end{figure}

\begin{figure}[H]
    \centering
    \begin{subfigure}[b]{0.32\textwidth}
        \centering
        \includegraphics[width=\textwidth]{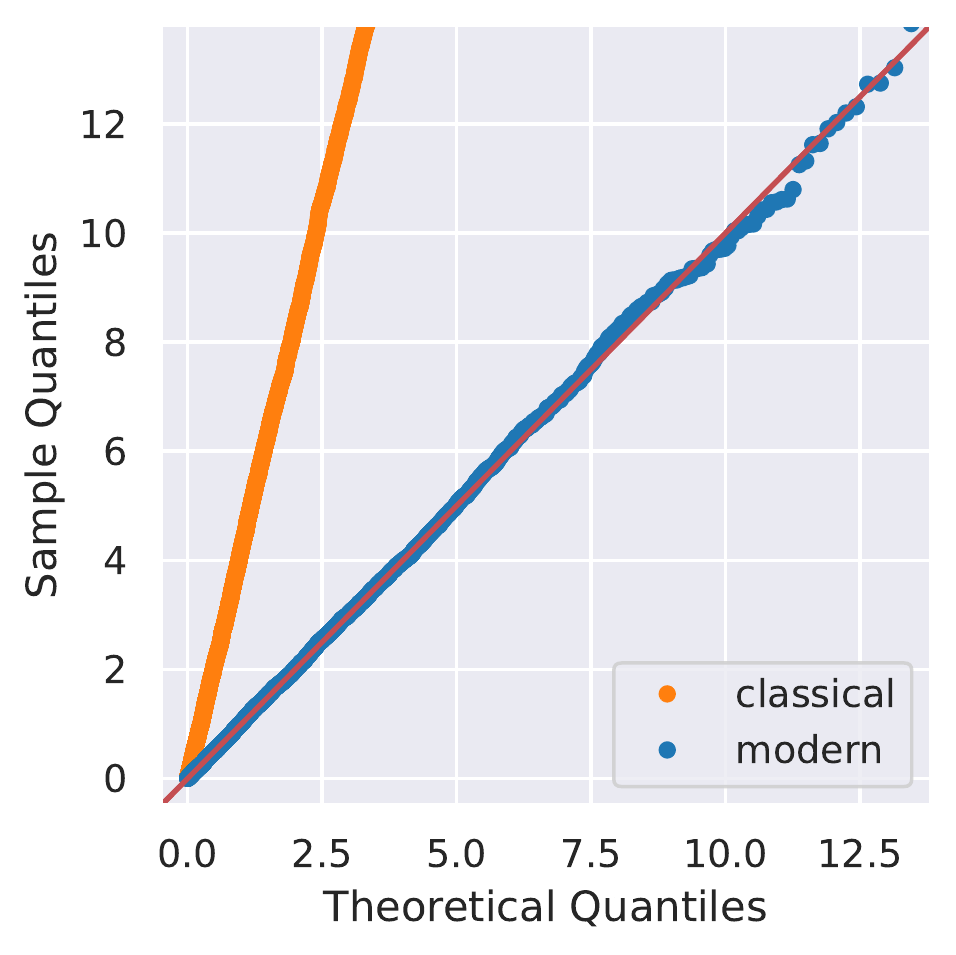}
        % \caption{\footnotesize $(n,K)=(4000,2)$}
    \end{subfigure}
    \hfill
    \begin{subfigure}[b]{0.32\textwidth}
        \centering
        \includegraphics[width=\textwidth]{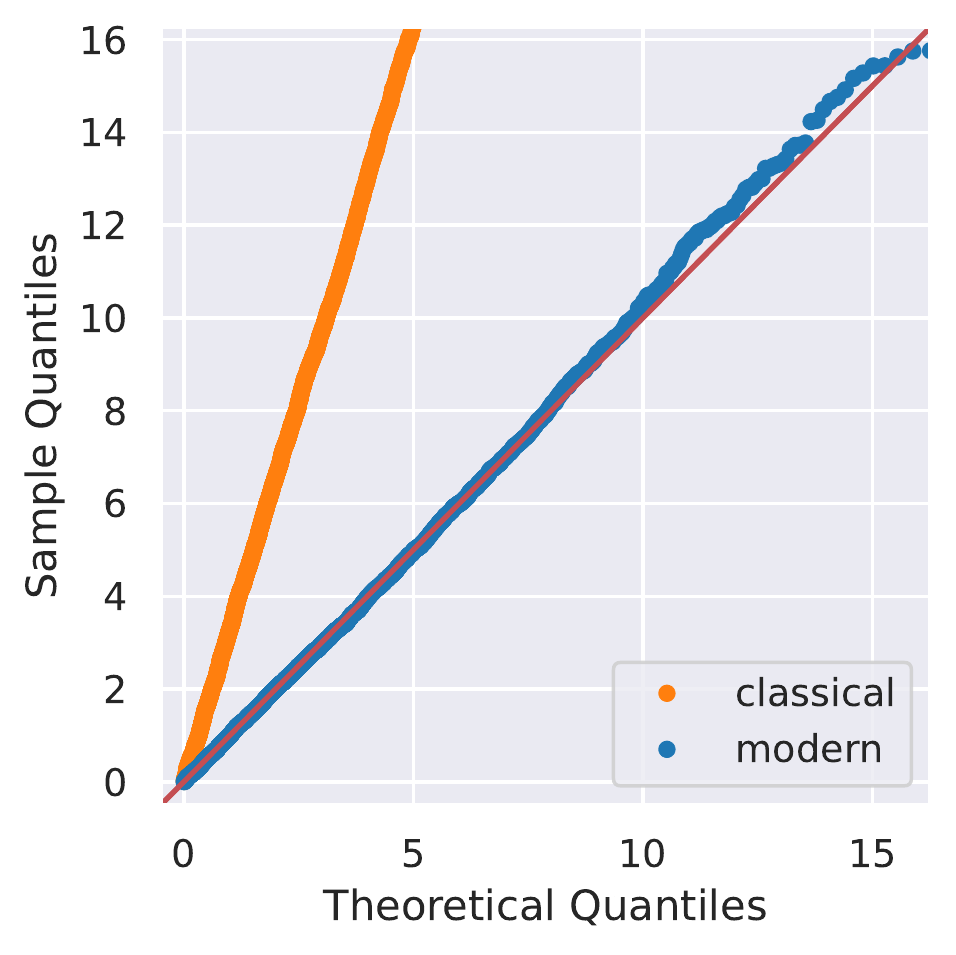}
        % \caption{\footnotesize $(n,K)=(5000,3)$}
    \end{subfigure}
    \hfill
    \begin{subfigure}[b]{0.32\textwidth}
        \centering
        \includegraphics[width=\textwidth]{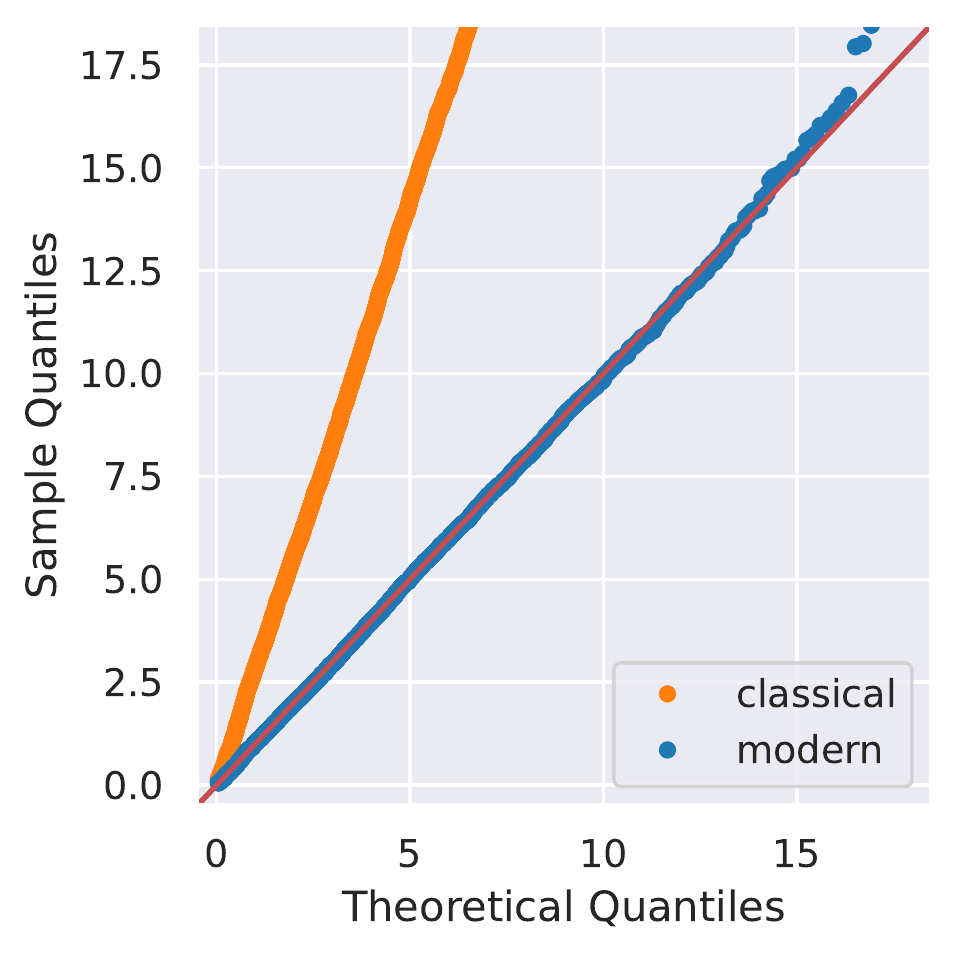}
        % \caption{\footnotesize $(n,K)=(6000,4)$}
    \end{subfigure}
    \\
    \begin{subfigure}[b]{0.32\textwidth}
        \centering
        \includegraphics[width=\textwidth]{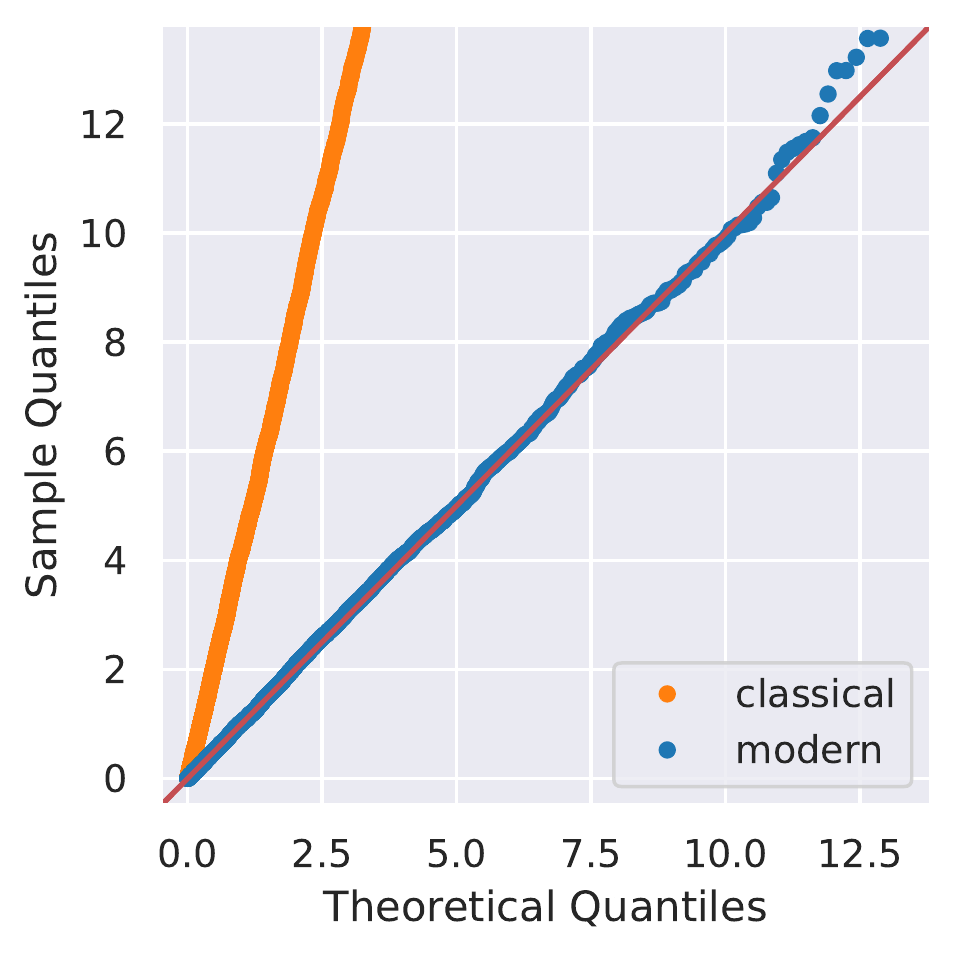}
        \caption{\footnotesize $(n,K)=(4000,2)$}
    \end{subfigure}
    \hfill
    \begin{subfigure}[b]{0.32\textwidth}
        \centering
        \includegraphics[width=\textwidth]{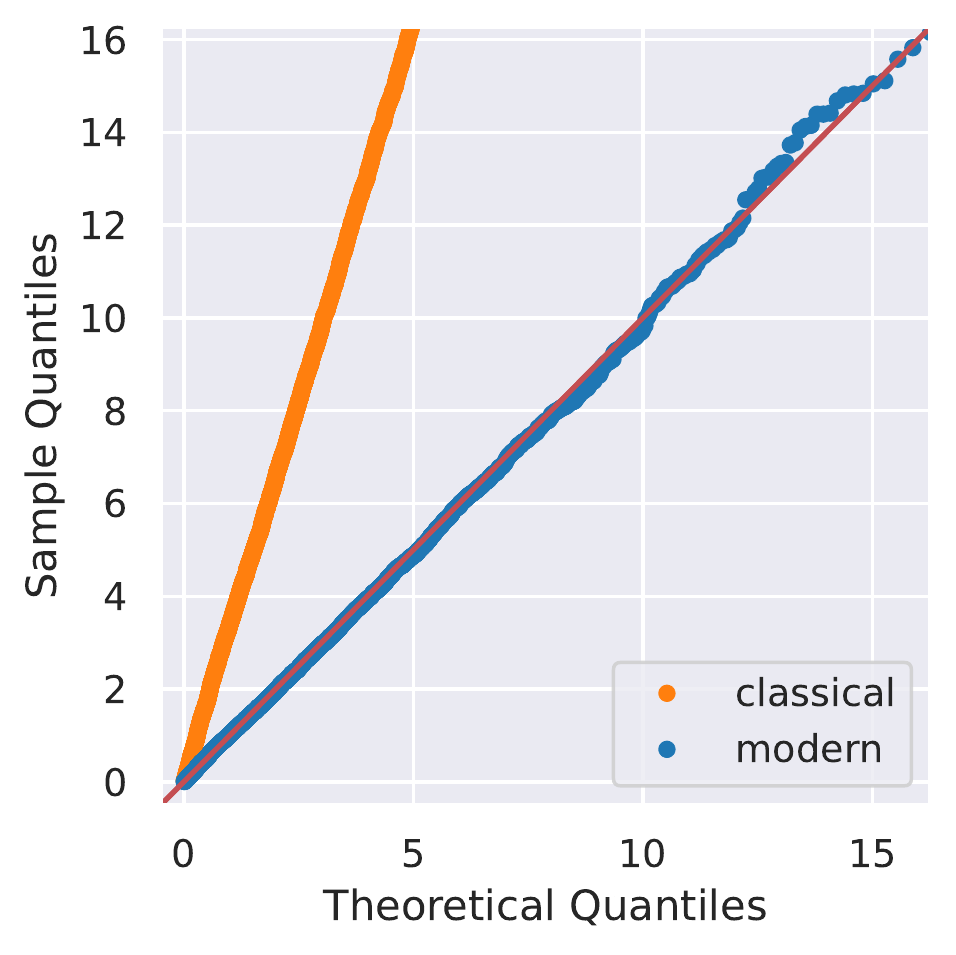}
        \caption{\footnotesize $(n,K)=(5000,3)$}
    \end{subfigure}
    \hfill
    \begin{subfigure}[b]{0.32\textwidth}
        \centering
        \includegraphics[width=\textwidth]{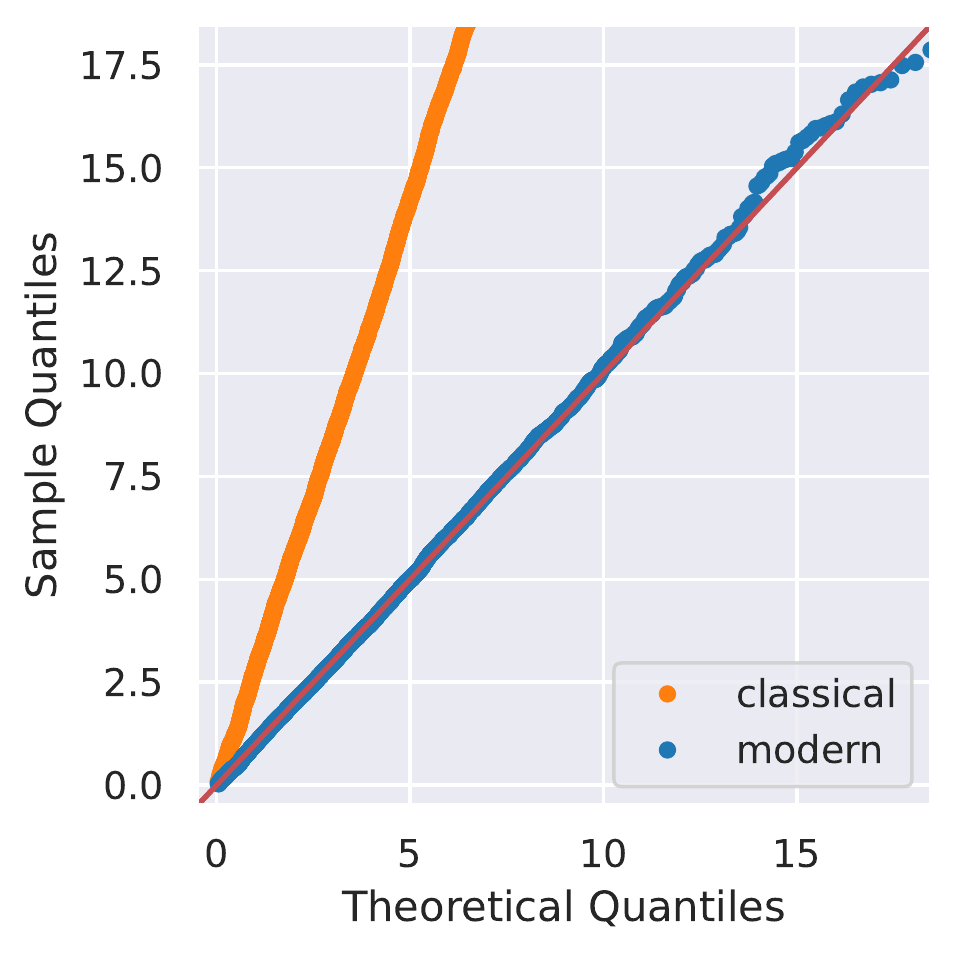}
        \caption{\footnotesize $(n,K)=(6000,4)$}
    \end{subfigure}
    \caption{
    Q-Q plots of the test statistics from \eqref{eq:chi2-A} (in blue) and \eqref{eq:chi2-classical} (in orange) for different $(n,K)$ and $p=1000$ using $\hat\Omega_{jj}$.
    \textbf{The upper row:} covariates are sampled from Rademacher distribution. 
    \textbf{The lower row:} covariates are sampled from distribution of SNPs.}
    \label{fig:qqplots-non-Gaussian}
\end{figure}

\begin{figure}[H]
    \centering
    \begin{subfigure}[b]{0.32\textwidth}
        \centering
        \includegraphics[width=\textwidth]{scatter_Gaussian_true_n2000_p600_q1.pdf}
        \caption{\footnotesize $q=1$}
    \end{subfigure}
    \hfill
    \begin{subfigure}[b]{0.32\textwidth}
        \centering
        \includegraphics[width=\textwidth]
        {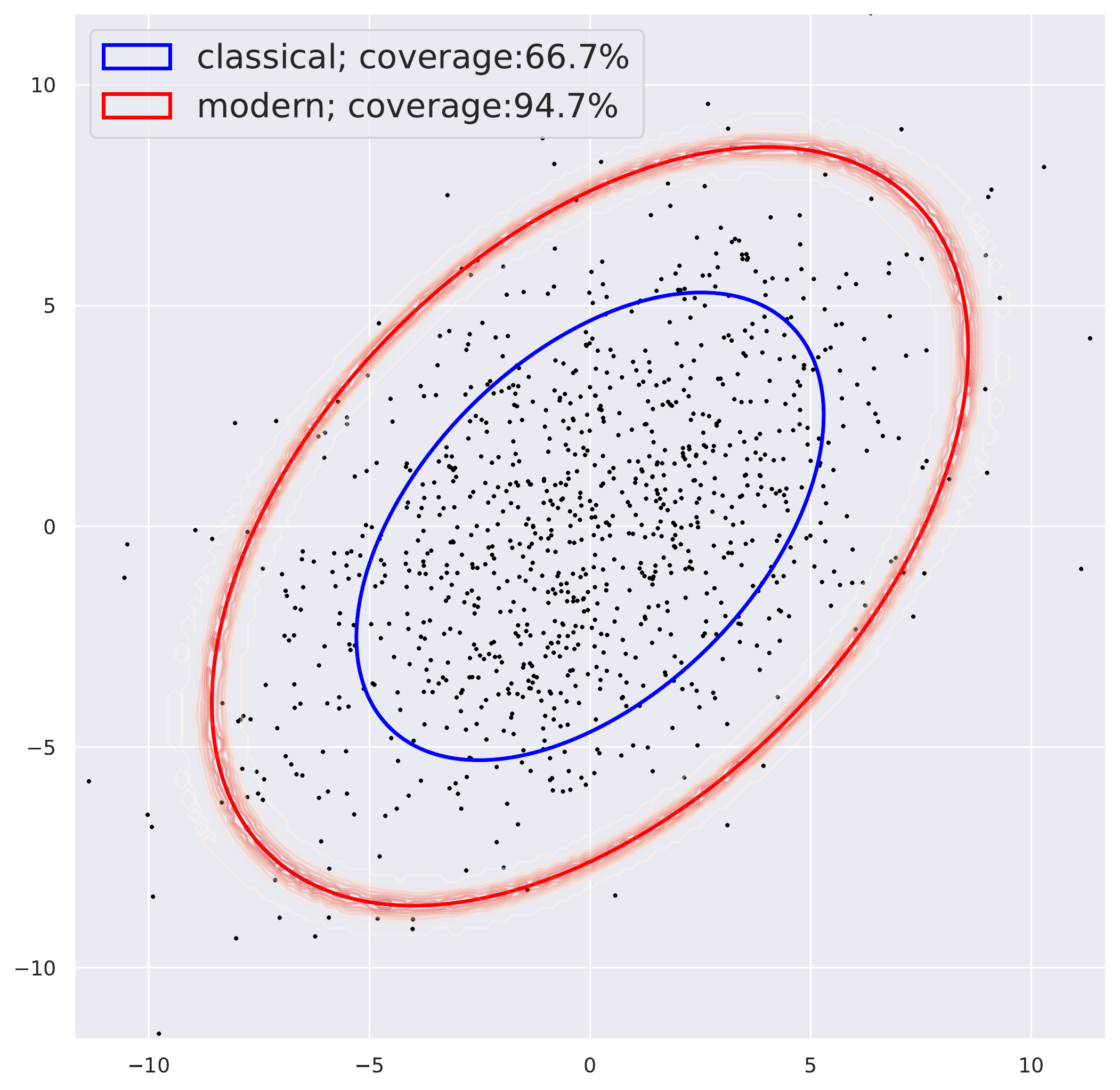}
        % {scatter_Gaussian_true_n2000_p600_q2.pdf}
        \caption{\footnotesize $q=2$}
        \end{subfigure}
    \hfill
    \begin{subfigure}[b]{0.32\textwidth}
        \centering
        \includegraphics[width=\textwidth]
        {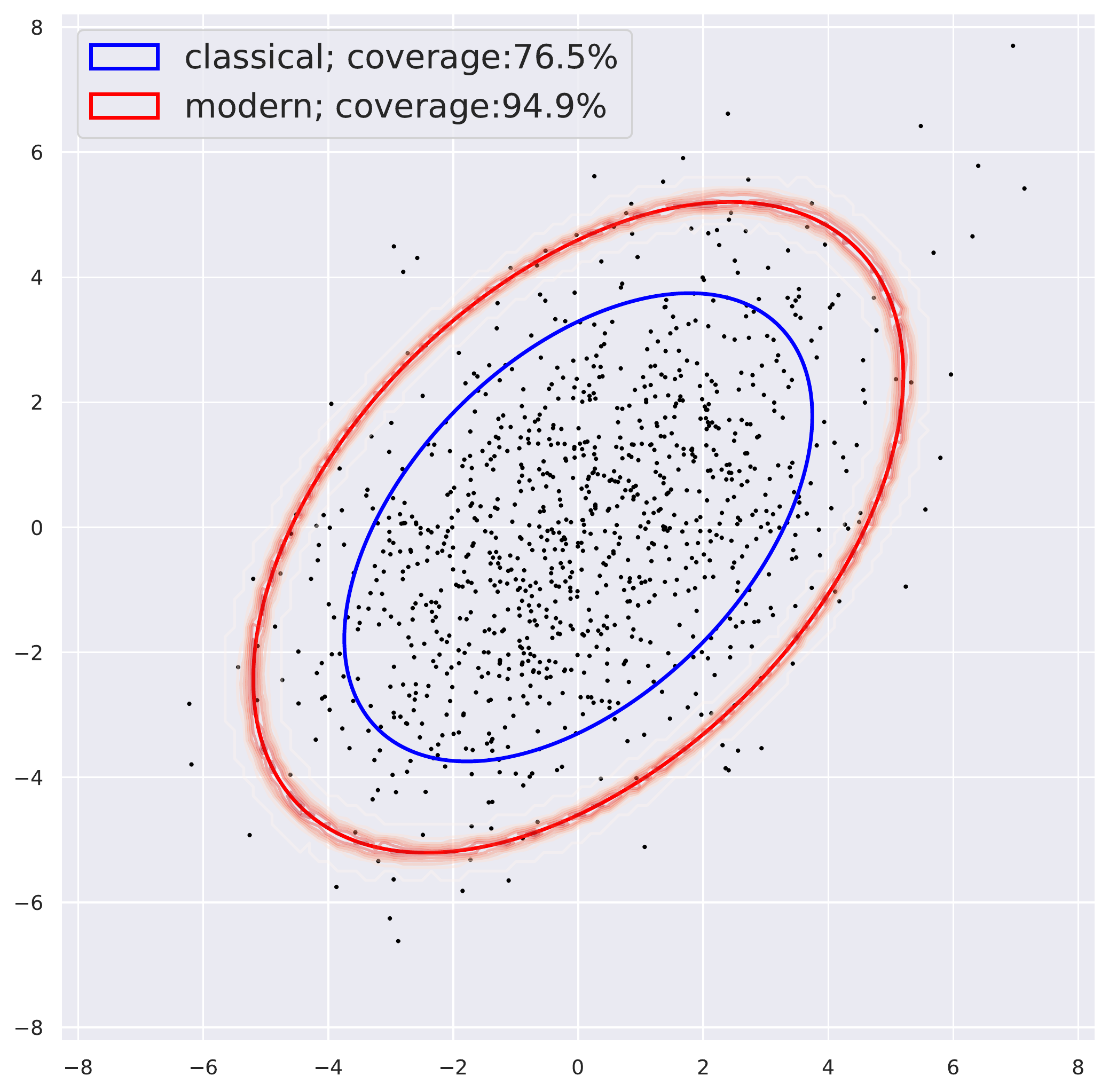}
        % {scatter_Gaussian_true_n2000_p600_q4.pdf}
        \caption{\footnotesize $q=4$}
    \end{subfigure}
    \caption{Scatter plot of pairs $(\sqrt{n}\hat A_{j1}, \sqrt{n}\hat A_{j2})$ with the same data generating process as in \Cref{fig:scatter-plot} (a) except using different $q$. 
    }
    % for different $q$, with the same data generating process as in \Cref{fig:scatter-plot} (a). }
    \label{fig:scatter_q}
\end{figure}

%\section{Move to appendix, not in main text}

\section{Proof of \Cref{thm:normal-K+1}}\label{proof:K+1}
Before proving \Cref{thm:normal-K+1}, we present another parametrization of the multinomial logistic regression model. 
The asymptotic theory of MLE for this new parametrized multinomial logistic model will be used to prove \Cref{thm:normal-K+1}. 
\subsection{Another parametrization of multinomial logistic regression}\label{sec:model-K}
Recall the symbol ``$\in$" is used in \eqref{barB} to emphasize that the minimizer $\hat\B$ in \eqref{barB} is not unique: if $\hat\B$ is a minimizer of \eqref{barB}
then $\hat\B - b\bm1_{K+1}^T$ is also a minimizer of \eqref{barB}, 
for any $b\in \R^p$ and the all-ones vector $\bm 1_{K+1}$ in $\R^{K+1}$.

Besides the log-odds model \eqref{eq:log-odds}, here we consider another identifiable parametrization of multinomial logistic regression, whose unknown parameter, denoted by $B^*$, is in $\R^{p\times K}$. 
\paragraph{Orthogonal complement.}
To obtain an identifiable multinomial logistic regression model from \eqref{model:over-specified}, 
we consider the symmetric constraint 
$\B^* \bm{1} = \sum_{k=1}^{K+1} \B^* e_k = \bm{0}$ as in \citep{zhu2004classification}, thus
$\B^* = \B^* (I_{K+1} - \frac{\bm 1\bm 1^T}{K+1})$, where $\bm 1$ is the all-ones vector in $\R^{K+1}$. 
Let $Q\in \R^{(K+1)\times K}$ be any matrix such that
\begin{equation}
    \label{Q}
I_{K+1} - \tfrac{1}{K+1} \bm 1 \bm 1^T
= QQ^T,
\qquad Q^TQ = I_K.
\end{equation}
We fix one choice of $Q$ satisfying \eqref{Q} throughout this supplement.
% where $\bm 1 = (1,...,1)^T$ is the all-ones vector in $\R^{K+1}$.
Let $B^* = \B^* Q$, then $\B^* = B^* Q^T$ and
the model \eqref{model:over-specified} can be parameterized using $B^*$ as 
\begin{equation}\label{model:B}
    \P(\y_{ik} = 1 | x_i) 
    = 
    \frac{\exp(x_i^T B^* Q e_k)}{\sum_{k'=1}^{K+1} \exp(x_i^T B^* Q e_{k'})}, \quad k\in\{1,2, \ldots, K+1\} . 
\end{equation}

The multinomial logistic MLE of $B^*$ in \eqref{model:B} is given by 
\begin{equation}\label{eq:mle}
    \textstyle
	\hat B = \argmin_{B\in \R^{p\times K}} \sum_{i=1}^n L_i (B^T x_i),
\end{equation} 
where 
$L_i: \R^K\to \R$ is defined by $L_i(u) = \L_i (Qu) \text{ for all } u\in \R^K.$
By this construction, 
we have $\hat B = \hat\B Q$ for any minimizer $\hat\B$ of \eqref{barB}.
Furthermore, by the chain rule using
the expressions \eqref{eq:g_bar-H_bar},
the gradient and Hessian of $L_i$ evaluated at $\hat B^T x_i$ are
\begin{equation}\label{eq:g,H}
    g_i 
    := \nabla L_i(\hat B^T x_i) 
    = Q^T\g_i, \qquad 
    H_i 
    := \nabla^2 L_i(\hat B^T x_i) 
    = Q^T \H_i Q. 
\end{equation}

Throughout, we use serif upright letters to denote quantities defined
on the unidentifiable parameter space $\R^{p\times(K+1)}$:
$$
\B^*, \hat\B\in\R^{p\times(K+1)},
\quad
\L_i:\R^{K+1}\to\R,
\qquad
\hat\p_i,\y_i,\g_i\in\R^{K+1},
\qquad
\H_i\in\R^{(K+1)\times (K+1)}
$$
and the normal italic font to denote analogous quantities
for the identifiable parameter space $\R^{p\times K}$:
$$B^*, \hat B\in\R^{p\times K},
\qquad
L_i:\R^K\to\R,
\qquad
g_i\in\R^K, 
\qquad
H_i\in\R^{K\times K}.$$

\Cref{thm:normal} provides the asymptotic normality and the chi-square approximation of null MLE coordinates in high-dimensions where $n,p\to\infty$ with the ratio $n/p$ converging to a finite limit. 
\begin{theorem}[Proof is given on page~\pageref{sec:application_multinomial}]
    \label{thm:normal}
    Let \Cref{assu:X,assu:Y,assu:MLE} be fulfilled. {Assume that either \Cref{assu:one_hot} or \Cref{assu:one_hot_q} holds.}
    Then for any $j\in[p]$
    such that $H_0$ in \eqref{H0} holds,
   
    \begin{equation}\label{eq:B-Normal}
        \sqrt{n}\Omega_{jj}^{-1/2} \Bigl(\frac 1n\sum_{i=1}^n g_i g_i^T\Bigr)^{-1/2} \Bigl(\frac1n\sum_{i=1}^n V_i \Bigr) \hat B^T e_j \limd N(0, I_K),
    \end{equation}
    where 
    $V_i = H_i - (H_i \otimes x_i^T) [\sum_{l=1}^n H_l \otimes (x_lx_l^T)]^{-1}
    (H_i  \otimes x_i)$.

    A direct consequence of \eqref{eq:B-Normal} is the $\chi^2$ result, 
    \begin{equation}
        \label{eq:B-chi2}
    \|\sqrt{n}\Omega_{jj}^{-1/2} \bigl(\frac 1n\sum_{i=1}^n g_i g_i^T\bigr)^{-1/2} \Bigl(\frac 1 n \sum_{i=1}^n V_i \Bigr) \hat B^T e_j \|^2 \limd 
    \chi^2_K. 
    \end{equation}
\end{theorem}
The proof of \Cref{thm:normal} is deferred to \Cref{sec:preliminary} and \Cref{sec:application_multinomial}. 
In the next subsection, 
we prove \Cref{thm:normal-K+1} using \Cref{thm:normal}. 

\subsection{Proof of \Cref{thm:normal-K+1}}
We restate \Cref{thm:normal-K+1} for convenience.
\mytheorem*

The proof of \Cref{thm:normal-K+1} is a consequence of \Cref{thm:normal}. 
To begin with, we state the following useful lemma. 
\begin{lemma}\label{lem:V-Vbar}
    For $\V_i$ and $V_i$ defined in \Cref{thm:normal,thm:normal-K+1}, we have 
     $V_i = Q^T \V_i Q$. 
\end{lemma}
\begin{proof}[Proof of \Cref{lem:V-Vbar}]
    Since $H_i = Q^T \H_i Q$, we have 
    \begin{align*}
        V_i =
        & H_i - (H_i \otimes x_i^T) \Bigl[\sum_{i=1}^nH_i\otimes (x_ix_i^T)\Bigr]^{-1} (H_i\otimes x_i) \\
        = & H_i - [(Q^T \H_i Q) \otimes x_i^T] \Bigl[\sum_{i=1}^n
        (Q^T\H_i Q) \otimes (x_ix_i^T)
        \Bigr]^{-1} [(Q^T \H_i Q)\otimes x_i]\\
        = & H_i - Q^T  (\H_i \otimes x_i^T) (Q \otimes I_p) \Bigl[(Q^T \otimes I_p) [\sum_{i=1}^n\H_i \otimes (x_ix_i^T)] (Q \otimes I_p) \Bigr]^{-1} (Q^T \otimes I_p)  (\H_i \otimes x_i) Q\\
        = & Q^T \H_i Q - Q^T (\H_i \otimes x_i^T) \Bigl[\sum_{i=1}^n(\H_i\otimes x_ix_i^T)\Bigr]^{\dagger} (\H_i\otimes x_i) Q\\
        =& Q^T \V_i Q, 
    \end{align*}
    where the penultimate equality is proved as follows. 

    Let $A = Q \otimes I_p$ and $\D = \sum_{i=1}^n\H_i \otimes (x_ix_i^T)$ 
    only in the remaining of this proof.  
    It remains to prove 
    \begin{equation}\label{eq:inverse-dagger}
    A [A^T \D A]^{-1} A^T
    = \D^{\dagger}. 
    \end{equation}
    Since $\H_i \bm1 = 0$, we have $\D (\bm 1 \otimes I_p) =0$. Since $Q^T \bm{1} = 0$ by definition of $Q$, we have 
    $A^T (\bm 1 \otimes I_p) = 0$. 
    If we write the eigen-decomposition of $\D$ as 
    $\D = \sum_{i=1}^{pK} \lambda_i u_i u_i^T$, then $u_i^T (\bm 1\otimes I_p) =0$.
    Hence, with $v_i = A^T u_i$, 
    \begin{align*}
        A^T \D A
        =~& \sum_{i=1}^{pK} \lambda_i v_i v_i^T.
    \end{align*}
    Since 
    $v_i^T v_{i'} 
    = u_i^T A A^T u_{i'} 
    % = u_i^T (QQ^T \otimes I_p) u_{i'} 
    = u_i^T [(I_{K+1} - \tfrac{\bm{1}\bm{1}^T}{K+1}) \otimes I_p] u_{i'} 
    = u_i^T u_{i'} 
    = I(i=i')$, we have 
    \begin{align*}
        A[A^T \D A]^{-1}A^T
        =A \bigl(\sum_{i=1}^{pK} \lambda_i^{-1} v_i v_i^T\bigr) A^T
        =\sum_{i=1}^{pK} \lambda_i^{-1} u_i u_i^T = \D^{\dagger},
    \end{align*}
    where the second equality uses $A v_i = A A^T u_i = u_i$. 
    The proof of \eqref{eq:inverse-dagger} is complete. 
\end{proof}

Now we are ready to prove that \Cref{thm:normal-K+1} is 
a consequence of \Cref{thm:normal}.
\begin{proof}[Proof of \Cref{thm:normal-K+1}]
    
By definition of $\g_i$ and $\V_i$, we have 
$\bm{1}^T \g_i = 0$ and 
$\bm{1}^T \V_i = \bm{0}^T$.
Thus, we have $QQ^T \g_i = \g_i$ and 
$QQ^T \V_i = \V_i$. 
Therefore, we can rewrite the left-hand side of \eqref{eq:normal_K+1} (without $\sqrt{n}\Omega_{jj}^{-1/2}$) as 
\begin{align*}
    &\Bigl(
        \Bigl(
            \frac1n\sum_{i=1}^n \g_i \g_i ^T
            \Bigr)^{1/2} \Bigr)^\dagger
            \Bigl(\frac1n\sum_{i=1}^n \V_i
            \Bigr) \hat\B^T e_j\\
    =~& \Bigl(
    \Bigl(
    \frac1n\sum_{i=1}^n QQ^T \g_i \g_i ^T QQ^T
    \Bigr)^{1/2} \Bigr)^\dagger
    \Bigl(\frac1n\sum_{i=1}^n QQ^T \V_i QQ^T\Bigr)\hat\B^T e_j\\
    =~& \Bigl(
        \Bigl(
        \frac1n\sum_{i=1}^n Q g_i g_i ^T Q^T
        \Bigr)^{1/2} \Bigr)^\dagger
        \Bigl(\frac1n\sum_{i=1}^n Q V_i Q ^T\Bigr)\hat\B^T e_j\\
    =~& Q\Bigl(
        \Bigl(
        \frac1n\sum_{i=1}^n g_i g_i ^T 
        \Bigr)^{1/2} \Bigr)^\dagger Q^T Q
        \Bigl(\frac1n\sum_{i=1}^n  V_i Q ^T\Bigr)\hat\B^T e_j\\
    =~& Q\Bigl(
        \Bigl(
        \frac1n\sum_{i=1}^n g_i g_i ^T 
        \Bigr)^{1/2} \Bigr)^\dagger 
        \Bigl(\frac1n\sum_{i=1}^n  V_i \Bigr)\hat B^T e_j,
\end{align*}
where the first equality uses $QQ^T \g_i = \g_i$ and $QQ^T \V_i = \V_i$,
the second equality uses 
$g_i = Q^T \g_i$ and $V_i = Q^T \V_i Q$ from \Cref{lem:V-Vbar}, 
the third equality follows from the same argument of \eqref{eq:inverse-dagger},
and the last equality uses $Q^T Q = I_K$ and 
$\hat B = \hat \B Q$. 

Therefore, \Cref{thm:normal} implies that the limiting covariance for the left-hand side of \eqref{eq:normal_K+1}
is $QQ^T = I_K - \tfrac{\bm{1}\bm{1}^T}{K+1}$. 
This completes the proof. 
\end{proof}

\section{Proof of \Cref{thm:normal-A}}\label{proof:A}
We restate \Cref{thm:normal-A} for convenience. 
\mytheoremA*

The proof is a direct consequence of \Cref{thm:normal-K+1}. 
\begin{proof}[Proof of \Cref{thm:normal-A}]
    By definition of $\hat A$ in \eqref{eq:relationship_A_B}, we have 
    $\hat A = \hat\B (I_K, -\bm{1}_K)^T$ and 
    $$
    % [\hat A, 0_K] 
    \hat A ~ (I_K, \bm{0}_K)
    =\hat\B (I_K, -\bm{1}_K)^T (I_K, \bm{0}_K)
    =\hat\B (I_{K+1} - e_{K+1}\bm{1}^T)
    =\hat\B - (\hat\B e_{K+1})\bm{1}^T,
    $$
    which is of the form $\hat \B - b\bm{1}^T$ with $b = \hat\B e_{K+1}$. 
    Therefore, $\hat A ~ (I_K, \bm{0}_K)$ is also a solution of \eqref{barB}. 
    Taking $\hat \B$ 
    in \Cref{thm:normal-K+1} to be 
    $\hat A ~ (I_K, \bm{0}_K) = \hat A R^T$ gives the desired $\chi^2$ result \eqref{eq:chi2-A} and 
    \begin{equation} 
    \sqrt{n}\Omega_{jj}^{-1/2}
    \Bigl(
    \Bigl(
    \frac1n\sum_{i=1}^n \g_i \g_i^T 
    \Bigr)^{1/2} \Bigr)^\dagger
    \Bigl(\frac1n\sum_{i=1}^n \V_i\Bigr)
    R \hat A^T e_j
    \limd  N\Bigl(0, I_{K+1}-\tfrac{\bm{1}\bm{1}^T}{K+1}\Bigr). 
\end{equation}

Multiplying $(R^T (I_{K+1}-\tfrac{\bm{1}\bm{1}^T}{K+1}) R)^{-1/2} R^T$ to 
the left of the above display gives 
the desired normality result \eqref{normal-A} by observing
$(R^T (I_{K+1}-\tfrac{\bm{1}\bm{1}^T}{K+1}) R)^{-1/2} = (I_K + \frac{\bm{1}_K\bm{1}_K^T}{\sqrt{K+1} + 1}). 
$
This completes the proof. 
\end{proof}

\section{Preliminary results for proving \Cref{thm:normal}}
\label{sec:preliminary}
\subsection{Results for general loss functions}\label{sec:general-loss}
In this subsection, we will work under the following assumptions 
with a general convex loss function.
Later in \Cref{sec:application_multinomial}, we will apply 
the general results of this subsection to the multinomial logistic loss
discussed in \Cref{sec:model-K}.

\begin{assumption}
    \label{assum:general_model}
    Suppose we have data $(Y, X)$, where 
    $Y\in\R^{n\times (K+1)}$ with rows $(\y_1,...,\y_n)$, and
    $X\in\R^{n\times p}$ has \iid rows $(x_1,...,x_n)$ with $x_i\sim N(\bm{0},\Sigma)$ and invertible $\Sigma$. 
    The observations 
    $(\y_i, x_i)_{i\in[n]}$
    are \iid and $\y_i$ has the form $\y_i=f(U_i,x_i^T B^*)$ for some deterministic
    function $f$, deterministic $B^* \in\R^{p\times K}$,
    and latent random variable $U_i$ independent of $x_i$.
    Assume $p/n\le \delta^{-1}<1$.
\end{assumption}

\begin{assumption}
    \label{assum:general_loss}
    Given data $(Y, X)$, 
    consider twice continuously differentiable
    and strictly convex loss functions $(L_i)_{i\in[n]}$
    with each $L_i:\R^K\to\R$ depending on $\y_i$ but not
    on $x_i$.
\end{assumption}
    Provided that the following minimization problem admits a solution,
    define
    \begin{align*}
        \hat B(Y,X) = \argmin_{B\in \R^{p\times K}} \sum_{i=1}^n L_i (B^T x_i).
    \end{align*}
    Define for each $i\in[n]$, 
    \begin{align*}
        g_i(Y,X) = \nabla L_i(\hat B(Y,X)^T x_i), \quad
        H_i(Y,X) = \nabla^2 L_i (\hat B(Y,X)^T x_i),
    \end{align*}
    so that $g_i(Y,X)\in\R^K$ and $H_i(Y,X)\in\R^{K\times K}$.
    Define 
    \begin{align*}
        G(Y,X)&=\sum_{i=1}^n e_i g_i(Y,X)^T,\\
    V(Y, X)&= \sum_{i=1}^n \Bigl(
    H_i(Y,X) - 
    (H_i(Y,X) \otimes x_i^T)
    \Bigl[\sum_{l=1}^n H_l(Y,X) \otimes (x_lx_l^T)\Bigr]^{\dagger}
    (H_i(Y,X)\otimes x_i)\Bigr), 
    \end{align*}
    so that $G(Y,X)\in\R^{n\times K}$ and $V(Y,X)\in\R^{K\times K}$.
If the dependence on data $(Y,X)$ is clear from context, we will simply write $\hat B$, $g_i$, $H_i$
$G$, and $V$. 
\begin{theorem}
    \label{thm:general_Sigma}
    Let \Cref{assum:general_loss,assum:general_model} be fulfilled.
    Let $c_*,m_*,m^*,K$ be positive constants independent of $n,p$.
    Let $U^*\subset\R^{p\times (K+1)} \times \R^{n\times p}$ be an open set
    satisfying
    \begin{itemize}
        \item[(1)]
        If $\{(Y,X)\in U^*\}$, then the minimizer $\hat B(Y,X)$ in \Cref{assum:general_loss} exists, 
        $H_i \preceq I_K$ for each $i\in [n]$,
        $
        \frac 1 n \sum_{i=1}^n H_i(Y,X) \otimes (x_i x_i^T) \succeq
        c_* (I_K\otimes \Sigma)$
        and
        $m_* I_K\preceq \frac 1 n G(Y,X)^TG(Y,X) \preceq m^* I_K$. 

        \item[(2)]
            For any $\{(Y,X),(Y,\tilde X)\}\subset U^*$, 
            $\|G(Y,X) - G(Y,\tilde X)\|_F\le L \|(X-\tilde X)\Sigma^{-1/2}\|_F$ holds for some positive constant $L$.
    \end{itemize}
    Then for any $j\in[p]$ such that $e_j^TB^*=\bm{0}^T_K$, there exists
    a random variable $\xi\in\R^K$ such that
    $$\E\bigl[
    I\{(Y,X)\in U^*\}
    \bigl\|
    \tfrac{(G^TG)^{-1/2}V\hat B^Te_j}{\sqrt{\Omega_{jj}}}
    - 
    \xi
    \bigr\|^2
    \bigr]
    \le \tfrac{C}{p-K},
    $$
    and
    $\P(\|\xi\|^2>\chi^2_K(\alpha))\le\alpha$ for all $\alpha\in(0,1)$, C is a positive constant depending on $(c_*, m_*, m^*, K, L)$ only. 
    If additionally $\P((Y,X)\in U^*)\to 1$, then $\xi$ in the previous display satisfies $\xi\limd N(\bm{0},I_K)$ and 
    $$
    \frac{(G^TG)^{-1/2}V \hat B^T e_j}{\sqrt{\Omega_{jj}}} \limd N(\bm{0}, I_K). 
    $$
 \end{theorem}
The proof of \Cref{thm:general_Sigma} is given in next subsection. 
\subsection{Proof of \Cref{thm:general_Sigma}}
In this subsection and next subsection, we will slightly abuse the notations $A^*$ and $\hat A$, which have different definitions than the definitions in the main text.

Let $\Sigma^{1/2} B^* = \sum_{k=1}^K s_k u_k  v_k^T$ be the singular value decomposition of $\Sigma^{1/2} B^*$, where $u_1, ..., u_K$ are the left singular vectors and $v_1,...,v_k$ the right singular vectors.
If $\Sigma^{1/2}B^*$ is of rank strictly less than $K$, we allow some $s_k$ to
be equal to 0 so that $\Sigma^{1/2} B^* = \sum_{k=1}^K s_k u_k  v_k^T$ still holds with orthonormal $(u_1,...,u_K)$ and orthonormal $(v_1,...,v_K)$.
We consider an orthogonal matrix $\tilde P\in \R^{p\times p}$ such that 
\begin{equation}\label{eq:rotation-P}
    \tilde P \tilde P^T = \tilde P^T \tilde P = I_p, \quad 
    \tilde P \frac{\Sigma^{-1/2} e_j}{\|\Sigma^{-1/2} e_j\|} = e_1, \quad 
    \tilde P u_k = e_{p-K+k}, \quad \forall k\in[K].
\end{equation}
Since $e_j^T B^* =\bm{0}^T$ implies $e_j^T \Sigma^{-1/2} u_k= 0$, 
we can always find a matrix $\tilde P$ satisfying \eqref{eq:rotation-P}. 
From now on we fix this matrix $\tilde P$ and consider the following change of variable, 
\begin{equation}\label{eq:rotation}
Z = X\Sigma^{-1/2} \tilde P^T, \qquad A^* = \tilde P \Sigma^{1/2} B^*.
\end{equation}
It immediately follows that $Z$ has \iid $N(0, 1)$ entries 
and the first $p-K$ rows of $A^*$ are all zeros.
Since the response $\y_i$ has the expression
$\y_i = f(U_i, x_i^T B^*)$, 
$Y$
is unchanged by the change of variable \eqref{eq:rotation} from $Z A^*=X B^*$.
We now work on the 
multinomial logistic estimation with data $(Y, Z)$ and the underlying coefficient matrix $A^*$ in \eqref{eq:rotation}. 
Parallel to the estimate $\hat B$  of $B^*$ in 
\Cref{assum:general_loss}, 
we define the estimate of $A^*$ using data $(Y, Z)$ as 
\begin{align*}
    \hat A (Y, Z) = \argmin_{A\in \R^{p\times K}} \sum_i L_i (A^T z_i),
\end{align*}
where $z_i = Z^T e_i$ is the $i$-th row of $Z$. 
By construction, we have $\hat A = \tilde P\Sigma^{1/2}\hat B$, hence $Z \hat A = X\hat B$ and $e_1^T \hat A = e_j^T\hat B /\sqrt{\Omega_{jj}}$. Furthermore, the quantities depending on $(Y, X\hat B)$ remain unchanged after the change of variable. In particular, the gradient and Hessian
\begin{align*}
    \nabla L_i (\hat B^T x_i) = \nabla L_i (\hat A^T z_i),
    \qquad 
    \nabla^2 L_i (\hat B^T x_i) = \nabla^2 L_i (\hat A^T z_i)
\end{align*}
are unchanged. It follows that the matrix $G$ and $V$ are unchanged.
Therefore, we have 
\begin{align*}
\frac{e_{j}^T \hat B V (G^T G)^{-1/2}}{\sqrt{\Omega_{jj}}}
= e_{1}^T \hat A V(G^T G)^{-1/2}.
\end{align*}
In conclusion, with the change of variables \eqref{eq:rotation}, 
we only need to prove 
\Cref{thm:general_Sigma} in the special case, where the design matrix $X$ \iid $N(0, 1)$ entries and the response $Y$ is independent of the first $p-K$ columns of $X$. 
To this end, we introduce the following \Cref{thm:general_identity}, and the 
proof of \Cref{thm:general_Sigma} is a consequence of \Cref{thm:general_identity}
as it proves the desired
result for $e_{1}^T \hat A V (G^TG)^{-1/2}$.

\begin{theorem}\label{thm:general_identity}
    Let $c_*,m_*,m^*,K$ be constants independent of $n,p$.
    Let $Z\in\R^{n\times p}$ have \iid rows $(z_1,...,z_n)$ with $z_i\sim N(\bm{0},I_p)$.
    Let $\y_1,...,\y_n\in\R^{(K+1)}$ such that $(\y_1,...,\y_n)$
    is independent of the first $p-K$ columns of $Z$.
    Consider twice continuously differentiable
    and strictly convex loss functions $(L_i)_{i=1,...,n}$
    with each $L_i:\R^K\to\R$ depending on $\y_i$ but not on $z_i$
    and define, provided that
    the minimizer admits a solution,
    \begin{align*}
        \hat A(Y,Z) = \argmin_{A\in \R^{p\times K}} \sum_{i=1}^n L_i (A^T z_i),
        \quad
        g_i(Y,Z) = \nabla L_i(\hat A(Y,Z)^T z_i),
        \quad
        H_i(Y,Z) = \nabla^2 L_i (\hat A(Y,Z)^T z_i),
    \end{align*}
    $G(Y,Z)=\sum_{i=1}^n e_i g_i(Y,Z)^T\in\R^{n\times K}$, and $V(Y, Z) = \sum_{i=1}^n \bigl(
    H_i - 
    (H_i \otimes z_i^T)
    \bigl[\sum_{l=1}^n H_l \otimes (z_lz_l^T)\bigr]^{\dagger}
    (H_i\otimes z_i)\bigr) \in \R^{K\times K}$, where we dropped the dependence of $H_i$ on $(Y, Z)$ for simplicity. 

    Let $O\subset\R^{n\times (K+1)} \times \R^{n\times p}$ be an open set
    satisfying
    \begin{itemize}
        \item 
        If 
        $(Y,Z)\in O$, then the minimizer $\hat A(Y,Z)$ exists, 
        $H_i \preceq I_K$ for each $i\in [n]$,
        $c_* I_{pK} \preceq\frac 1 n \sum_{i=1}^n H_i(Y,Z) \otimes(z_iz_i^T)$, 
        and
        $m_* I_K\preceq \frac 1 n \sum_{i=1}^n G(Y, Z)^T G(Y, Z) \preceq m^* I_K$. 
        \item
            With the notation
            $G(Y,Z)=\sum_{i=1}^n e_i g_i(Y,Z)^T$, we have  
            if two $Z,\tilde Z\in\R^{n\times p}$ satisfy
            $\{(Y,Z),(Y,\tilde Z)\}\subset O$ then
            $\|G(Y,Z) - G(Y,\tilde Z)\|\le L \|Z-\tilde Z\|$.
    \end{itemize}
    For $e_1\in\R^p$ the first canonical basis vector, there exists
    a random variable $\xi\in\R^K$ such that
    $$\E\bigl[
    I\{(Y,Z)\in O\}
    \bigl\|
    (G^TG)^{-1/2}V\hat A^Te_1
    - 
    \xi
    \bigr\|^2
    \bigr]
    \le \tfrac{C}{p-K},
    $$
    and
    $\P(\|\xi\|^2>\chi^2_K(\alpha))\le\alpha$ for all $\alpha\in(0,1)$, C is a positive constant depending on $(c_*, m_*, m^*, K, L)$ only. 
    If additionally $\P((Y,Z)\in O)\to 1$, then $\xi$ in the previous display satisfies $\xi\limd N(\bm{0},I_K)$ and 
    $$e_{1}^T \hat A V (G^TG)^{-1/2} \limd N(\bm{0}, I_K).$$
\end{theorem}
The proof of \Cref{thm:general_identity} is presented in \Cref{sec:proof-A}. 

\subsection{Proof of \Cref{thm:general_identity}}\label{sec:proof-A}
We first present a few useful lemmas, whose proofs are given at the end of this subsection. 

% \section{A useful lemma for asymptotic normality result}
\begin{lemma}[Proof is given on \cpageref{proof_stein_chi2}]\label{lem:chi2}
    Let $z\sim N(\bm{0},\sigma^2 I_n)$ and $F:\R^n \to \R^{n\times K}$ be weakly differentiable with $\E\|F(z)\|_F^2 < \infty$.
    Let $\tilde z$ be an independent copy of $z$. Then
    $$
    \E\Bigl[
    \Bigl\|z^T F (z) - \sigma^2\sum_{i=1}^n \frac{\partial e_i^T F(z)}{\partial z_i}
    - z^T F(\tilde z)
    \Bigr\|^2
    \Bigr]
    \le
    3 \sigma^4
    \E
    \sum_{i=1}^n
    \Big\|
    \frac{\partial F(z)}{\partial z_i}
    \Big\|_{F}^2.
    $$
\end{lemma}

\begin{lemma}[Proof is given on \cpageref{proof_lipschitz_G}]\label{lem:lipschitz-G}
    If 
    $G,\tilde G \in \R^{n\times K}$ satisfy
    $m_* I_K \preceq \frac 1 n G^TG \preceq m^* I_K$
    and
    $m_* I_K \preceq \frac 1 n \tilde G^T\tilde G \preceq m^* I_K$ for some positive constants $m_*$ and $m^*$. Then
    \begin{align*}
        \|(G^TG)^{-1/2} - (\tilde G^T\tilde G)^{-1/2}\|_F &\le L_1 n^{-1}\|G-\tilde G\|_F,
        \\
        \|G(G^TG)^{-1/2} - \tilde G(\tilde G^T\tilde G)^{-1/2}\|_F &\le L_2 n^{-1/2} \|G-\tilde G\|_F,
    \end{align*}
    where $L_1,L_2$ are positive constants depending on $(K,m_*,m^*)$ only.
\end{lemma}

\begin{lemma}[Proof is given on \cpageref{proof_dot_g}]\label{lem:dot-g}
    Let the assumptions in \Cref{thm:general_identity} be fulfilled. 
    Let $Y\in\R^{n\times (K+1)}$ be fixed. 
    If a minimizer
    $\hat A(Y,Z)$ exists at $Z$, then $Z\mapsto \hat A(Y,Z)$ exists and is differentiable in a neighborhood of $Z$
    with derivative
    \begin{align*}
        &\pdv{\vect(\hat A)}{z_{ij}} 
        = - M [g_i \otimes e_j + (H_i\hat A^Te_j \otimes z_i)],\\
        &\pdv{g_l}{z_{ij}} 
        = - (H_l \otimes z_l^T) M [g_i \otimes e_j + (H_i\hat A^Te_j \otimes z_i)] + I(l=i) H_l\hat A^T e_j,
    \end{align*}
    where $M = [\sum_{i=1}^n H_i \otimes (z_iz_i^T)]^{-1}$. 
    It immediately follows that 
    \begin{equation*}
    \pdv{g_i}{z_{ij}} 
    = [H_i - (H_i \otimes z_i^T) M (H_i  \otimes z_i)] \hat A^T e_j - (H_i \otimes z_i^T) M (g_i \otimes e_j). 
    \end{equation*}
\end{lemma}

\begin{corollary}[Proof is given on \cpageref{proof_dot_G_normalized}]
    \label{cor:dot-G-nomrlized}
    Under the same conditions of \Cref{lem:dot-g}, 
    for $G= \sum_{i=1}^n e_i g_i^T$, we have 
    for each $i\in[n], j\in[p]$, 
    \begin{align*}
    &\sum_{i=1}^n \pdv{e_i^T G(G^TG)^{-1/2}}{z_{ij}} \\
    =~& e_j^T\hat A V_i^T (G^TG)^{-1/2} 
    + 
    \sum_{i=1}^n \Bigl[-(g_i^T \otimes e_j^T)M(H_i\otimes z_i)(G^TG)^{-1/2} 
    + e_i^T G \pdv{(G^TG)^{-1/2}}{z_{ij}}\Bigr]. 
    \end{align*}
\end{corollary}

Now we are ready to prove \Cref{thm:general_identity}. 
\begin{proof}[Proof of \Cref{thm:general_identity}]
    Let $h: O \to \R^{n\times K}$ be $h(Y, Z) = G(Y, Z)(G(Y, Z)^T G(Y, Z))^{-1/2}$.
    In most of this proof, we will omit the dependence $(Y,Z)$
    on $h, \hat A, g_i, H_i, G, V$ to lighten notation.
    % We will omit the dependence $h$ on $(Y,Z)$ to lighten notation. 
    By \Cref{lem:lipschitz-G}, we know this $h$ is $L L_2 n^{-1/2}$-Lipschitz in the sense that 
    $
    \|h(Y, Z) - h(Y, \tilde Z)\|_{F} \le L L_2 n^{-1/2} \|Z - \tilde Z\|_{F}
    $
    for all $\{(Y, Z), (Y, \tilde Z)\}\subset O$. 
    By Kirszbraun theorem, there exists a function $H: \R^{n\times (K+1)} \times \R^{n\times p}\to \R^{n\times K}$ (an extension of $h$ from $O$ to $\R^{n\times (K+1)} \times \R^{n\times p}$) such that $H(Y, Z) = h(Y, Z)$ for all $(Y, Z)\in O$, 
    $\|H(Y, Z)\|_{op}\le 1$ and 
    $\|H(Y, Z) - H(Y, \tilde Z)\|_{F} \le L L_2 n^{-1/2} \|Z - \tilde Z\|_{F}
    $
    for all $\{(Y, Z), (Y, \tilde Z)\}\subset \R^{n\times (K+1)} \times \R^{n\times p}$.

    For each $j\in[p]$, let $\z_j = Z e_j$ be the $j$-th column of $Z$ to distinguish it from the notation $z_i$, which means the $i$-th row of $Z$. 
    Let $\check \z\sim N(\bm{0}, I_n)$ be an independent copy of each columns of $Z$, 
    and
    $\check Z^j = Z(I_p-e_je_j^T) + \check \z e_j^T$. 
    That is, $\check Z^j$ replaces the $j$-th column of $Z$ by $\check \z$. 
    By definition, $\z_1\perp \check \z$ and $\z_1 \perp \check Z^1$. 
    
    Let $\xi = -[H(Y, \check Z^1)]^T \z_1 \in \R^{K}$, then 
    $\|\xi\|^2 \le \|\z_1\|$ since $\|H(Y, \check Z^1)\|_{op}\le 1$.
    It follows that 
    $$\P(\|\xi\|^2 > \chi^2_K(\alpha)) 
    \le 
    \P(\|\z_1\|^2 > \chi^2_K(\alpha)) = \alpha. 
    $$
    Note that the first $p - K$ columns of $Z$ are exchangeable,  because they are \iid and independent of the response $Y$,  
    we have for each $\ell\in[p-K]$,
    \begin{align*}
        \textstyle
    &\E\bigl[
    I\{(Y,Z)\in O\}
    \bigl\|
    (G^TG)^{-1/2}V\hat A^Te_1 
    - 
    \xi
    \bigr\|^2
    \bigr]\\
    =~&\E \Bigl[ 
        I\{(Y,Z)\in O\}
    \|
    e_{1}^T \hat A V (G^TG)^{-1/2}
    +\z_1^T H(Y, \check Z^1)
    \|^2
    \Bigr]\\
    =~& \E \Bigl[ 
        I\{(Y,Z)\in O\}
            \|
            e_{\ell}^T \hat A V (G^TG)^{-1/2}
            +\z_{\ell}^T H(Y, \check Z^{\ell})
            \|^2
        \Bigr],
    \end{align*}
    where the last line holds for any $\ell\in[p-K]$ because 
    $(\z_1, e_1^T \hat A, \check Z) \stackrel{d}{=} (\z_{\ell}, e_{\ell}^T \hat A, \check Z^{\ell})$. 
    Therefore,
    \begin{align*}
        &\E \Bigl[ 
            I\{(Y,Z)\in O\}
            \Bigl\|
                e_{1}^T \hat A V (G^TG)^{-1/2}
                + \z_1^T H(Y, \check Z^1)
            \Bigr\|^2
        \Bigr]\\
    =~ & \frac{1}{p-K} \sum_{\ell=1}^{p-K}
    \E \Bigl[ 
        I\{(Y,Z)\in O\}
            \Bigl\|
                e_{\ell}^T \hat A V (G^TG)^{-1/2}
                + \z_{\ell}^T H(Y, \check Z^{\ell})
            \Bigr\|^2
        \Bigr]\\
    =~& \frac{1}{p-K} \sum_{\ell=1}^{p-K}
    \E \Bigl[ 
        I\{(Y,Z)\in O\}
            \Bigl\|
                 \sum_i \pdv{ e_i^T G(G^TG)^{-1/2}}{z_{i\ell}} 
                +\z_{\ell}^T H(Y, \check Z^{\ell})
                - \rem_{\ell}
            \Bigr\|^2
        \Bigr],
    \end{align*}
    where 
    $\rem_{\ell} = \sum_{i=1}^n\bigl[-(g_i^T \otimes e_{\ell}^T)M(H_i\otimes z_i)(G^TG)^{-1/2} 
    + e_i^T G \pdv{(G^TG)^{-1/2}}{z_{i\ell}}\bigr]$ from \Cref{cor:dot-G-nomrlized}, and $M = [\sum_{i=1}^n (H_i \otimes z_i z_i^T)]^{-1}$ from \Cref{lem:dot-g}. 
    Using $(a+b)^2 \le 2a^2 + 2b^2$, the above display can be bounded by sum of two terms, denoted by $(RHS)_1$ and $(RHS)_2$. 
    
    For the first term,  
    \begin{align*}
    (RHS)_1 = \frac{2}{p-K} \sum_{\ell=1}^{p-K}
    \E \Bigl[ 
        I\{(Y,Z)\in O\}
            \Bigl\|
                \sum_i \pdv{e_i^T h(Y, Z)}{z_{i\ell}} 
                + \z_{\ell}^T H(Y, \check Z^{\ell}) 
            \Bigr\|^2
        \Bigr].
    \end{align*}
    Let 
    $F(\z_{\ell}) = H(Y, Z(I-e_{\ell}e_{\ell}^T) + \z_{\ell} e_{\ell}^T) = H(Y, Z)$, then 
    $F(\check \z) = H(Y, \check Z^{\ell})$. 
    Apply \Cref{lem:chi2} to $F(\z_{\ell})$ conditionally on $Z(I-e_{\ell}e_{\ell}^T)$, we obtain  
    \begin{align*}
    &\E \Bigl[ I\{(Y,Z)\in O\} \Big\|\sum_{i} \pdv{e_i^T h(Y, Z)}{z_{i\ell}} + \z_{\ell}^T F(\check \z)\Big\|^2\Bigr]\\
    =~ &\E \Bigl[ I\{(Y,Z)\in O\} \Big\|\z_{\ell}^T F(\z_{\ell}) - \sum_{i} \pdv{e_i^T F(\z_{\ell})}{z_{i\ell}} - \z_{\ell}^T F(\check \z) \Big\|^2\Bigr]\\
    \le~ &\E \Bigl[\Big\|\z_{\ell}^T F(\z_{\ell}) - \sum_{i} \pdv{e_i^T F(\z_{\ell})}{z_{i\ell}} - \z_{\ell}^T F(\check \z) \Big\|^2\Bigr]\\
    \le~& 3\sum_i \E \Big\|\pdv{F(\z_{\ell})}{z_{i\ell}}\Big\|_F^2\\
    =~& 3\sum_i \E \Big\|\pdv{H(Y, Z)}{z_{i\ell}}\Big\|_F^2,
    \end{align*}
    where the first equality uses $\z_{\ell}^T h(Y, Z) = 0$ from the KKT conditions $Z^T G = 0$ and $h(Y, Z) = G(G^TG)^{-1/2}$. It follows that 
    \begin{align*}
        (RHS)_1 
        \le~& \frac{6}{p-K} \E \Bigl[\sum_{\ell=1}^{p}\sum_{i=1}^n 
        \Big\|
        \pdv{H(Y, Z)}{ z_{i\ell}} 
        \Big\|_F^2\Bigr]. 
        % \le~& \frac{3KL^2}{p-K}. 
    \end{align*}
    Note that the integrand in the last display is actually the squared Frobenius norm of the Jacobian of the mapping from 
    $\R^{n\times p}$ to $\R^{n\times K}$: $Z \mapsto H(Y, Z)$. 
    This Jacobian is a matrix with $nK$ rows and $np$ columns, has rank at most $nK$ and operator norm at most $L L_2 n^{-1/2}$ because $Z\mapsto H(Y, Z)$ is $L L_2n^{-1/2}$-Lipschitz from \Cref{lem:lipschitz-G}. 
    Using $\|A\|_F^2 \le \rank(A) \|A\|_{op}^2$, we obtain 
    $$
    (RHS)_1 \le 6K(LL_2)^2/(p-K).
    $$ 
    
    For the second term 
    $(RHS)_2 = \frac{2}{p-K} \sum_{\ell=1}^{p-K} \E 
    \bigl[ I\{(Y,Z)\in O\}\| \rem_{\ell}\|^2 \bigr]$. By definition of $\rem_{\ell}$ and $(a+b)^2 \le 2a^2 + 2b^2$, we obtain 
    \begin{align}
        (RHS)_2 \le &    \frac{4}{p-K} \sum_{\ell=1}^{p-K}
        \E \Bigl[ I\{(Y,Z)\in O\}
                \Big\|
                \sum_{i=1}^n (g_i^T \otimes e_{\ell}^T)M(H_i\otimes z_i)(G^TG)^{-1/2} \Big\|^2\Bigr]\label{eq:rem1}\\
        & + \frac{4}{p-K} \sum_{\ell=1}^{p-K}
        \E \Bigl[ I\{(Y,Z)\in O\}
                \Bigl\|\sum_{i=1}^n e_i^T G \pdv{(G^TG)^{-1/2}}{z_{i\ell}} 
                \Bigr\|^2
            \Bigr]. \label{eq:rem2}
    \end{align}
    We next bound \eqref{eq:rem1} and \eqref{eq:rem2} one by one. 
    For \eqref{eq:rem1}, we focus on the norm without
    $(G^TG)^{-1/2}$ which is
    $\|\sum_{i=1}^n(g_i^T \otimes e_{\ell}^T)M(H_i\otimes z_i)\|$. 
    With $\|a\|=\max_{u:\|u\|=1} a^Tu$ in mind,
    let us multiply to the right by a unit vector $u \in \R^K$
    and instead bound
    $$
        \sum_{i=1}^n(g_i^T \otimes e_{\ell}^T)M(H_i\otimes z_i) u
        =
        \trace\Bigl[(I_K \otimes e_{\ell}^T)M\sum_i(H_i u\otimes z_i) g_i^T \Bigr]
        \le
        K \|(I_K \otimes e_{\ell}^T)M\|_{op}
        \|\sum_i(H_iu\otimes z_i) g_i^T\|_{op}
    $$
    because the rank of the matrix inside the trace is at most $K$
    and $\trace[\cdot]\le K \|\cdot\|_{op}$ holds.
    Then
    $$
    \|\sum_i(H_i u\otimes z_i) g_i^T\|_{op}
    = \|(I_K \otimes Z^T) \sum_i (H_i u\otimes e_i)e_i^T G\|_{op}
    \le \|Z\|_{op} \|\sum_{i} (H_i u\otimes e_i e_i^T)\|_{op} \|G\|_{op}.
    $$
    Next,
    $\|\sum_{i} (H_iu\otimes e_i e_i^T)\|_{op} 
    = \|\sum_{i} (H_i\otimes e_i e_i^T) (u\otimes I_n)\|_{op}
    \le 1
    $ 
    because $H_i \preceq I_K$ and $\|u\| = 1$. 
    In summary,
    the norm in \eqref{eq:rem1} is bounded from above by 
    $$
    K
    \|(G^TG)^{-1/2}\|_{op}
    \|M\|_{op}
    \|Z\|_{op}
    \|G\|_{op}.
    $$
    
    To bound \eqref{eq:rem1}, 
    since in the event $(Y, Z)\in O$, $m_* I_K \preceq \frac1n G^TG\preceq m^* I_K$ and 
    $\frac{1}{n}\sum_{i=1}^n (H_i \otimes z_i z_i^T) 
    \succeq c_* I_{pK}$, 
    we have 
    $\|G\|_{op} \le \sqrt{m^* n}$, 
    hence 
    $\|M\|_{op} \le c_*^{-1}$. 
    Thus, the above display can be bounded by 
    $$
    K (m_* n)^{-1/2} c_*^{-1} \|Z\|_{op} \sqrt{m^* n} = (m^*/m_*)^{1/2} c_*^{-1} K \|Z\|_{op}.
    $$
    Since $Z\in \R^{n\times p}$ has \iid $N(0, 1)$ entries, \cite[Theorem II.13]{DavidsonS01} implies that 
    $
    \E \|Z\|_{op} \le \sqrt{n} + \sqrt{p} \le 2 \sqrt{n}.
    $
    Therefore, 
    \begin{equation*}
        \eqref{eq:rem1} \le C(c_*, K, L) n^{-1}. 
    \end{equation*}
    
    Now we bound \eqref{eq:rem2}. 
    Since 
    \begin{align*}
        &\sum_{\ell=1}^{p-K} \Bigl\|\sum_{i=1}^n e_i^T G \pdv{(G^TG)^{-1/2}}{z_{i\ell}} \Bigr\|^2 \\
        \le~& \sum_{j=1}^{p} \Bigl\|\sum_{i=1}^n e_i^T G \pdv{(G^TG)^{-1/2}}{z_{ij}} \Bigr\|^2\\
        =~& \sum_{j=1}^{p} \sum_{k'=1}^K \Bigl(\sum_{i=1}^n \sum_{k=1}^K e_i^T G e_k e_k^T\pdv{(G^TG)^{-1/2}}{z_{ij}}e_{k'} \Bigr)^2\\
        \le~& \sum_{j=1}^{p} \sum_{k'=1}^K \Bigl[ \sum_{i=1}^n \sum_{k=1}^K (e_i^T G e_k)^2  \sum_{i=1}^n \sum_{k=1}^K\Bigl(e_k^T\pdv{(G^TG)^{-1/2}}{z_{ij}}e_{k'} \Bigr)^2 \Bigr]\\
        =~& \|G\|_F^2 \sum_{i=1}^n \sum_{j=1}^{p} 
        \Bigr\|\pdv{(G^TG)^{-1/2}}{z_{ij}} \Bigr\|^2. 
    \end{align*}
    Using $\|G\|_F^2 \le nK$, and
    the mapping $Z\mapsto (G^TG)^{-1/2}$ is $L L_1 n^{-1}$- Lipschitz on $O$ using \Cref{lem:lipschitz-G}, we conclude that 
    $$\eqref{eq:rem2} \le 4 K^3 L L_1/(p-K). 
    $$
    
    Combining the above bounds on $(RHS)_1$ and $(RHS)_2$, we have 
    \begin{equation}\label{eq:B5}
        \E \Bigl[ 
            I\{(Y, Z)\in O\}
            \|
            (G^TG)^{-1/2} V \hat A^T e_{1}
            - \xi
            \|^2
        \Bigr] \le \frac{C(c_*, K, m_*, m^*, L, L_1, L_2)}{p-K}, 
    \end{equation}
    where the constant depends on $(c_*, K, m_*, m^*, L)$ only because
    $L_1$ and $L_2$ are constants depending on $(K, m_*, m^*)$ only. 

    If additionally $\P((Y, Z)\in O) \to 1$, 
    we have $\P((Y, \check Z^1)\in O) \to 1$ using 
    $(Y, Z) \stackrel{d}{=} (Y, \check Z^1)$. 
    % $Y$ does not depend on the first $p-K$ columns of $Z$, we also have $(Y, Z) \stackrel{d}{=} (Y, \check Z^1)$. 
    Therefore, 
    \begin{equation}\label{eq:xi-normal}
        \xi 
        = -[h(Y, \check Z^1)^T \z_1] I((Y, \check Z^1) \in O) - 
        [H(Y, \check Z^1)^T \z_1] I((Y, \check Z^1) \notin O) \limd N(\bm{0}, I_K). 
    \end{equation}
    By \eqref{eq:B5}, we know 
    $(G^TG)^{-1/2} V \hat A^T e_1 - \xi \limd 0$ when $\P((Y, Z)\in O) \to 1$. 
    Hence, we conclude 
    \begin{align*}
        (G^TG)^{-1/2} V \hat A^T e_1 \limd N(\bm{0}, I_K) 
        \quad \text{ and } \quad
        \|(G^TG)^{-1/2} V \hat A^T e_1\|^2 \limd \chi^2_K.
    \end{align*}
    \end{proof}

    We next prove \Cref{lem:chi2,lem:lipschitz-G,lem:dot-g,cor:dot-G-nomrlized}.
    \begin{proof}[Proof of \Cref{lem:chi2}]\label{proof_stein_chi2}
        Let $z_0 = (z^T, \tilde z^T)^T\in \R^{2n}$, then $z_0 \sim N(\bm{0}, \sigma^2 I_{2n})$. 
        For each  $k\in[K]$, let $f^{(k)}:\R^{2n}\to \R^{2n}$
        be 
        $$f^{(k)}(z_0) =
        \begin{pmatrix}
        [F(z) - F(\tilde z)]e_k
        \\
        0_n
        \end{pmatrix},
        $$
        so that $z_0^T f^{(k)}(z_0) = z^T [F(z) - F(\tilde z)]e_k$, and $ \div f^{(k)}(z_0)= \sum_{i=1}^n \frac{\partial e_i^T F(z) e_k}{\partial z_i}$. 
        Applying the second order Stein formula \citep{bellec2021second} to $f^{(k)}$ gives,
        with $\Jac$ denoting the Jacobian,
        \begin{align*}
            &\E\Bigl[
        \Bigl(
        z^T F (z) e_k - \sigma^2 \sum_{i=1}^n \frac{\partial e_i^T F(z) e_k}{\partial z_i} 
        - z^T F(\tilde z)e_k
        \Bigr)^2
        \Bigr]\\
        =~& \E\Bigl[
            \Bigl(
                z_0^T f^{(k)}(z_0) - \sigma^2\div f^{(k)}(z_0)
            \Bigr)^2
            \Bigr]\\
        =~& \sigma^2 \E \|f^{(k)}(z_0)\|^2 + \sigma^4 \E\trace[(\Jac f^{(k)}(z_0))^2]\\
        =~& \sigma^2\E \|[F(z) - F(\tilde z)]e_k\|^2  
            + \sigma^4\E \trace\Bigl[
            \begin{pmatrix}
                \Jac[F(z)e_k] & - \Jac[F(\tilde z)e_k] \\
                0_{n\times n} & 0_{n\times n}
            \end{pmatrix}^2
            \Bigr]\\
        =~& 2 \sigma^2 \E \|[F(z) - \E F(z)]e_k\|^2 + \sigma^4 \E \trace((\Jac[F(z)e_k])^2)\\
        \le~& 3 \sigma^4 \E \|\Jac[F(z) e_k]\|_F^2,
        \end{align*} 
    where the last inequality uses the Gaussian Poincar\'e inequality, and the Cauchy-Schwarz inequality $\trace(A^2) \le \|A\|_F^2$. 
    Summing over $k\in[K]$ gives the desired inequality. 
    \end{proof}

\phantomsection % for correct jump of \pageref
\begin{proof}[Proof of \Cref{lem:lipschitz-G}]\label{proof_lipschitz_G}
    We first prove $G \mapsto G^TG$ is Lipschitz by noting 
    \begin{align*}
        &\|G^TG - \tilde G^T \tilde G\|_{op} \\
        =~&  \|(G - \tilde G)^TG +  \tilde G^T (G - \tilde G)\|_{op}\\
        \le~& \|G - \tilde G\|_{op} (\|G\|_{op} + \|\tilde G\|_{op})\\
        \le~& 2 \sqrt{m^* n} \|G - \tilde G\|_{op}. 
    \end{align*}
    Then we show $G^TG \mapsto (G^TG)^{-1}$ is Lipschitz. 
    Let $A = G^TG$ 
    and $\tilde A = \tilde G^T \tilde G$, 
     we have  
    \begin{align*}
        &\|A^{-1} - \tilde A^{-1}\|_{op} \\
        =~&  \|A^{-1} (\tilde A - A) \tilde A^{-1}\|_{op}\\
        \le~& \|A - \tilde A\|_{op} \|A^{-1}\|_{op} \|\tilde A^{-1}\|_{op}\\
        \le~& (m_* n)^{-2}\|A - \tilde A\|_{op}.
    \end{align*}
    We next prove $(G^TG)^{-1}\mapsto (G^TG)^{-1/2}$ is Lipschitz. 
    Let $S= (G^TG)^{-1}$, $S'= (\tilde G^T\tilde G)^{-1}$, and if $u$ with $\|u\|=1$ is the eigenvector of 
    $\sqrt{S} - \sqrt{\tilde S}$ with eigenvalue $d$, then 
    \begin{align*}
        u^T (S - \tilde S) u 
        &= u^T (\sqrt{S} - \sqrt{\tilde S})\sqrt{S} u
        + u^T \sqrt{\tilde S} (\sqrt{S} - \sqrt{\tilde S}) u\\
        &= du^T \sqrt{S} u + du^T\sqrt{\tilde S} u\\
        &= d u^T (\sqrt{S} + \sqrt{\tilde S}) u.
    \end{align*}
    As $d$ can be chosen as $\pm \|\sqrt{S} - \sqrt{\tilde S}\|_{op}$
    (this argument is a special case of the Hemmen-Ando inequality
    \citep{van1980inequality}),
    this implies 
    \begin{align*}
        \|\sqrt{S} - \sqrt{\tilde S}\|_{op} = \frac{|u^T (S - \tilde S) u |}{u^T (\sqrt{S} + \sqrt{\tilde S}) u}
        \le \frac{\|S - \tilde S\|_{op}}{\lambda_{\min}(\sqrt{S} + \sqrt{\tilde S})}
        \le \frac{\|S - \tilde S\|_{op}}{2/\sqrt{m^* n}}.
    \end{align*}
    Combining the above Lipschitz results, we have 
    \begin{align*}
        \|(G^TG)^{-1/2} - (\tilde G^T\tilde G)^{-1/2}\|_{op}
        \le (m^* n)^{1/2} (m_* n)^{-2} (m^* n)^{1/2} \|G - \tilde G\|_{op}
        = \frac{m^*}{m_*^2} n^{-1} \|G - \tilde G\|_{op}.
    \end{align*}
    It immediately follows that 
    \begin{align*}
        \|(G^TG)^{-1/2} - (\tilde G^T\tilde G)^{-1/2}\|_{F}
        \le \sqrt{K} \frac{m^*}{m_*^2} n^{-1} \|G - \tilde G\|_{F}.
    \end{align*}
    That is, the mapping $G\mapsto (G^TG)^{-1/2}$ is 
    $L_1n^{-1}$-Lipschitz, where $L_1 = \sqrt{K} m_*^{-2}m^*$. 

    For the second statement, the result follows by 
    \begin{align*}
        &\|G(G^TG)^{-1/2} - \tilde G(\tilde G^T\tilde G)^{-1/2}\|_{op}\\
        \le~& \|G - \tilde G\|_{op} \|(G^TG)^{-1/2}\|_{op} +
         \|\tilde G\|_{op} \|(G^TG)^{-1/2} - (\tilde G^T\tilde G)^{-1/2}\|_{op}\\
         \le~& \|G - \tilde G\|_{op} (m_* n)^{-1/2} + (m^* n)^{1/2} L_1 n^{-1} \|G - \tilde G\|_{op}\\
    \end{align*}
    Hence, 
    \begin{align*}
        &\|G(G^TG)^{-1/2} - \tilde G(\tilde G^T\tilde G)^{-1/2}\|_{F}\\
         \le~&  \sqrt{K} \bigl(m_* ^{-1/2} + (m^*)^{1/2} L_1\bigr) n^{-1/2}\|G - \tilde G\|_{F}, 
    \end{align*}
    where $L_2 = \sqrt{K} \bigl(m_* ^{-1/2} + (m^*)^{1/2} L_1\bigr)$. 
\end{proof}

\phantomsection % for correct jump of \pageref
\begin{proof}[Proof of \Cref{lem:dot-g}]\label{proof_dot_g}
    Recall the KKT conditions $\sum_{l=1}^n z_l g_l^T = \bm{0}_{p\times K}$. 
    We look for the derivative with respect to $z_{ij}$. 
    Denoting derivatives with a dot, we find by the chain rule and product rule 
    \begin{align*}
        \dot z_l 
        &= \frac{\partial z_l}{\partial z_{ij}} = I(l=i) e_j,\\
        \dot g_l 
        &= \frac{\partial g_l}{\partial z_{ij}} 
        = \frac{\partial g_l}{\partial \hat A^\top z_l} \frac{\partial \hat A^\top z_l}{\partial z_{ij}} 
        = H_l [\dot A^\top z_l + I(l=i)\hat A^\top e_j].
    \end{align*}
    Thus, differentiating the KKT conditions w.r.t. $z_{ij}$ by the product rule gives
    $$
    \sum_{l=1}^n 
    \bigl[ I(l=i) e_j g_l^T + x_l 
    \bigl(\dot A^\top z_l + I(l=i)\hat A^\top e_j\bigr)^T 
    H_l \bigr] 
    = 0.
    $$
    That is, 
    $$
    e_j g_i^T + \sum_{l=1}^n z_l z_l^T \dot A H_l + z_i e_j^T \hat A H_i = 0.
    $$
    We then move the term involving $\dot A$ to one side, and vectorize both sides, 
    $$
    g_i \otimes e_j + (H_i\hat A^Te_j \otimes z_i) = 
    - \sum_{l=1}^n (H_l \otimes z_lz_l^T)\vect(\dot A). 
    $$
    With $M = [\sum_{l=1}^n (H_l \otimes z_lz_l^T)]^{-1}$, we obtain 
    $$
    \vect(\dot A) = - M [g_i \otimes e_j + (H_i\hat A^Te_j \otimes z_i)]. 
    $$
    Hence, using $\vect(H_l \dot A^\top z_l) = \vect(z_l^T \dot A H_l) = (H_l \otimes z_l^T)\vect(\dot A)$ gives 
    \begin{align*}
        \dot g_l 
        &= (H_l \otimes z_l^T)\vect(\dot A) + I(l=i) H_l\hat A^T e_j\\
        &= - (H_l \otimes z_l^T) M [g_i \otimes e_j + (H_i\hat A^Te_j \otimes z_i)] + I(l=i) H_l\hat A^T e_j. 
    \end{align*}
    Thus, 
    \begin{align*}
        \dot g_i 
        &= - (H_i \otimes z_i^T) M [g_i \otimes e_j + (H_i\hat A^Te_j \otimes z_i)] + H_i\hat A^T e_j\\
        &= - (H_i \otimes z_i^T) M (H_i\hat A^Te_j \otimes z_i) + H_i\hat A^T e_j - (H_i \otimes z_i^T) M (g_i \otimes e_j)\\
        &= [H_i - (H_i \otimes z_i^T) M (H_i  \otimes z_i)]  \hat A^T e_j - (H_i \otimes z_i^T) M (g_i \otimes e_j)\\
        &= V_i \hat A^T e_j - (H_i \otimes z_i^T) M (g_i \otimes e_j), 
    \end{align*}
    where $V_i = [H_i - (H_i \otimes z_i^T) M (H_i  \otimes z_i)]$. 
    \end{proof}

\phantomsection % for correct jump of \pageref
\begin{proof}[Proof of \Cref{cor:dot-G-nomrlized}]
\label{proof_dot_G_normalized}
For each $i\in[n],j\in[p]$, we have by the product rule
\begin{align*}
    &\pdv{e_i^T G(G^TG)^{-1/2}}{z_{ij}} \\
    =~& \pdv{g_i^T}{z_{ij}} (G^TG)^{-1/2} + e_i^T G \pdv{(G^TG)^{-1/2}}{z_{ij}}\\
    =~& [V_i \hat A^T e_j - (H_i \otimes z_i^T) M (g_i \otimes e_j)]^T (G^TG)^{-1/2} + e_i^T G \pdv{(G^TG)^{-1/2}}{z_{ij}}\\
    =~& e_j^T\hat A V_i (G^TG)^{-1/2} 
    + 
    \Bigl[-(g_i^T \otimes e_j^T)M(H_i\otimes z_i)(G^TG)^{-1/2} 
    + e_i^T G \pdv{(G^TG)^{-1/2}}{z_{ij}}\Bigr].
\end{align*}
With $V = \sum_{i=1}^n V_i$, we further have 
\begin{align*}
    &\sum_{i=1}^n \pdv{e_i^T G(G^TG)^{-1/2}}{z_{ij}} \\
    =~& e_j^T\hat A V_i^T (G^TG)^{-1/2} 
    + 
    \sum_{i=1}^n \Bigl[-(g_i^T \otimes e_j^T)M(H_i\otimes z_i)(G^TG)^{-1/2} 
    + e_i^T G \pdv{(G^TG)^{-1/2}}{z_{ij}}\Bigr].
\end{align*}
\end{proof}

% \section{Application to the multinomial logistic MLE}
\section{Proof of \Cref{thm:normal}}
\label{sec:application_multinomial}
Recall that \Cref{thm:general_Sigma} holds for general loss function $L_i: \R^{K} \to \R$ provided that conditions (1) and (2) in \Cref{thm:general_Sigma} hold.
In this section, we consider the multinomial logistic loss function $L_i$ defined in \Cref{sec:model-K}. To be specific, 
\begin{equation}\label{eq:L_i}
    L_i(u) 
    = - \sum_{k=1}^{K+1} \y_{ik} e_k^T Qu + \log \sum_{k'=1}^{K+1} \exp(e_{k'}^T Qu), \qquad \forall u\in \R^K. 
\end{equation}
In order to apply \Cref{thm:general_Sigma}, we need to verify that, when $L_i$ in \eqref{eq:L_i} is used, the two conditions (1) and (2) in \Cref{thm:general_Sigma} hold. 
To this end, we present a few lemmas in the following two subsections, which will be useful for asserting the conditions (1) and (2) when we apply \Cref{thm:general_Sigma} to prove \Cref{thm:normal}. 

\subsection{Control of the singular values of the gradients and Hessians} 
Before stating the lemmas that assert the conditions in \Cref{thm:general_Sigma}, define
\begin{align*}
U &= 
    \bigl\{
        (Y, X) \in\R^{n\times (K+1)} \times \R^{n\times p}: \hat\B \text{ exists}, 
        \|X \hat \B (I_{K+1} - \tfrac{\bm{1}\bm{1}^T}{K+1})\|_F^2 < n \tau
    \bigr\}, \\
U_y &= \Bigl\{
        Y\in \R^{n\times (K+1)}:
            \sum_{i=1}^n I(\y_{ik}=1) \ge \gamma n \text{ for all } k\in[K+1] 
    \Bigr\}. 
\end{align*}

\begin{lemma}[deterministic result on gradient]
\label{lem:gradient>0}
Let $L_i$ be defined as in \eqref{eq:L_i}. 
{Assume that either \Cref{assu:one_hot} or \Cref{assu:one_hot_q} holds.}
If $Y\in U_y$, for any $M\in \R^{n\times K}$ such that $\|MQ^T\|_{F}^2\le n\tau$, 
    we have
    \begin{equation*} 
        m_* I_K \preceq 
        n^{-1} \sum_{i=1}^n  
        \nabla L_i (M^T e_i)
        \nabla L_i (M^T e_i)^T
        \preceq K I_K,
    \end{equation*}
        where $m_*$ is a positive constant depending on $(K,\gamma, \tau)$ only. 
\end{lemma}

\begin{proof}[Proof of \Cref{lem:gradient>0}]
Without loss of generality, let's assume that $\gamma n$ is an integer. Otherwise, we can replace it with the greatest integer less than or equal to $\gamma n$, denoted as $\lfloor \gamma n \rfloor$.

If $Y \in U_y$, there exists at least $\gamma n$ many disjoint index sets $\{S_1, \ldots, S_{\gamma n}\}$ such that the following hold for each $l\in [\gamma n]$, 
$$
(i)~ S_{l} \subset [n]; 
(ii)~ |S_l| = K+1; 
% (iii)~ \cup_{i\in S_l} \{y_i\} = [K+1].
(iii)~ \sum_{i\in S_l} \y_{ik} = 1, \quad \forall k \in [K+1]. 
$$
Since $S_l$ are disjoint and $\cup_{l=1}^{\gamma n} S_l \subset [n]$, we have
\begin{equation*}
    \sum_{l=1}^{\gamma n} \sum_{i\in S_l} \|Q M^T e_i\|^2 \le 
    \sum_{i=1}^{n} \|Q M^T e_i\|^2 = \|Q M^T\|_F^2 < n\tau. 
\end{equation*}
It follows that at most $\alpha n$ many of $l\in \{1, 2, ..., \gamma n\}$ 
s.t. $\sum_{i\in S_l} \|Q M^T e_i\|^2 > \tau/\alpha$, otherwise the previous display can not hold. 
In other words, there exists a subset $L^* \subset \{1, 2, ..., \gamma n\}$ with $|L^*| \ge (\gamma -\alpha) n$ s.t. $\sum_{i\in S_l} \|Q M^T e_i\|^2 \le \tau/\alpha$ for all $l\in L^*$. 
Define the index set $I = \cup_{l\in L^*} S_l $, then 
$|I|\ge (K+1) n(\gamma-\alpha)$, and 
$\|Q M^T e_i\|_{\infty} \le  \sqrt{\tau/\alpha}$ for all $i \in I$. 
Let us take $\alpha = \gamma/2$, then $|L^*| \ge \frac{\gamma}{2}n$ and $|I| \ge \gamma(K+1)n/2$. 
Recall that $L_i (u) = \L_i (Qu)$, we have 
$\nabla L_i (u) = Q^T \nabla \L_i (Qu)$. Thus, 
\begin{align*}
    \nabla L_i(M^T e_i)
    = Q^T \nabla \L_i (Q M^T e_i) 
    = Q^T (-\y_{i} + \p_i),
\end{align*}
where 
$\p_i\in \R^{K+1}$ and its $k$-th entry satisfying
\begin{equation}\label{eq:softmax}
    \p_{ik} 
    = \frac{\exp( e_k^T Q M^T e_i)}{\sum_{k'=1}^{K+1} \exp(e_{k'}^T Q M^T e_i)} 
\in [c, 1-c],
\end{equation}
for some constant $c\in (0, 1)$ depending on $(\tau, \alpha, K)$ only. 
Therefore, 
\begin{align*}
    &n^{-1} \sum_{i=1}^n  
        \nabla L_i (M^T e_i)
        \nabla L_i (M^T e_i)^T\\
    =~& n^{-1} Q^T \sum_{i=1}^n  (\y_i - \p_i) (\y_i - \p_i)^T
    Q\\
    \succeq~& 
    n^{-1} Q^T \sum_{l=1}^{\gamma n} \sum_{i\in S_l}  (\y_i - \p_i) (\y_i - \p_i)^T
    Q\\
    \succeq~& 
    n^{-1} Q^T \sum_{l\in L^*} \sum_{i\in S_l}  (\y_i - \p_i) (\y_i - \p_i)^T
    Q\\
    :=~& n^{-1} Q^T \sum_{l\in L^*} A_l^T A_l
    Q,
\end{align*}
where $A_l \in\R^{(K+1)\times (K+1)}$ has $K+1$ rows
$\{\y_i - \p_i: i\in S_l\}$. 
We further note that $A_l$ is of the form 
$ (I_{K+1} - \sfP_l)$
up to a rearrangement of the columns,
where $\sfP_l \in \R^{(K+1) \times (K+1)}$ is a stochastic matrix with entries of the form
$$
\frac{\exp( e_k^T QM^T e_i)}{\sum_{k'=1}^{K+1} \exp(e_{k'}^T QM^T e_i)}, \qquad i \in S_l,\quad k\in [K+1]. 
$$
By \eqref{eq:softmax}, for each $l\in L^*$, the stochastic matrix $\sfP_l$ is irreducible and aperiodic and $\ker(I_{K+1} - \sfP_l)$ is the span of the all-ones vector $\bm{1}$.
Therefore, 
$$
(I_{K+1} - \sfP_l) QQ^T 
= (I_{K+1} - \sfP_l) (I_{K+1} - \tfrac{\bm{1}\bm{1}^T}{K+1}) 
= (I_{K+1} - \sfP_l).
$$
It follows that 
$$
K = \rank(I_{K+1} - \sfP_l) 
= \rank((I_{K+1} - \sfP_l) Q Q^T)
\le \rank((I_{K+1} - \sfP_l)Q) \le K.
$$
We conclude that the rank of  
$(I_{K+1} - \sfP_l)Q$ is $K$.

If $\mathcal P$ denotes the set of matrices
$\{
\sfP\in\R^{(K+1) \times (K+1)}: \text{ stochastic with entries in } [c, 1-c] \}$, 
and $S^{K-1} = \{a\in \R^K: \|a\|=1\}$. 
By compactness of $\mathcal P$ and $S^{K-1}$, we obtain  
\begin{align*}
    &\frac 1n \lambda_{\min} (\sum_{i=1}^n\nabla L_i (M^T e_i)
    \nabla L_i (M^T e_i)^T) \\
    ~\ge&  \frac 1n \sum_{l\in L^*} \lambda_{\min}(Q^T A_l^T A_l Q)\\
    ~\ge&  \frac 1n \sum_{l\in L^*} \min_{a\in S^{K-1}} a^T Q^T A_l^T A_l Q a\\
    ~\ge&  \frac 1n |L^*|
    \min_{a\in S^{K-1}, \sfP\in \mathcal P} a^T Q^T (I_{K+1} - \sfP)^T (I_{K+1} - \sfP) Q a\\
    ~\ge & \frac{\gamma}{2} a_*^T Q^T (I_{K+1} - \sfP_*)^T (I_{K+1} - \sfP_*) Q a_*\\
    ~\ge & \frac{\gamma}{2}  a_*^T Q^T (I_{K+1} - \sfP_*)^T (I_{K+1} - \sfP_*) Q a_*\\
    ~:= & m_*,
\end{align*}
where
$a_*\in S^{K-1}$, $\sfP_* \in \mathcal P$, and 
$m_*$ is a positive constant depending on $(K, \gamma, \tau)$ only. 
The first inequality above uses the property $\lambda_{\min}(A + B) \ge \lambda_{\min}(A) + \lambda_{\min}(B)$, where $A$ and $B$ are two positive semi-definite matrices.

In other words, $\frac 1n \sum_{i=1}^n \nabla L_i (M^T e_i)
\nabla L_i (M^T e_i)^T\succeq m_* I_{K}$. 
For the upper bound, since $\|Q\|_{op}\le 1$ by definition of $Q$ and all the entries of $(\y_i - \p_i) (\y_i - \p_i)^T$ are between $-1$ and $1$ {if \Cref{assu:one_hot} or \Cref{assu:one_hot_q} holds}, we have 
$$
\|\sum_{i=1}^n \nabla L_i (M^T e_i)
\nabla L_i (M^T e_i)^T\|_{op} 
= \| Q^T \sum_{i=1}^n  (\y_i - \p_i) (\y_i - \p_i)^T
Q\|_{op} 
\le nK.$$ 
\end{proof}

\begin{lemma}[deterministic result on Hessian]\label{lem:Hessian>0}
Let $L_i$ be defined as in \eqref{eq:L_i}. 
For all $i\in [n]$, we have 
$\nabla^2 L_i(u) \preceq I_K$ for any $u\in\R^{K}$ and 
\begin{equation*}
    \min_{u \in \R^K, \|Qu\|_{\infty}\le r}\nabla^2 L_i (u) 
    \succeq c_* I_K, 
\end{equation*}
where $c_*$ is a positive constant depending on $(K, r)$ only.
\end{lemma}

\begin{proof}[Proof of \Cref{lem:Hessian>0}] 
Recall that $L_i (u) = \L_i (Qu)$, we have 
\begin{align*}
    \nabla^2 L_i (u) = Q^T \nabla^2 \L_i (Qu) Q, 
\end{align*}
where 
$\nabla^2 \L_i (Qu) 
= \diag(\p_i) - \p_i \p_i^T$
and the $k$-th entry of $\p_i\in\R^{K+1}$ is defined as 
$\p_{ik} = \frac{\exp(e_k^T Q u)}{\sum_{k'=1}^{K+1}\exp(e_{k'}^T Q u)} $ for all $i\in[n], k\in[K+1]$. 
Thus, $\p_{ik} \le 1$ for all $i\in[n], k\in[K+1]$. 
For any vector $a\in S^{K-1}$, we have
\begin{align*}
    a^T \nabla^2 L_i (u)  a
    =&a^T Q^T [\diag(\p_i) - \p_i \p_i^T] Q a \\
    \le & a^T Q^T  \diag(\p_i) Qa \\
    \le & \|Q^T \diag(\p_i) Q\|_{op} \\
    \le & 1,
\end{align*}
where the last inequality uses $\|Q\|_{op} \le 1$ and $\p_{ik}\le 1$ for any $k$. 
Hence, $\nabla^2 L_i (u) \preceq I_K$ for any $i\in [n]$.

Now we prove the lower bound. 
For any $u\in \R^{K}$ such that $\|Qu\|_{\infty}\le r$, we have 
$$\p_{ik}=\frac{\exp(e_k^T Q u)}{\sum_{k'=1}^{K+1}\exp(e_{k'}^T Q u)}\in [c, 1-c]$$ 
for some constant $c$ depending on $(K, r)$ only. 
For any vector $a \in S^{K-1}$, let $\eta = Q a \in \R^{K+1}$, then $\bm{1}^T \eta =0$ and 
    \begin{align*}
        a^T \bigl[\nabla^2 L_i (u)\bigr] a
        =&a^T Q^T [\diag(\p_i) - \p_i \p_i^T] Q a \\
        =& \eta^T [\diag(\p_i) - \p_i \p_i^T] \eta \\
        =& \sum_{k=1}^{K+1} \p_{ik} \eta_k^2 - (\sum_{k=1}^{K+1} \p_{ik} \eta_k)^2\\
        >& \sum_k \p_{ik} \eta_k^2 - \sum_k \p_{ik}\eta_k^2 \sum_k \p_{ik}\\
        =& \sum_k \p_{ik} \eta_k^2 (1 - \sum_k \p_{ik})\\
        =& 0,
    \end{align*}
    where the last equality uses $\sum_{k=1}^{K+1} \p_{ik} =1$, 
    and the inequality follows by 
    $(\sum_k \p_{ik} \eta_k)^2 = (\sum_k \sqrt{\p_{ik}} \sqrt{\p_{ik}} \eta_k)^2 \le \sum_k \p_{ik} \sum_k \p_{ik}\eta_k^2$ using the Cauchy-Schwarz inequality, 
    and here $``="$ holds if and only if $\sqrt{\p_{ik}} \propto \sqrt{\p_{ik}} \eta_k$ for each $k$, 
    which is not true since $\p_{ik} \in [c, 1-c]$ and $\bm{1}^T \eta =0$.

    Let $\mathcal H = \{Q^T (\diag(\p) - \p \p^T) Q:  \p\in [c, 1-c]^{K+1}\}$, then $\mathcal H$ is compact and 
    \begin{align*}
        \min_{u \in \R^K, \|Qu\|_{\infty}\le r} 
        \lambda_{\min} (\nabla^2 L_i (u)) 
        \ge  
        \min_{a\in S^{K-1}, H\in \mathcal H} a^T H a
        = a_*^T H_* a_*>0
    \end{align*}
    for some $a_*\in S^{K-1}$ and $H_*\in \mathcal H$. 
    Therefore, 
    $$
    \min_{u \in \R^K, \|Qu\|_{\infty}\le r} 
   \nabla^2 L_i (u))
    \succeq c_* I_K
    $$
    where $c_*$ is a positive constant depending on $(K, r)$ only. 
\end{proof}

\subsection{Lipschitz conditions}
We first restate the definitions of following sets, 
\begin{align*}
    U &= 
        \bigl\{
            (Y, X) \in\R^{n\times (K+1)} \times \R^{n\times p}: \hat\B \text{ exists}, 
            \|X \hat \B (I_{K+1} - \tfrac{\bm{1}\bm{1}^T}{K+1})\|_F^2 < n \tau
        \bigr\}, \\
    U_y &= \Bigl\{
        Y\in \R^{n\times (K+1)}:
            \sum_{i=1}^n I(\y_{ik}=1) \ge \gamma n \text{ for all } k\in[K+1] 
    \Bigr\}
    \end{align*}

\begin{lemma}\label{lem:XtoG-Lipschitz}
Assume $p/n \le \delta^{-1} < 1-\alpha$ for some $\alpha \in (0, 1)$ and $\delta >1$. 
Let $\mathcal I = \{I\subset [n]: |I| = \lceil n (1 - \alpha)\rceil \}$ and $P_I = \sum_{i\in I} e_i e_i^T$. 
Define 
\begin{align}
    \label{U_x}
    U_x = \bigl\{X\in \R^{n\times p}: 
        \min_{I\in \mathcal I} \lambda_{\min}(\tfrac{\Sigma^{-1/2}X^T P_I X \Sigma^{-1/2}}{{n}}) \ge \phi_*^{2}, 
        \tfrac{\|X \Sigma^{-1/2}\|_{op}}{\sqrt{n}} \le \phi^*
        %  \min_{\|Qu\|_{\infty}\le \sqrt{\tfrac{2\tau}{\alpha}}} \lambda_{\min} [\nabla^2 L_i(u)] \ge c_*.
    \bigr\}
\end{align}

% $$
% \bar U = 
% \bigl\{(Y, X)\in\R^{n\times K} \times \R^{n\times p}: 

%     \min_{I\in \mathcal I} \phi_*(I) > \bar \phi_*, 
%     c_* >0, \frac{\|X \Sigma^{-1/2}\|_{op}}{\sqrt{n}} \le \phi^*, 
%     \bar m_* < \min_{I\in \mathcal I} m_*(I)
% \bigr\}
% $$
for some positive constants $\phi_*, \phi^*$, which depend on $(\delta, \alpha)$ only. 
Let $U^* = \{(Y, X)\in U: Y\in U_y, X\in U_x\}$. 
Then under \Cref{assu:X,assu:Y,assu:MLE},
{and if either \Cref{assu:one_hot} or \Cref{assu:one_hot_q} holds, }
we have
\begin{enumerate}
    \item $\P((Y, X) \in U^*) \to 1$ as $n, p\to \infty$. 
    \item  
    Let $G$ be defined in \Cref{assum:general_loss}. 
    If $\{(Y, X),(Y, \tilde X)\}  \subset U^*$,
    we have 
    $$
    \|G(Y, X) - G(Y, \tilde X)\|_F 
    \le L \|(X - \tilde X)\Sigma^{-1/2}\|_F, 
    $$
    where $L$ is a positive constant depending on $(K, \gamma, \tau,\alpha)$ only. 
\end{enumerate}

\end{lemma}

\begin{proof}[Proof of \Cref{lem:XtoG-Lipschitz}]
We first prove statement (i). 
Under \Cref{assu:X}, \cite[Lemma 7.7]{bellec2022observable} implies 
$$
\P\bigl(\min_{I\in \calI}  
\lambda_{\min}(\tfrac{\Sigma^{-1/2}X^T P_I X \Sigma^{-1/2}}{{n}})
\ge \phi_*^{2}\bigr) 
\to 1
$$
for some positive constant $\phi_*$ depending on $(\delta, \alpha)$ only. 
Furthermore, \cite[Theorem II.13]{DavidsonS01} implies 
$$\P(\tfrac{\|X \Sigma^{-1/2}\|_{op}}{\sqrt{n}} \le \phi^*) \to 1$$
for some positive constant $\phi^*$ depending on $\delta$ only. 
Therefore, 
$\P(X\in U_x) \to 1$. 
Under \Cref{assu:Y}, we have 
$\P(Y\in U_y)\to 1$. 
Under \Cref{assu:MLE}, we have 
$\P((Y, X) \in U) \to 1$. 
In conclusion, under \Cref{assu:X,assu:Y,assu:MLE}, we have 
$\P((Y, X) \in U^*) \to 1$ as $n, p\to \infty$. 
    
Now we prove the statement (ii). 
For a fixed $Y$, let $(Y, X), (Y, \tilde X) \in U^*$,  
$\hat B, \tilde B$ be their corresponding minimizers of \eqref{eq:mle}, 
and $G, \tilde G$ be their corresponding gradient matrices. 
We first provide some useful results derived from the KKT conditions. 
From the KKT conditions $X^T G = \tilde X^T \tilde G = 0$, we have
\begin{align*}
    \langle X\hat B - \tilde X \tilde B, G - \tilde G \rangle 
    = & \langle \hat B - \tilde B, \tilde X^T \tilde G - X^T G \rangle
    +
    \langle X\hat B - \tilde X \tilde B, G - \tilde G \rangle  \\
    = & - \langle (X - \tilde X)(\hat B - \tilde B), G\rangle + \langle (X - \tilde X) \hat B, G - \tilde G\rangle. 
\end{align*}
Since $\|\nabla^2 L_i(u)\|_{op} \le 1$ for any $u\in\R^{K}$ from \Cref{lem:Hessian>0}, $\nabla L_i(\cdot)$ is 1-Lipschitz. Thus,  
\begin{align*}
    \langle X\hat B - \tilde X \tilde B, G - \tilde G\rangle
    =& \sum_{i=1}^n \langle \hat B^T x_i - \tilde B^T \tilde x_i, \nabla L_i(\hat B^T x_i) - \nabla L_i(\tilde B^T \tilde x_i)\rangle \\
    \ge& \sum_{i=1}^n \langle \nabla L_i(\hat B^T x_i) - \nabla L_i(\tilde B^T \tilde x_i), \nabla L_i(\hat B^T x_i) - \nabla L_i(\tilde B^T \tilde x_i)\rangle\\
    =& \|G - \tilde G\|_F^2. 
\end{align*}
If $(Y, X), (Y, \tilde X) \in U$, we have 
$\|X\hat B Q^T \|_F^2 + \|\tilde X\tilde B Q^T \|_F^2\le 2n \tau$. 
That is, 
$$
\sum_{i=1}^n 
\bigl(
    \|Q \hat B^T x_i\|^2 + \|Q \tilde B^T \tilde x_i\|^2
\bigr)
\le 2n\tau. 
$$
Define the index set 
$$
I = \{i\in [n]: \|Q \hat B^T x_i\|^2 + \|Q \tilde B^T \tilde x_i\|^2 \le \tfrac{2\tau}{\alpha}\},
$$
then we have $|I| \ge (1-\alpha)n $ by Markov's inequality. 
Thus, for all $i\in I$, we have 
$\|Q \hat B^T x_i\|_{\infty} \vee \|Q \tilde B^T \tilde x_i\|_{\infty} \le \sqrt{\tfrac{2\tau}{\alpha}}$. 

Applying \Cref{lem:Hessian>0} with $r = \sqrt{\tfrac{2\tau}{\alpha}}$ gives 
$$
\min_{\|Q u\|_{\infty} \le \sqrt{\tfrac{2\tau}{\alpha}}} 
\nabla^2 L_i(u) \succeq c_* I_K, 
$$
where $c_*$ is a constant depending on $(K, \tau, \alpha)$. 
Therefore, 
\begin{align*} 
    c_*  \|P_I (X\hat B - \tilde X \tilde B)\|_F^2 =
     & c_* \sum_{i\in I} \|\hat B^T x_i - \tilde B^T \tilde x_i\|^2 \\
    \le & \sum_{i\in I}
        \langle \hat B^T x_i - \tilde B^T \tilde x_i , 
        \nabla L_i(\hat B^T x_i) - \nabla L_i(\tilde B^T \tilde x_i) 
        \rangle \\
    \le & \sum_{i=1}^n
    \langle \hat B^T x_i - \tilde B^T \tilde x_i , 
    \nabla L_i(\hat B^T x_i) - \nabla L_i(\tilde B^T \tilde x_i) 
    \rangle \\
    = & \langle X\hat B - \tilde X \tilde B, G - \tilde G \rangle \\
    = & - \langle (X - \tilde X)(\hat B - \tilde B), G\rangle + \langle (X - \tilde X) \hat B, G - \tilde G\rangle. 
\end{align*}
We next bound the first line from below by 
expanding the squares,
\begin{align*}
    \|P_I (X\hat B - \tilde X \tilde B)\|_F^2 
    &= \|P_I \tilde X (\hat B - \tilde B) + P_I (X - \tilde X) \hat B\|_F^2 \\
    &\ge \|P_I \tilde X (\hat B - \tilde B)\|_F^2 +
    2 \langle P_I \tilde X (\hat B - \tilde B), P_I (X - \tilde X) \hat B \rangle  \\
    &\ge n \phi_*^2 \|\Sigma^{1/2} (\hat B - \tilde B)\|_F^2 + 2 \langle \tilde X (\hat B - \tilde B), P_I (X - \tilde X) \hat B \rangle,
\end{align*}
where in the last inequality we use the constant $\phi^*$ in \eqref{U_x}.
Therefore, we obtain 
\begin{align*}
    &c_*\phi_*^2 n \|\Sigma^{1/2} (\hat B - \tilde B)\|_F^2\\
    \le & - \langle (X - \tilde X)(\hat B - \tilde B), G\rangle + \langle (X - \tilde X) \hat B, G - \tilde G\rangle
    - 2 c_*\langle \tilde X (\hat B - \tilde B), P_I (X - \tilde X) \hat B \rangle. 
\end{align*}
Together with the inequality that 
$\|G - \tilde G\|_F^2 
\le 
\langle X\hat B - \tilde X \tilde B, G - \tilde G \rangle, 
$
we obtain 
\begin{align*}
    &c_* \phi_*^2 n \|\Sigma^{1/2} (\hat B - \tilde B)\|_F^2 + \|G - \tilde G\|_F^2 \\
    \le & - 2\langle (X - \tilde X)(\hat B - \tilde B), G\rangle + 2 \langle (X - \tilde X) \hat B, G - \tilde G\rangle
    - 2 c_*\langle \tilde X (\hat B - \tilde B), P_I (X - \tilde X) \hat B \rangle \\
    \le & (4 + 2c_*\phi^* ) \|(X - \tilde X)\Sigma^{-1/2}\|_{op}
    \bigl(\|\Sigma^{1/2}(\hat B - \tilde B)\|_F \vee \frac{\|G - \tilde G\|_F}{\sqrt{n}}\bigr) \bigl(\|\Sigma^{1/2} \hat B\|_F \vee \frac{\|G\|_{op}}{\sqrt{n}}\bigr) \sqrt{n},
\end{align*}
where we bound $\langle \tilde X (\hat B - \tilde B), 
P_I (X - \tilde X) \hat B \rangle$ by definition of $\phi^* $, 
\begin{align*}
    &   \langle \tilde X (\hat B - \tilde B), 
        P_I (X - \tilde X) \hat B \rangle \\
    =&  \langle \Sigma^{1/2} (\hat B - \tilde B), 
        \Sigma^{-1/2}\tilde X^T P_I (X - \tilde X) \hat B \rangle \\
    \le& \|\Sigma^{1/2} (\hat B - \tilde B)\|_F 
        \|P_I \tilde X \Sigma^{-1/2}\|_{op}
        \|(X - \tilde X) \Sigma^{-1/2}\|_{op} \|\Sigma^{1/2}\hat B\|_F\\
    \le& \sqrt{n} \phi^*  \|\Sigma^{1/2} (\hat B - \tilde B)\|_F 
    \|(X - \tilde X) \Sigma^{-1/2}\|_{op} \|\Sigma^{1/2}\hat B\|_F. 
\end{align*}
Now we derive a bound of the form $\|\Sigma^{1/2}\hat B\|_F  \lesssim \|G\|_F/\sqrt{n}$. 
To this end, since 
$\phi_*\|\Sigma^{1/2}\hat B\|_F 
\le \|P_I X\hat B\|_F/\sqrt{n} 
\le \|X \hat B\|_F/\sqrt{n} 
= \|X\hat B Q^T\|_F/\sqrt{n} 
\le \sqrt{\tau}$. 

Applying \Cref{lem:gradient>0} to $M = X\hat B$, we have $\frac{1}{n} \sum_{i=1}^n g_i g_i^T \succeq m_* I_K$. 
Therefore, 
\begin{align*}
    \frac 1n\|G\|_F^2 
    = \frac 1n\sum_{i=1}^n \|g_i\|^2
    = \frac 1n \sum_{i=1}^n \trace(g_i g_i^T)
    \ge K m_*. 
\end{align*}
This implies that 
\begin{align*}
    \phi_*^2 \|\Sigma^{1/2} \hat B\|_F^2 \le \tau \le \frac{\tau}{K m_*(I)} \|G\|_F^2/n. 
\end{align*}
In conclusion, if $\{(Y,X), (Y,\tilde X)\} \subset U^*$ then 
\begin{align*}
    \sqrt{n}\|\Sigma^{1/2} (\hat B - \tilde B)\|_F + \|G - \tilde G\|_F \le Cn^{-1/2} \|(X - \tilde X) \Sigma^{-1/2}\|_{op} \|G\|_{F}
    \le C K \|(X - \tilde X) \Sigma^{-1/2}\|_{op}, 
\end{align*}
where $C$ is a constant depending on $(K, \gamma, \tau, \alpha)$ only. 
Note that $\|G\|_F \le \sqrt{nK}$ since all entries of $G$ are in $[-1,1]$. 
\end{proof}

\begin{lemma}\label{lem:HX}
    If $p/n\le \delta^{-1} < (1-\alpha)$ for some $\alpha \in (0, 1)$ and $\delta >1$. 
    If $(Y, X)\in U$ and $X\in U_x$, where $U_x$ is defined in \Cref{lem:XtoG-Lipschitz}, 
    we have 
    $$\frac 1n \sum_{i=1}^n H_i \otimes (x_i x_i^T) \succeq c_1 
    (I_K \otimes \Sigma), $$
    where $c_1$ is a positive constant depending on 
    $(K,\tau, \alpha, \phi_*)$ only. 
    \end{lemma}
    \begin{proof}[Proof of \Cref{lem:HX}]
    % By definitions of $\hat B$ and $Q$, we have $\hat B Q^T = \hat \B (I_{K+1} - \tfrac{\bm{1}\bm{1}^T}{K+1})$. 
    If $(Y, X) \in U$, 
    we have $\|X\hat B Q^T\|_F^2 \le n\tau$. 
    Define the index set 
    $$
    I = \{i\in [n]: \|Q \hat B^T x_i\|\le \tfrac{\tau}{\alpha}\},
    $$
    then we have $|I| \ge (1-\alpha)n$ by Markov's inequality. 
    Therefore, for any $i\in I$, 
    $\|Q \hat B^T x_i\|_{\infty} 
    \le \tfrac{\tau}{\alpha}$.
    Applying \Cref{lem:Hessian>0} with $u= \hat B^T x_i$ and $r=\tfrac{\tau}{\alpha}$, we have for any $i\in I$, 
    $
    H_i = \nabla^2 L_i(\hat B^T x_i) 
    \succeq c_* I_K
    $
    for some positive constant $c_*$ depending on $(K, \tau, \alpha)$ only. 
    Therefore, if $(Y, X)\in U$ and $X\in U_x$, we have 
    \begin{align*}
        \frac 1n \sum_{i=1}^n H_i \otimes (x_i x_i^T) 
        \succeq~ & \frac 1n\sum_{i\in I} H_i \otimes (x_i x_i^T)\\
        \succeq~ & c_* \frac 1n\sum_{i\in I} I_K \otimes (x_i x_i^T)\\
        =~ & c_* (I_K\otimes \tfrac{X^T P_I X}{n})\\
        \succeq~ & 
        c_* \phi_* (I_K\otimes \Sigma)\\
        =~& c_1 (I_K\otimes \Sigma),
    \end{align*}
    where $P_I = \sum_{i\in I} e_i e_i^T$ and 
    $c_1$ is a positive constant depending on 
    $(K,\tau, \alpha, \phi_*)$ only. 
    \end{proof}

% \section{Proof of main results}\label{sec:main-proofs}
\subsection{Proof of \Cref{thm:normal}}
The proof of \Cref{thm:normal} is a direct consequence of 
\Cref{thm:general_Sigma} by noting 
$$ \sqrt{n}\Omega_{jj}^{-1/2} \Bigl(\frac 1n\sum_{i=1}^n g_i g_i^T\Bigr)^{-1/2} \Bigl(\frac1n\sum_{i=1}^n V_i\Bigr) \hat B^T e_j 
= \Omega_{jj}^{-1/2} (G^TG)^{-1/2}V\hat B^Te_j,$$
which is a consequence of the identities 
$G=\sum_{i=1}^n e_i g_i^T$ and $V = \sum_{i=1}^n V_i$. 

It thus remains to verify the conditions (1) and (2) in \Cref{thm:general_Sigma} from the assumptions in \Cref{thm:normal}. 

Applying \Cref{lem:XtoG-Lipschitz} with $\alpha$ chosen as $1- \delta^{-1/2}$, we have for $\{(Y, X), (Y, \tilde X)\} \subset U^*$, 
$$
\|G(Y, X) - G(Y, \tilde X)\|_F 
\le L \|(X - \tilde X)\Sigma^{-1/2}\|_F, 
$$
where $L$ is a positive constant depending on $(K, \gamma, \tau,\delta)$ only. 

Apply \Cref{lem:HX} with the same $\alpha = 1- \delta^{-1/2}$, we have for $(Y, X) \in U^*$,
$$\frac 1n \sum_{i=1}^n H_i \otimes (x_i x_i^T) \succeq c_* 
(I_K \otimes \Sigma), $$
where $c_*$ is a positive constant depending on $(K, \tau, \delta)$ only. 

Applying \Cref{lem:gradient>0} with $M = X\hat B$, we have for $(Y, X) \in U^*$, 
\begin{equation*} 
    m_* I_K \preceq 
    n^{-1} \sum_{i=1}^n  
    \nabla L_i (M^T e_i)
    \nabla L_i (M^T e_i)^T
    \preceq K I_K,
\end{equation*}
    where $m_*$ is a positive constant depending on $(K,\gamma, \tau)$ only. 
Therefore, the conditions (1) and (2) in \Cref{thm:general_Sigma} hold when the multinomial logistic loss is used. 
This completes the proof of \Cref{thm:normal}. 

\section{Other proof}
\subsection{Proof of \Cref{eq:classical-A} (Classical asymptotic theory with fixed $p$)}
Here we provide a derivation of the  asymptotic distribution of MLE under classical setting, where $p$ is fixed and $n$ tends to infinity. 

We first calculate the Fisher information matrix of the multinomial logistic log-odds model \eqref{eq:log-odds} with covariate $x\sim N(\bm{0}, \Sigma)$ and response $\y\in \{0, 1\}^{K+1}$ one-hot encoded satisfying $\sum_{k=1}^{K+1} \y_k = 1$. 
Note that the model \eqref{eq:log-odds} can be rewritten as 
\begin{align*}
    \P(\y_k=1|x) &= \frac{\exp(x^TA^*e_k)}{1 + \sum_{k'=1}^{K} \exp(x^TA^*e_{k'})}, \quad \forall k\in\{1, ..., K\}\\
    \P(\y_{K+1}=1|x) &= \frac{1}{1 + \sum_{k'=1}^{K} \exp(x^TA^*e_{k'})}. 
\end{align*}
The likelihood function of a parameter $A\in \R^{p\times K}$ is 
\begin{align*}
    L(A) 
    &= \prod_{k=1}^K \Bigl[\frac{\exp(x^TAe_k)}{1 + \sum_{k'=1}^{K} \exp(x^TAe_{k'})} \Bigr]^{\y_k} 
    \Bigl[\frac{1}{1 + \sum_{k'=1}^{K} \exp(x^TAe_{k'})}\Bigr]^{\y_{K+1}} \\
    &= \prod_{k=1}^K [\exp(x^TAe_{k})]^{\y_k} \frac{1}{1 + \sum_{k'=1}^{K} \exp(x^TAe_{k'} )},
\end{align*}
where we used 
$\sum_{k=1}^{K+1}\y_k = 1$. 
Thus, the log-likelihood function is 
\begin{align*}
    \ell(A) 
    &= \sum_{k=1}^K \y_k x^TAe_{k} - \log\bigl[1 + \sum_{k'=1}^{K} \exp(x^TAe_{k'})\bigr]. 
\end{align*}
It is more convenient to calculate the Fisher information matrix on the vector space $\R^{pK}$ instead of the matrix space $\R^{p\times K}$. 
To this end, let $\theta = \vect(A^T)$, then 
$x^T A e_k = e_k^T A^T x = (x^T \otimes e_k^T)\vect(A^T) = (x^T \otimes e_k^T)\theta$, and the log-likelihood function parameterized by $\theta$ is 
$$\ell(\theta) = \sum_{k=1}^K y_k (x^T \otimes e_k^T) \theta - \log[1 + \sum_{k'=1}^{K} \exp((x^T \otimes e_{k'}^T) \theta)]. $$
By multivariate calculus, we obtain the Fisher information matrix evaluated at $\theta^* = \vect(A^*{}^T)$, 
\begin{align*}
    \calI(\theta^*) 
    &= -\E \Bigl[\frac{\partial }{\partial \theta} \frac{\partial \ell(\theta)}{\partial \theta^T}\Bigr]\Big|_{\theta = \theta^*} \\
    &= \E [(xx^T) \otimes (\diag(\pi^*) - \pi^* \pi^*{}^T)],
\end{align*}
where $\pi^*\in \R^{K}$ with $k$-th entry
$\pi_k^* = \frac{\exp(x^TA^*e_k)}{1 + \sum_{k'=1}^{K} \exp(x^TA^*e_{k'})}$ for each $k\in[K]$. 

From classical maximum likelihood theory, for instance \cite[Chapter 5]{van1998asymptotic}, we have 
\begin{align*}
    \sqrt{n} (\hat \theta - \theta^*) \limd N(\bm {0}, \calI_{\theta^*}^{-1}),
\end{align*}
where $\hat \theta = \vect(\hat A^T)$ and $\hat A$ is the MLE of $A^*$. 
Furthermore, if the $j$-th covariate is independent of the response, we know $e_j^T A^*= \bm{0}^T$, then 
\begin{align*}
    \sqrt{n}\hat A^Te_j 
    = \sqrt{n} (\hat A^Te_j - A^*{}^Te_j)
    = \sqrt{n} (e_j^T \otimes I_K) (\hat\theta - \theta^*)
    \limd N(\bm{0}, S_j),
\end{align*}
where 
\begin{align*}
S_j 
&= (e_j^T \otimes I_K) \calI_{\theta^*}^{-1} (e_j\otimes I_K) \\
&= e_j^T \cov(x)^{-1} e_j [\E (\diag(\pi^*) - \pi^* \pi^*{}^T)]^{-1}
\end{align*}
holds by the independence between the $j$-th covariate and the response under $H_0$ in \eqref{H0}. 
This completes the proof of \Cref{eq:classical-A}. 

\end{document}